\documentclass{article}

\PassOptionsToPackage{numbers, compress}{natbib}

\RequirePackage{jmlrutils}

\usepackage[utf8]{inputenc}
\usepackage[T1]{fontenc}
\usepackage{hyperref}
\usepackage{url}
\usepackage{booktabs}
\usepackage{amsfonts}
\usepackage{nicefrac}
\usepackage{microtype}
\usepackage{enumitem}
\usepackage{xcolor}
\usepackage{etoolbox}
\usepackage{setspace}
\usepackage{mathtools}

\newcommand{\reals}{\mathbb{R}}

\newcommand{\argmin}{\mathop{\rm argmin}}

\usepackage{relsize}
\usepackage{amsmath, amssymb, color, float, natbib, graphicx, url}
\usepackage[ruled]{algorithm2e}
\usepackage{thmtools}
\usepackage{enumitem}
\usepackage{cleveref}

\SetNlSty{textrm}{}{}
\DeclareTextFontCommand{\texttt}{\ttfamily\upshape}

\newtheorem{observation}{Observation}

\newtheorem{lem}{Lemma}
\newtheorem{coro}{Corollary}
\newtheorem{rem}{Remark}
\newtheorem{defn}{Definition}

\newtheorem{fact}{Fact}
\newtheorem{sublemma}{Lemma}[lem]

\newcommand{\R}{\reals}
\newcommand{\Z}{\mathbb{Z}}
\newcommand{\ep}{\epsilon}
\newcommand{\su}[2]{\mathlarger{\sum\limits_{#1}^{#2}}}

\newcommand{\f}[2]{\dfrac{#1}{#2}}
\newcommand{\ff}[2]{\tfrac{#1}{#2}}

\newcommand{\lt}{\left(}
\newcommand{\rt}{\right)}

\newcommand{\A}{\alpha}
\newcommand{\B}{\beta}
\newcommand{\M}{\lambda}
\newcommand{\X}{\mathcal{X}}
\newcommand{\G}{\nabla}
\newcommand{\grad}{\nabla}
\newcommand{\I}[2]{\mathlarger{\int}\limits_{#1}^{#2}}
\renewcommand{\k}{\kappa}
\newcommand{\ra}{\rightarrow}

\newcommand{\floor}[1]{\left\lfloor#1\right\rfloor}
\newcommand{\ceil}[1]{\left\lceil#1\right\rceil}
\newcommand{\w}{\omega}
\newcommand{\Ll}{\hat{L}}
\newcommand{\epp}{\tilde{\ep}}
\providecommand{\norm}[1]{\left\lVert#1\right\rVert}

\newcommand{\citen}[1]{[\citenum{#1}]}
\newcommand{\citex}[1]{\citen{#1}}

\newcommand{\defeq}{\triangleq}

\newcommand{\qx}{x^*}

\makeatletter
\newcommand*{\centerfloat}{%
  \parindent \z@
  \leftskip \z@ \@plus 1fil \@minus \textwidth
  \rightskip\leftskip
  \parfillskip \z@skip}
\makeatother

\newcommand{\xStar}{x^{*}}

\newcommand{\1}{\gamma}
\newcommand{\2}{\mu}
\newcommand{\ind}[1]{^{(#1)}}

\title{Near-Optimal Methods for Minimizing~ \\ 
Star-Convex Functions and Beyond}

\author{%
  Oliver Hinder ~~~ Aaron Sidford ~~~ Nimit S. Sohoni \\
  \small{\emph{Stanford University}} \\
  \small \texttt{\{ohinder, sidford, nims\}@stanford.edu}
}

\begin{document}

\maketitle

\begin{abstract}
In this paper, we provide near-optimal accelerated first-order methods for minimizing a broad class of smooth nonconvex functions that are unimodal on all lines through a minimizer. This function class, which we call the class of smooth \emph{quasar-convex} functions, is parameterized by a constant $\gamma \in (0,1]$: $\gamma = 1$ encompasses the classes of smooth convex and star-convex functions, and smaller values of $\gamma$ indicate that the function can be ``more nonconvex.'' We develop a variant of accelerated gradient descent that computes an $\epsilon$-approximate minimizer of a smooth $\gamma$-quasar-convex function with at most $O(\gamma^{-1} \epsilon^{-1/2} \log(\gamma^{-1} \epsilon^{-1}))$ total function and gradient evaluations. We also derive a lower bound of $\Omega(\gamma^{-1} \epsilon^{-1/2})$ on the worst-case number of gradient evaluations required by any deterministic first-order method, showing that, up to a logarithmic factor, no deterministic first-order method can improve upon ours.

\end{abstract}

\section{Introduction}
\label{sec:intro}

Acceleration \citep{nemirovski1982, nesterov1983method} is a powerful tool for improving the performance of first-order optimization methods. Accelerated gradient descent (AGD) obtains asymptotically optimal runtimes for smooth convex minimization. Furthermore, acceleration has been used to obtain improved rates for stochastic optimization \citep{JohnsonZh13,allen2017katyusha,ghadimi2016accelerated,woodworth2016tight,xu2018pca}, coordinate descent methods \citep{nesterov2012efficiency,fercoq2015accelerated,hanzely2018acd,shalev2014accelerated}, proximal methods \citep{frostig2015regularizing,li2015prox,lin2015universal}, and higher-order optimization
\citep{bubeck2018near, gasnikov18arxiv, jiang2018arxiv}.
Acceleration has also been successful in a wide variety of practical applications, such as image deblurring \citep{beck2009fast} and neural network training \citep{sutskever2013importance}.
In addition, there has been extensive work giving alternative interpretations of acceleration \citex{allen2014linear,bubeck2015geometric,su2014differential}.

More recently, acceleration techniques have been applied to speed up the computation of $\ep$-stationary points (points where the gradient has norm at most $\ep$) of smooth \emph{nonconvex} functions \citep{agarwal2017finding,carmon2017convex,carmon2018accelerated}. In particular, while gradient descent's $O(\ep^{-2})$ rate for finding $\ep$-stationary points of nonconvex functions with Lipschitz gradients is optimal among first-order methods, if higher-order smoothness assumptions are made accelerated methods can improve this to $O(\ep^{-5/3}\log(\ep^{-1}))$ \citep{carmon2017convex}. Further, \cite{carmon2017lower2} shows that under the same assumptions, any dimension-free deterministic first-order method requires at least $\Omega(\ep^{-8/5})$ iterations to compute an $\ep$-stationary point in the worst case. These bounds are significantly worse than the corresponding $O(\ep^{-1/2})$ bound that AGD achieves for smooth convex functions.

Still, in practice it is often possible to find approximate stationary points, and even approximate global minimizers, of nonconvex functions faster than these lower bounds suggest. This performance gap stems from the fairly weak assumptions underpinning these generic bounds. For example, \cite{carmon2017lower2,carmon2017lower} only assume Lipschitz continuity of the gradient and some higher-order derivatives. However, functions minimized in practice often admit significantly more structure, even if they are not convex. For example, under suitable assumptions on their inputs, several popular nonconvex optimization problems, including matrix completion, deep learning, and phase retrieval, display ``convexity-like'' properties, e.g. that all local minimizers are global \citep{bartlett2019gradient,ge2016matrix}. Much more research is needed to characterize structured sets of functions for which minimizers can be efficiently found; our work is a step in this direction.

The class of ``structured'' nonconvex functions that we focus on in this work is the class of functions we term \textit{quasar-convex}. Informally, quasar-convex functions are unimodal on all lines that pass through a global minimizer. This function class is parameterized by a constant $\1 \in (0,1]$, where $\1 = 1$ implies the function is star-convex \citep{nesterov2006cubic} (itself a generalization of convexity), and smaller values of $\1$ indicate the function can be ``even more nonconvex.''\footnote{An example of a practical problem that is
quasar-convex but \emph{not} star-convex is the objective for learning linear dynamical systems (under certain conditions) \cite{weakquasiconvexity}. \mbox{We present numerical experiments for this problem in Appendix \ref{sec:experiments}.}}
 We produce an algorithm that, given a smooth $\1$-quasar-convex function, uses $O(\1^{-1}\ep^{-1/2}\log(\1^{-1}\ep^{-1}))$ function and gradient queries to find an $\epsilon$-optimal point. Additionally, we provide nearly matching query complexity lower bounds of $\Omega(\gamma^{-1}\ep^{-1/2})$ for \textit{any} deterministic first-order method applied to this function class. Minimization on this function class has been studied previously \citep{guminov2017accelerated,nesterov2018primal}; our bounds more precisely characterize its complexity.

\paragraph{Basic notation.} Throughout this paper, we use $\norm{\cdot}$ to denote the Euclidean norm (i.e. $\norm{\cdot}_2$). We say that a function $f : \R^n \ra \R$ is $L$-smooth, or $L$-Lipschitz differentiable, if $\norm{\nabla f(x)-\nabla f(y)} \le L\norm{x-y}$ for all $x,y \in \R^n$. (We say a function is \emph{smooth} if it is $L$-smooth for some $L \in [0, \infty)$.) We denote a minimizer of $f$ by $x^*$, and we say that a point $x$ is ``$\ep$-optimal'' or an ``$\ep$-minimizer'' if $f(x) \le f(x^*) + \ep$. We use `$\log$' to denote the natural logarithm and $\log^+(\cdot)$ to denote $\max\{\log(\cdot), 1\}$.

\subsection{Quasar-Convexity: Definition, Motivation, and Prior Work}
\label{sec:motivate-and-define-quasar}
In this work, we improve upon the state-of-the-art complexity of first-order minimization of \emph{quasar-convex} functions,\footnote{The concept of quasar-convexity was first introduced by \cite{weakquasiconvexity}, who refer to it as `weak quasi-convexity'. We introduce the term `quasar-convexity' because we believe it is linguistically clearer. In particular, `weak quasi-convexity' is a misnomer because it does not subsume quasi-convexity. Moreover, using this terminology, strong quasar-convexity would be confusingly termed `strong weak quasi-convexity.'} which are defined as follows.

\begin{defn}
Let $\1 \in (0,1]$ and let $\xStar$ be a minimizer of the differentiable function $f : \R^{n} \rightarrow \R$.
The function $f$ is \emph{$\1$-quasar-convex} with respect to $x^*$ if for all $x \in \R^n$,
\begin{equation}
\label{eq:qc}
f(\xStar) \ge  f(x) + \frac{1}{\1} \grad f(x)^\top (\xStar-x).
\end{equation}
Further, for $\2 \ge 0$, the function $f$ is \emph{$(\1,\2)$-strongly quasar-convex}\footnote{By Observation \ref{obs:unique}, $x^*$ is unique if $\mu > 0$.} (or \emph{$(\1,\2)$-quasar-convex} for short) if for all $x \in \R^n$,
\begin{equation}
\label{eq:sqc}
f(\xStar) \ge f(x) + \frac{1}{\1} \grad f(x)^\top (\xStar-x) + \frac{\2}{2} \norm{ \xStar -x }^2.
\end{equation}
\label{defn:defns}
\end{defn}
We simply say that $f$ is \emph{quasar-convex} if \eqref{eq:qc} holds for some minimizer $\qx$ of $f$ and some constant $\1 \in (0, 1]$, and \emph{strongly quasar-convex} if \eqref{eq:sqc} holds with some constants $\1 \in (0,1], \mu > 0$. We refer to $\qx$ as the ``quasar-convex point'' of $f$.
Assuming differentiability, in the case $\1 = 1$, condition \eqref{eq:qc} is simply star-convexity \citep{nesterov2006cubic};\footnote{When $\1 = 1$, condition \eqref{eq:sqc} is also known as \textit{quasi-strong convexity} \citep{necoara} or \textit{weak strong convexity} \citep{karimi2016linear}.} if in addition the conditions \eqref{eq:qc} or \eqref{eq:sqc} hold for all $y \in \R^n$ instead of just for $x^*$, they become the standard definitions of convexity or $\2$-strong convexity, respectively \citep{BoydVa04}. Definition~\ref{defn:defns} can also be straightforwardly generalized to the case where the domain of $f$ is a convex subset of $\R^n$ (see Definition~\ref{defn:gen-defns} in \Cref{sec:quasar-structure}). Thus, our definition of quasar-convexity \linebreak generalizes the standard notions of convexity and star-convexity in the differentiable case. Lemma \ref{lem:star_char} in Appendix~\ref{sec:equivs} shows that quasar-convexity is equivalent to a certain ``convexity-like'' condition on line segments to $\xStar$. In Figure~\ref{fig:examples}, we plot example quasar-convex functions.

We say that a one-dimensional function is \textit{unimodal} if it monotonically decreases to its minimizer and then monotonically increases thereafter. As Observation~\ref{obs:unimodalImpliesQuasar} shows, quasar-convexity is closely related to unimodality. Therefore, like the well-known quasiconvexity \citep{quasi} and pseudoconvexity \citep{pseudo}, quasar-convexity can be viewed as an approximate generalization of unimodality to higher dimensions. We remark that beyond one dimension, neither quasiconvexity nor pseudoconvexity subsumes or is subsumed by quasar-convexity. The proof of Observation~\ref{obs:unimodalImpliesQuasar} appears in Appendix~\ref{sec:unimodal-quasar}, and follows readily from the definitions.

\begin{figure}[t]
\center
\includegraphics[height=1.5in]{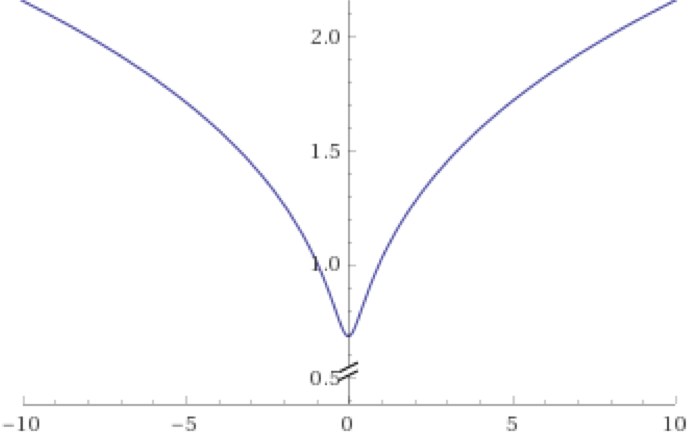}
\includegraphics[height=1.6in]{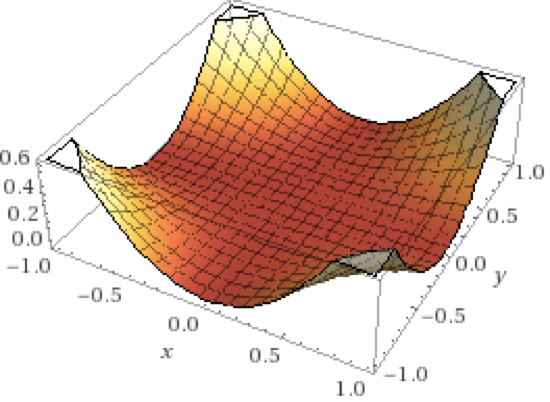}
\includegraphics[height=1.5in]{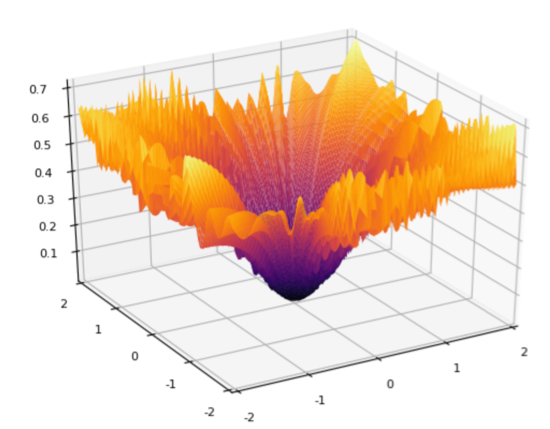}
\caption[Examples of Quasar-Convex Functions]{Examples of quasar-convex functions. [Left: $f(x) = \lt x^2 + \ff{1}{8}\rt^{1/6}$ (quasar-convex with $\gamma=\ff{1}{2}$). Middle: $f(x,y) = x^2y^2$ (star-convex). The rightmost function is described in Appendix \ref{app:quasar_construction}.]}
\label{fig:examples}
\end{figure}
\begin{restatable}{observation}{unimodalImpliesQuasar}\label{obs:unimodalImpliesQuasar}
Let $a < b$ and let $f : [a,b] \rightarrow \R$ be continuously differentiable. The function $f$ is $\1$-quasar-convex for some $\1 \in (0, 1]$ iff $f$ is unimodal and all critical points of $f$ are minimizers. Additionally, if $h : \reals^{n} \rightarrow \reals$ is $\gamma$-quasar-convex with respect to a minimizer $\qx$, then for any $d \in \R^n$ with $\norm{ d } = 1$, the 1-D function $f(\theta) \defeq h(x^{*} + \theta d)$ is $\gamma$-quasar-convex.
\end{restatable}

\subsubsection{Related Work}
\label{sec:star-related}
There are several other `convexity-like' conditions in the literature related to quasar-convexity. For example, star-convexity is a condition that relaxes convexity, and is a strict subset of quasar-convexity in the differentiable case. \cite{nesterov2006cubic} introduces this condition when analyzing cubic regularization. \cite{lee2016optimizing} further investigates star-convexity, developing a cutting-plane method to minimize general star-convex functions. Star-convexity is an interesting property because there is some evidence to suggest the loss function of neural networks might conform to this structure in large neighborhoods of the minimizers \citep{kleinberg2018alternative,zhou2019sgd}. Furthermore, under mild assumptions, the objective for learning linear dynamical systems is quasar-convex \cite{weakquasiconvexity}; this problem is closely related to the training of recurrent neural networks. Another relevant class of functions is those for which a small gradient implies approximate optimality. This is known as the Polyak-\L{}ojasiewicz (PL) condition \citep{polyak} and is weaker than strong quasar-convexity \citep{guminov2017accelerated}. For linear residual networks, the PL condition holds in large regions of parameter space \citep{hardt2016identity}.
In addition to pseudoconvexity, quasiconvexity, star-convexity, and the PL condition, other relaxations of convexity or strong convexity include invexity \citex{craven1985invex}, semiconvexity \citex{ngai2007}, quasi-strong convexity \citex{necoara}, restricted strong convexity \citex{zhang2013gradient}, one-point convexity \citex{li2017relu}, variational coherence \citex{zhou2017vc}, the quadratic growth condition \citex{anitescu}, and the error bound property \citex{fabian2010}. A more thorough discussion is provided in \Cref{sec:related-classes}.

We are not the first to study acceleration on quasar-convex functions; recent work by \cite{guminov2017accelerated} and \cite{nesterov2018primal} shows how to achieve accelerated rates for minimizing quasar-convex functions. For a function that is $L$-smooth and $\gamma$-quasar-convex with respect to a minimizer $x^*$, with initial distance to $x^*$ bounded by $R$, the algorithm of \cite{guminov2017accelerated} yields an $\ep$-optimal point in $O(\1^{-1} L^{1/2} R \ep^{-1/2})$ iterations, while the algorithm of \cite{nesterov2018primal} does so in $O(\1^{-3/2} L^{1/2} R\ep^{-1/2})$ iterations. For convex functions (which have $\gamma = 1$), these bounds match the \textit{iteration} bounds achieved by AGD \citep{nesterov1983method}, but use a different oracle model. In particular, to achieve these iteration bounds, the method in \cite{guminov2017accelerated} relies on a low-dimensional subspace optimization method within each iteration, while \cite{nesterov2018primal} uses a one-dimensional line search over the function value in each iteration, as well as a restart criterion that requires knowledge of the true optimal function value.\footnote{We discuss \cite{nesterov2018primal} in more detail in Appendix \ref{app:nesterov}.}
However, quasar-convex functions are not necessarily unimodal along the arbitrary low-dimensional regions or line segments being searched over. Therefore, even finding an approximate minimizer \emph{within} these subregions may be computationally expensive, making \emph{each iteration} potentially costly; by contrast, our methods only require a function and gradient oracle. In addition, neither paper provides lower bounds nor studies the ``strongly quasar-convex'' regime.
Independently, recent work by \cite{sra2019} uses a differential equation discretization to approach the accelerated $O(\k^{1/2}\log(\ep^{-1}))$ rate for minimization of smooth strongly quasar-convex functions in a neighborhood of the optimum, in the special case $\1 = 1$.\footnote{$\k = L/\2$ denotes the \textit{condition number} of an $L$-smooth $(\1,\2)$-strongly quasar-convex function.} Similarly, in the $\1 = 1$ case, geometric descent \citex{bubeck2015geometric} achieves $O(\k^{1/2}\log(\ep^{-1}))$ running times in terms of the number of calls to a one-dimensional line search oracle (although, as previously noted, the number of function and gradient evaluations required may still be large).\footnote{Although this result is not explicitly stated in the literature, upon careful inspection of the analysis in \citex{bubeck2015geometric} it can be observed that the $\2$-strong convexity requirement in \citex{bubeck2015geometric} may be relaxed to the requirement of $(1,\2)$-strong quasar-convexity, with no changes to the algorithm necessary.}

\subsection{Summary of Results}
For functions that are $L$-smooth and $\gamma$-quasar-convex, we provide an algorithm which finds an $\ep$-optimal solution in $O(\1^{-1} L^{1/2} R \ep^{-1/2})$ iterations  (where, as before, $R$ is an upper bound on the initial distance to the quasar-convex point $\qx$). Our iteration bound is the same as that of \cite{guminov2017accelerated}, and a factor of $\1^{1/2}$ better than the $O(\1^{-3/2} L^{1/2} R \ep^{-1/2})$ bound of \cite{nesterov2018primal}. Additionally, we are the first to provide bounds on the total number of \emph{function and gradient evaluations} required; our algorithm uses $O(\1^{-1} L^{1/2} R \ep^{-1/2}\log(\1^{-1} \ep^{-1}))$ evaluations to find a $\ep$-optimal solution.

We also provide an algorithm for $L$-smooth, $(\1,\2)$-strongly quasar-convex functions; our algorithm uses $O(\1^{-1} \kappa^{1/2} \log(\1^{-1}\ep^{-1}) )$ iterations and $O(\1^{-1} \kappa^{1/2} \log(\1^{-1} \kappa) \log(\1^{-1} \ep^{-1}) )$ total function and gradient evaluations to find an $\ep$-optimal point, where $\kappa \defeq L / \mu$ ($\kappa$ is typically referred to as the \textit{condition number}). For constant $\1$, this matches accelerated gradient descent's bound for smooth strongly convex functions, up to a logarithmic factor.

The key idea behind our algorithm is to take a close look at which essential invariants need to hold during the momentum step of AGD, and use this insight to carefully redesign the algorithm to accelerate on general smooth quasar-convex functions. By observing how the function behaves along the line segment between current iterates $x\ind{k}$ and $v\ind{k}$, we show that for any smooth quasar-convex function, there always exists a point $y\ind{k}$ along this segment with the properties needed for acceleration. Furthermore, we show that an efficient binary search can be used to find such a point, \mbox{even without the assumption of convexity along the segment.}

To complement our upper bounds, we provide lower bounds of $\Omega(\1^{-1} L^{1/2}R \ep^{-1/2})$ for the number of gradient evaluations that \textit{any} deterministic first-order method requires to find an $\ep$-minimizer of a quasar-convex function. This shows that up to logarithmic factors, our lower and upper bounds are tight. Our lower bounds extend the techniques from \cite{carmon2017lower2} to the class of smooth quasar-convex functions, allowing an almost exact characterization of the complexity of minimizing these functions.

\paragraph{Paper outline.} In Section~\ref{sec:acceleration_framework}, we provide a general framework for accelerating the minimization of smooth quasar-convex functions. In Section~\ref{sec:algorithms}, we apply our framework to develop specific algorithms tailored to both quasar-convex and strongly quasar-convex functions. In Section~\ref{sec:lb}, we provide lower bounds to show that the upper bounds for quasar-convex minimization of Section~\ref{sec:algorithms} are tight up to logarithmic factors. Full proofs, additional results, and numerical experiments are in the Appendix.
\section{Quasar-Convex Minimization Framework}
\label{sec:acceleration_framework}

In this section, we provide and analyze a general algorithmic template for accelerated minimization of smooth quasar-convex functions. In \Cref{sec:strongly-quasar-convex} we show how to leverage this framework to achieve accelerated rates for minimizing \textit{strongly} quasar-convex functions, and in \Cref{sec:quasar-convex} we show how to achieve accelerated rates for minimizing \textit{non-strongly} quasar-convex functions (i.e. when $\mu = 0$). For simplicity, we assume the domain is $\R^n$.

Our algorithm (Algorithm~\ref{alg:agd}) is a simple generalization of accelerated gradient descent (AGD).
Indeed, standard AGD can be written in the form of Algorithm~\ref{alg:agd}, for particular choices of the parameters $\alpha\ind{k},\beta\ind{k},\eta\ind{k}$.
Given a differentiable function $f \in \R^n \rightarrow \R$ with smoothness parameter $L > 0$ and initial point $x^{(0)} = v^{(0)} \in \R^n$, the algorithm iteratively computes points $x^{(k)}, v^{(k)} \in \R^n$ of improving ``quality.''
However, it is challenging to argue that Algorithm \ref{alg:agd} actually performs optimally \textit{without the assumption of convexity}. The crux of circumventing convexity is to show that there exists a way to efficiently compute the momentum parameter $\alpha^{(k)}$ to yield convergence at the desired rate. In this section, we provide general tools for analyzing this algorithm; in \Cref{sec:algorithms}, we leverage this analysis with specific choices of the parameters $\A^{(k)}, \B$, and $\eta^{(k)}$ to derive our fully-specified accelerated schemes for both quasar-convex and strongly quasar-convex functions.

\begin{algorithm}[H]
\caption{General AGD Framework}
\label{alg:agd}
\SetAlgoLined
\textbf{Input}: $L$-smooth function $f: \R^n \rightarrow \R$, initial point $x^{(0)} \in \R^n$, number of iterations $K$ \\
\BlankLine
Sequences $\{\A^{(k)}\}_{k=0}^{K-1}$, $\{\B\ind{k}\}_{k=0}^{K-1}$, $\{L^{(k)}\}_{k=0}^{K-1}$, $\{\eta^{(k)}\}_{k=0}^{K-1}$
are defined by the particular
\vskip 0ex
algorithm instance, where $\A^{(k)} \in [0,1], \, \B^{(k)} \in [0,1], \, L\ind{k} \in (0, 2L), \, \eta^{(k)} \ge \ff{\1}{L\ind{k}}$.
\BlankLine
\nl Set $v^{(0)} = x^{(0)}$ \\
\nl \For{$k = 0,1,2,\dots,K-1$}
{
\nl Set $y^{(k)} = \A^{(k)} x^{(k)} +(1-\A^{(k)})v^{(k)}$ \\
\nl Set $x^{(k+1)} = y^{(k)} - \frac{1}{L^{(k)}} \grad f( y^{(k)} )$ \quad \# {\scriptsize $L^{(k)}$ computed s.t. $f(x^{(k+1)}) \le f(y\ind{k}) - \ff{1}{2L\ind{k}} \norm{\G f(y\ind{k})}^2$} \\
\nl Set $v^{(k+1)} = \B\ind{k} v^{(k)} + (1-\B\ind{k}) y^{(k)} - \eta^{(k)} \G f ( y^{(k)} )$ \\
}
\nl \Return{$x^{(K)}$}
\BlankLine
\end{algorithm}

\noindent We first define notation that will be used throughout Sections~\ref{sec:acceleration_framework} and \ref{sec:algorithms}:
\begin{defn}
Let  $\ep^{(k)} \defeq f( x^{(k)} ) - f(\qx), \ep^{(k)}_y \defeq f( y^{(k)} ) - f(\qx), r^{(k)} \defeq \norm{v^{(k)} - \qx}^2,
\newline r^{(k)}_y \defeq \norm{y^{(k)} - \qx}^2,
Q^{(k)} \defeq \B\ind{k} \lt 2\eta^{(k)}\A^{(k)} \grad f( y^{(k)} )^\top (x^{(k)} - v^{(k)}) - ( \A^{(k)})^2 (1-\B\ind{k}) \norm{x^{(k)}-v^{(k)}}^2 \rt$.
\end{defn}

In the remainder of this section, we analyze Algorithm~\ref{alg:agd}. We assume that $f$ is $L$-smooth and $(\1, \2)$ strongly quasar-convex (possibly with $\2 = 0$) with respect to a minimizer $x^*$. First, we use Lemma~\ref{lem:agd_one_step} to bound how much the function error of $x^{(k)}$ and the distance from $v^{(k)}$ to $x^*$ decrease at each iteration. 

\begin{restatable}[One Step Framework Analysis]{lem}{agdonestep}
\label{lem:agd_one_step}
Suppose $f$ is $L$-smooth and $(\1,\2)$-quasar-convex with respect to a minimizer $x^*$. Then, in each iteration $k \ge 0$ of Algorithm~\ref{alg:agd} applied to $f$, it is the case that
\[
2 ( \eta^{(k)})^2 L\ind{k} \ep^{(k+1)} +
r^{(k+1)}  \leq \B\ind{k} r^{(k)} + \left[(1 - \B\ind{k}) -  \1 \2 \eta^{(k)}\right] r^{(k)}_y
+ 2 \eta^{(k)} \left[L\ind{k} \eta^{(k)} - \1 \right] \ep^{(k)}_y 
+ Q^{(k)}.
\]
\end{restatable}

\newcommand{\agdonestepProof}{
\begin{proof}
Let $z^{(k)} \defeq \beta\ind{k} v^{(k)} + {(1- \beta\ind{k})}y^{(k)}$. Since $v^{(k+1)} = z^{(k)} - \eta^{(k)} \grad f( y^{(k)} )$,
direct algebraic manipulation yields that
\begin{align}
	r^{(k+1)} 
	&= \norm{v^{(k+1)}-\qx}^2 =  \norm{z^{(k)} - \qx - \eta^{(k)} \G f( y^{(k)} )}^2  \nonumber \\
	&= \norm{z^{(k)}-\qx}^2 + 2 \eta^{(k)} \G f( y^{(k)} )^\top (\qx-z^{(k)}) + ( \eta^{(k)})^2 \norm{\G f( y^{(k)} )}^2 ~. \label{eq:agd_framework_1}
\end{align}
Using the definitions of $z^{(k)}$ and $y^{(k)}$, we have
\begin{align}
\norm{z^{(k)}-\qx}^2 &= \B\ind{k} \norm{ v^{(k)}-\qx}^2 + (1-\B\ind{k})\norm{y^{(k)}-\qx}^2 - \B\ind{k}(1-\B\ind{k})\norm{v^{(k)}-y^{(k)}}^2 \nonumber \\
&= \B\ind{k} r^{(k)} + (1-\B\ind{k}) r^{(k)}_y - \B\ind{k}(1-\B\ind{k})(\A^{(k)})^2\norm{v^{(k)}-x^{(k)}}^2~.
\label{eq:agd_framework_2}
\end{align}
Further, since $v^{(k)} = y^{(k)} + \A^{(k)} (v^{(k)} - x^{(k)})$ and $z^{(k)} = \B\ind{k} v^{(k)} + (1 - \B\ind{k})y^{(k)} = y^{(k)} + \A^{(k)} \B\ind{k} (v^{(k)} - x^{(k)})$,
it follows that 
\begin{equation}
\label{eq:agd_framework_3}
\grad f( y^{(k)} )^\top (\qx - z^{(k)})
= \grad f( y^{(k)} )^\top (\qx - y^{(k)})
+ \A^{(k)} \B\ind{k} \grad f( y^{(k)} )^\top (x^{(k)} - v^{(k)}) ~.
\end{equation}
Since $(\1, \2)$-strong quasar-convexity of $f$ implies
$-\ep^{(k)}_y \geq \frac{1}{\1} \grad f( y^{(k)} )^\top (\qx - y^{(k)}) + \frac{\2}{2} r^{(k)}_y$
and the definition of $x^{(k+1)}$ and $L\ind{k}$ implies $0 \le \norm{\grad f( y^{(k)} )}^2 \leq 2L^{(k)} [\ep^{(k)}_y - \ep^{(k+1)}]$,
combining with \eqref{eq:agd_framework_1}, \eqref{eq:agd_framework_2}, and \eqref{eq:agd_framework_3} yields the result.

Note that $L\ind{k}$ in Line 3 of Algorithm~\ref{alg:agd} can be set to the Lipschitz constant $L$ if it is known; otherwise, it can be efficiently computed to make $f(x\ind{k}) = {f(y\ind{k}-\ff{1}{L\ind{k}}\G f(y\ind{k}))} \le f(y\ind{k}) - \ff{1}{2L\ind{k}} \norm{\G f(y\ind{k})}^2$ and $L\ind{k} \in (0, 2L)$ hold using backtracking line search. See Lemma~\ref{lem:stepsize} (Appendix~\ref{sec:onestep-proofs}) for more details.
\end{proof}
}

\agdonestepProof

Lemma \ref{lem:agd_one_step}
provides our main bound on how the error $\ep^{(k)}$ changes between successive iterations of Algorithm~\ref{alg:agd}. The key step necessary to apply this lemma is to relate
$f(y^{(k)})$ and ${\grad f( y^{(k)} )^\top (x^{(k)} - v^{(k)})}$ to $f(x^{(k)})$, in order to bound $Q\ind{k}$. In the standard analysis of accelerated gradient descent, convexity is used to obtain such a connection.\footnote{See Appendix \ref{app:agd-analysis} for more details.} In our algorithms, we instead perform binary search to compute the momentum parameter $\alpha^{(k)}$ for which the necessary relationship holds without assuming convexity. The following lemma
shows that there always exists a setting of $\alpha^{(k)}$ that satisfies the necessary relationship.

\begin{figure}[t]
\center
\includegraphics[height=2in]{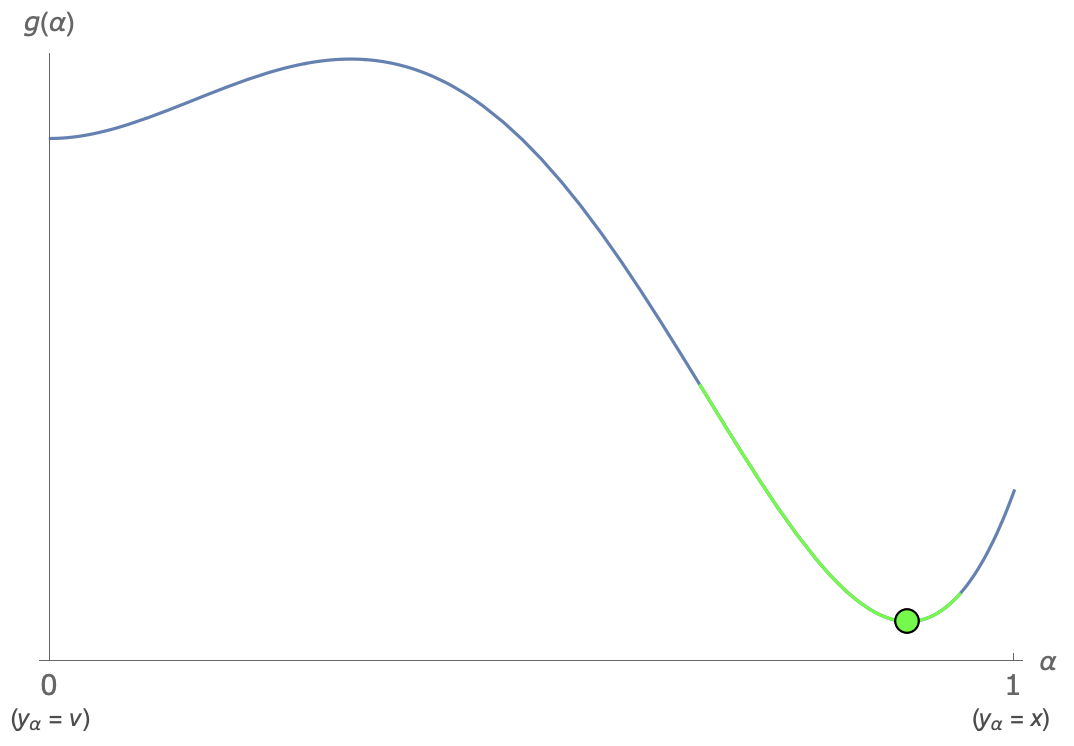}
\caption[Illustration of Lemma \ref{lem:ak_existence}]{Illustration of Lemma \ref{lem:ak_existence}. $g(\A)$ is defined as in the proof of the lemma; here, we depict the case where $g(0) > g(1)$ and $g'(1) > 0$. The points highlighted in green satisfy inequality \eqref{eq:ak_existence}; the circled point has $g'(\A) = 0$ and $g(\A) \le g(1)$. Here $c = 10$.}
\label{fig:lem2}
\end{figure}

\begin{restatable}[Existence of ``Good'' $\A$]{lem}{akexistence}
\label{lem:ak_existence}
Let $f : \R^n \rightarrow \R$ be differentiable and let $x, v \in \R^n$. For $\alpha \in \R$ define $y_\alpha \defeq \alpha x + (1 - \alpha) v$. For any $c \ge 0$ there exists $\alpha \in [0, 1]$ such that 
\begin{equation}
\label{eq:ak_existence}
\alpha \grad f(y_\alpha)^\top (x  - v) \leq c \left[f(x) - f(y_\alpha) \right] ~.
\end{equation}
\end{restatable}

\newcommand{\akexistproof}{
\begin{proof}
Define $g(\alpha) \defeq f(y_\alpha)$. Then for all $\alpha \in \R$ we have $g'(\alpha) = \grad f(y_\alpha)^\top (x - v)$. Consequently, \eqref{eq:ak_existence} is equivalent to the condition $\alpha g'(\alpha) \leq c [g(1) - g(\alpha)]$.

If $g'(1) \leq 0$, inequality \eqref{eq:ak_existence} trivially holds at $\alpha = 1$; if $f(v) = g(0) \le g(1) = f(x)$, the inequality trivially holds at $\alpha = 0$. If neither of these conditions hold, $g'(1) > 0$ and $g(0) > g(1)$, so Fact \ref{fact:minimizer} {from \Cref{sec:linesearch-proofs}} implies that there is a value of $\A \in (0,1)$ such that $g'(\A) = 0$ and $g(\A) \le g(1)$,
and therefore this value of $\A$ satisfies \eqref{eq:ak_existence}. Figure~\ref{fig:lem2} illustrates this third case graphically.
\end{proof}
}

\akexistproof

\noindent In our algorithms, we will not seek an $\A$ satisfying (\ref{eq:ak_existence}) exactly, but instead $\A \in [0, 1]$ such that
\begin{equation}
\label{eq:ak_existence_2}
\alpha \grad f(y_\alpha)^\top (x  - v) - \alpha^2 b\norm{x-v}^2 \leq c \left[f(x) - f(y_\alpha) \right] + \tilde{\ep}~,
\end{equation}
for some $b,c,\tilde{\ep} \ge 0$. As \eqref{eq:ak_existence_2} is a weaker statement than \eqref{eq:ak_existence}, the existence of $\A$ satisfying \eqref{eq:ak_existence_2} follows from Lemma \ref{lem:ak_existence}. Moreover, we will show how to lower bound the size of the set of points satisfying \eqref{eq:ak_existence_2}, which we use to bound the time required to compute such a point.

We can thus bound the quantity $Q^{(k)}$ from Lemma \ref{lem:agd_one_step} by selecting $\A^{(k)}$ to satisfy \eqref{eq:ak_existence_2} with appropriate settings of $b,c,\tilde{\ep}$, which we do in Lemma \ref{lem:agd_linesearch}.

{\setstretch{0.1}
\begin{restatable}{lem}{agdlinesearch}
\label{lem:agd_linesearch}
If $\B\ind{k} > 0$ and $\A^{(k)} \in [0,1]$ satisfies \eqref{eq:ak_existence_2} with $x = x^{(k)}, v = v^{(k)}, \, b = \ff{1-\B\ind{k}}{2\eta^{(k)}}$, and $c = \ff{L\ind{k}\eta^{(k)}-\1}{\B\ind{k}}$, or if $\B\ind{k} = 0$ and $\A^{(k)} = 1$, then
\begin{equation}
\label{eq:onestep_ineq}
Q\ind{k} \le 2\eta\ind{k}\left[(L\ind{k}\eta\ind{k} - \1 ) \cdot (\ep\ind{k} - \ep\ind{k}_y) + \B\ind{k} \tilde{\ep}  \right].
\end{equation}
\end{restatable}
}

\newcommand{\agdLinesearchProof}{
\begin{proof}
First suppose $\B\ind{k} > 0$. As by definition $y^{(k)} = \A^{(k)} x^{(k)} + (1-\A^{(k)})v^{(k)}$ and $L\ind{k}\eta^{(k)} \ge \1$, applying \eqref{eq:ak_existence_2} yields
\begin{align*}
Q^{(k)} &= 2\B\ind{k} \eta^{(k)} \lt \A^{(k)} \grad f( y^{(k)} )^\top (x^{(k)} - v^{(k)}) - \lt \A^{(k)}\rt^2 \f{(1-\B\ind{k}) \norm{x^{(k)}-v^{(k)}}^2}{2\eta^{(k)}}\rt \\
&\le 2\B\ind{k} \eta^{(k)} \lt \f{L\ind{k} \eta^{(k)} - \1}{\B\ind{k}} [ f(x^{(k)}) - f(y^{(k)})] + \tilde{\ep} \rt \\
&= 2 \eta^{(k)} \lt [L\ind{k} \eta^{(k)} - \1] \cdot [ \ep^{(k)} - \ep^{(k)}_y] + \B\ind{k} \tilde{\ep} \rt~.
\end{align*}

Alternatively, suppose $\B\ind{k} = 0$. Then $Q^{(k)} = 0$ as well; if we select $\A^{(k)} = 1$, then $y^{(k)} = x^{(k)}$ and \eqref{eq:onestep_ineq} trivially holds for any $\tilde{\ep}$, as $\ep^{(k)}_y = \ep^{(k)}$.
\end{proof}
}

\agdLinesearchProof

\noindent Now, in Algorithm~\ref{alg:linesearch} we show how to efficiently compute an $\alpha$ satisfying inequality \eqref{eq:ak_existence_2}.

\newcommand{\lo}{{\normalfont \textbf{lo}}}
\newcommand{\hi}{{\normalfont \textbf{hi}}}
\newcommand{\lhi}{\tau}

\setcounter{AlgoLine}{0}
\begin{algorithm}[H]
\setstretch{0.95}
\SetAlgoLined
\caption{\texttt{BinaryLineSearch}($f, x, v, b, c, \tilde{\ep}, [\text{guess}]$)}
\label{alg:linesearch}
\SetKwInOut{Input}{input}
\textit{Assumptions}: $f$ is $L$-smooth;\, $x,v \in \R^n$;\, $b,c,\tilde{\ep} \ge 0$;\, ``guess'' (optional) is in $[0,1]$ if provided.
\vskip 0ex
Define $g(\A) \defeq f(\A x + (1-\A) v)$ and $p \defeq b\norm{x-v}^2$. \\
\nl \textbf{if} \text{guess provided} \textbf{and }$c g(\text{guess}) + \text{guess} \cdot (g'(\text{guess}) - \text{guess} \cdot p) \le c g(1) + \tilde{\ep}$\, \mbox{\textbf{then return} guess}\\
\nl \lIf{$g'(1) \le \tilde{\ep}+p$}{\Return 1}
\nl \lElseIf{$c = 0\,\,\mathrm{ or }\,\,g(0) \le g(1) + \tilde{\ep} / c$\,}{\Return 0}
\nl $\lhi \leftarrow 1 - g'(1) \, / \, \texttt{BacktrackingSearch}(g, p, 1)$ \quad {\small \# one step of gradient descent on $g$ from 1,
using backtracking to select step size; see Algorithm~\ref{alg:backtracking} for \texttt{BacktrackingSearch} pseudocode} \\
\nl $\lo \leftarrow 0, \hi \leftarrow \lhi, \A \leftarrow \lhi$ \\
\nl \While{$c g(\A) + \A (g'(\A)- \A p) > c g(1) + \tilde{\ep}$}{
\nl $\A \leftarrow (\lo+\hi)/2$ \\
\nl \lIf{$g(\A) \le g(\lhi)$}{$\hi \leftarrow \A$}
\nl \lElse{$\lo \leftarrow \A$}
}
\nl \Return{$\A$}
\BlankLine
\end{algorithm}

The core idea behind Algorithm~\ref{alg:linesearch} is as follows: as in the proof of Lemma \ref{lem:ak_existence}, let $g(\A) \defeq f(\A x + (1-\A)v)$ be the restriction of the function $f$ to the line from $v$ to $x$. If either $g(0) \le g(1)$, or $g$ is decreasing at $\A = 1$, then \eqref{eq:ak_existence} is immediately satisfied. If this does not happen, then $g(0) > g(1)$ but $g'(1) > 0$, which means that $g$ switches from increasing to decreasing at some $\A \in (0,1)$, and so $g'(\A) = 0$. Such a value of $\A$ also satisfies \eqref{eq:ak_existence}. Algorithm~\ref{alg:linesearch} uses binary search to exploit this fact and thereby efficiently compute a value of $\A$ \textit{approximately} satisfying \eqref{eq:ak_existence} (i.e., satisfying \eqref{eq:ak_existence_2}). In Lemma \ref{lem:linesearch}, we bound the maximum number of iterations that Algorithm \ref{alg:linesearch} can take until \eqref{eq:ak_existence_2} holds and it thereby terminates. Lemma \ref{lem:linesearch} is proved in \Cref{sec:linesearch-proofs}.

``guess'' is an optional argument to Algorithm~\ref{alg:linesearch}; if given, the value of ``guess'' will be tested first, and chosen as the value of $\A$ if it satisfies \eqref{eq:ak_existence_2}. For instance, we can use the value of $\A^{(k)}$ prescribed by the standard version of AGD as an initial guess. We discuss this further in Section \ref{sec:guess}.

\begin{restatable}[Line Search Runtime]{lem}{linesearch}
\label{lem:linesearch}
For $L$-smooth $f : \R^n \rightarrow \R$, points $x, v \in \R^n$ and scalars $b,c,\tilde{\ep} \ge 0$, Algorithm~\ref{alg:linesearch} computes $\A \in [0,1]$ satisfying \eqref{eq:ak_existence_2} with at most
\[
8+3\ceil{\log_2^+ \lt (4+c) \min\left\{ \ff{2L^3}{b^3}, \ff{L\norm{x-v}^2}{2\tilde{\ep}} \right\} \rt}
\]
 function and gradient evaluations.
\end{restatable}

In summary, we achieve our accelerated quasar-convex minimization procedures (presented below) by setting $\eta^{(k)}, \beta\ind{k}$, $\ep$, and  $\alpha^{(k)}$ appropriately. In standard AGD, convexity is used to set a particular value of $\alpha\ind{k}$; by contrast, our accelerated quasar-convex minimization procedures relax the convexity assumption by computing an $\alpha\ind{k}$ satisfying \eqref{eq:ak_existence_2} via binary search (Algorithm~\ref{alg:linesearch}). By lower bounding the length of the interval of values of $\alpha^{(k)}$ satisfying \eqref{eq:ak_existence_2}, we show that this binary search only costs a logarithmic factor in the overall runtime.
\section{Algorithms}
\label{sec:algorithms}
In this section, we develop algorithms for accelerated minimization of strongly quasar-convex functions and quasar-convex functions, respectively, and analyze their running times in terms of the number of function and gradient evaluations required. We note that the Lipschitz constant $L$ does not need to be known; however, a lower bound $\hat{\1} > 0$ on $\1$ does need to be known, and the runtime depends inversely on $\hat{\1}$. In Appendix \ref{sec:experiments}, we provide numerical experiments on different types of quasar-convex functions, which validate the claim that our algorithm is not only efficient in theory but also empirically competitive with other first-order methods such as standard AGD.

\subsection{Strongly Quasar-Convex Minimization}
\label{sec:strongly-quasar-convex}
First, we provide and analyze our algorithm for $(\1, \2)$-strongly quasar-convex function minimization, where $\2 > 0$. The algorithm (Algorithm~\ref{alg:strongly_agd}) is a carefully constructed instance of the general AGD framework (Algorithm~\ref{alg:agd}).

As in the general AGD framework, the algorithm maintains two current points denoted $x^{(k)}$ and $v^{(k)}$ and at each step appropriately selects $y^{(k)} = \alpha^{(k)} x^{(k)} + (1 - \alpha^{(k)}) v^{(k)}$ as a convex combination of these two points. Intuitively, the algorithm iteratively seeks to decrease quadratic upper and lower bounds on the function value. $L$-smoothness of $f$ implies for all $x,y\in \R^n$ that $f(x) \leq UB_y(x) \defeq f(y) + {\grad f(y)^\top (x-y)} + \frac{L}{2} \norm{x - y}^2$; if $L^{(k)} = L$, then $x^{(k+1)}$ is the minimizer $y^{(k)} - \ff{1}{L} \G y^{(k)}$ of the upper bound $UB_{y^{(k)}}$. Similarly, by $(\1, \2)$ quasar-convexity, $f(x) \ge f(\qx) \ge \min_z LB_y(z)$ for all $x,y \in \R^n$, where $LB_y(x) \defeq f(y) + \frac{1}{\1} \grad f(y)^\top (x - y) + \frac{\2}{2} \norm{x-y}^2$. The minimizer of the lower bound $LB_{y^{(k)}}$ is $y^{(k)} - \frac{1}{\1 \2} \grad f(y^{(k)})$; we set $v^{(k+1)}$ to be a convex combination of $v^{(k)}$ and the minimizer of $LB_{y^{(k)}}$. 

\begin{algorithm}[H]
	\caption{Accelerated Strongly Quasar-Convex Function Minimization}
	\label{alg:strongly_agd}
	\SetAlgoLined
	\SetKwInOut{Input}{input}
	\textbf{Input}: $L$-smooth $f : \R^n \rightarrow \R$ that is $(\1,\2)$-strongly quasar-convex (with $\2 > 0$), \newline
	initial point $x^{(0)} \in \R^n$, number of iterations $K$, solution tolerance $\ep > 0$ \\
	\nl \textbf{return} output of Algorithm \ref{alg:agd} on $f$ with initial point $x^{(0)}$,
	where for all $k$,
	\vskip 0ex
	$L\ind{k} = \texttt{BacktrackingSearch}(f, \ff{\1\2}{2-\1}, x\ind{k})$, $\B\ind{k} =  1 - \1{\sqrt{\ff{\2}{L\ind{k}}}}$, \,$\eta^{(k)} = \frac{1}{\sqrt{\2 L\ind{k}}}$,
	\vskip 0ex and
	$\A^{(k)} = \Big(\texttt{BinaryLineSearch}(f, x^{(k)}, v^{(k)}, b = \ff{\1\2}{2}, c = \sqrt{\ff{L\ind{k}}{\mu}}, \tilde{\ep} = 0)$ \textbf{if} $\B\ind{k} > 0$ \textbf{else} 1\Big)
\end{algorithm}

We leverage the analysis from \Cref{sec:acceleration_framework} to analyze Algorithm \ref{alg:strongly_agd}. First, in Lemma \ref{lem:strong_converge} we show that the algorithm converges at the desired rate, by building off of Lemma \ref{lem:agd_one_step} and using the specific parameter choices in Algorithm \ref{alg:strongly_agd}.

\begin{restatable}[Strongly Quasar-Convex Convergence]{lem}{strongconverge}
\label{lem:strong_converge}
If $f$ is $L$-smooth and $(\1,\2)$-strongly quasar-convex with minimizer $x^*$, $\1 \in (0,1]$, and $\mu > 0$, then in each iteration $k \ge 0$ of Algorithm~\ref{alg:strongly_agd},
\begin{equation}
\label{eq:strong_agd_one_step}
\ep^{(k+1)} + \frac{\2}{2} r^{(k+1)} \le \left(1 - \frac{\1}{\sqrt{2\kappa}}\right) \left[\ep^{(k)} + \frac{\2}{2} r^{(k)} \right]~,
\end{equation}
where $\ep^{(k)} \defeq f( x^{(k)} ) - f(\qx), r^{(k)} \defeq \norm{ v^{(k)} - \qx }^2$, and $\k \defeq \ff{L}{\mu}$. Therefore,
if the number of iterations $K \ge \ceil{\frac{\sqrt{2\kappa}}{\1} \log^{+}\left( \frac{3\ep^{(0)}}{\1\ep} \right)}$,
then the output $x^{(K)}$ satisfies $f( x^{(K)} ) \le f(\qx) + \ep$.
\end{restatable}

\newcommand{\strongconvergeProof}[1]{
\begin{proof}
For all $k$, $\eta^{(k)} = \ff{1}{\sqrt{\mu L\ind{k}}} \ge \sqrt{\ff{\1}{(2-\1) (L\ind{k})^2}} \ge \ff{\1}{L\ind{k}}$ as required by Algorithm \ref{alg:agd}, since $\ff{x}{2-x} \ge x^2$ for all $x \in [0, 1]$
and since $\ff{(2-\1)L\ind{k}}{\1} \ge \2 > 0$ by definition of $L\ind{k}$ because we use $\ff{\1\2}{2-\1}$ (which is $\le L$ by Observation \ref{obs:l_vs_mu}) as the initial guess for $L\ind{k}$ and only increase it during the backtracking search.
Similarly, since $0 < \ff{\2}{L\ind{k}} \le \ff{2-\1}{\1}$ and $\1 \in (0, 1]$, we have
$0 < \1 \sqrt{\ff{\2}{L\ind{k}}} \le \sqrt{\1(2-\1)} \le 1$, meaning that $\B\ind{k} \in [0, 1)$.
Additionally, by construction, either $\B\ind{k} = 0$ and $\A^{(k)} = 1$, or $\B\ind{k} > 0$, $\A^{(k)} \in [0,1]$, and $(\A,x,y_{\alpha},v) = (\A^{(k)},x\ind{k},y\ind{k},v\ind{k})$ satisfies \eqref{eq:ak_existence_2} with $b = \ff{\1\2}{2}= \ff{1-\B\ind{k}}{2\eta^{(k)}}$, $c = \sqrt{\ff{L\ind{k}}{\mu}} = \ff{L\ind{k}\eta^{(k)}-\1}{\B\ind{k}}$, $\tilde{\ep} = 0$.
Consequently, by combining Lemmas \ref{lem:agd_one_step} and \ref{lem:agd_linesearch}, for each iteration $k \ge 0$ of Algorithm \ref{alg:strongly_agd} we have
{\small
\[
2 (\eta\ind{k})^2 L\ind{k} \ep^{(k+1)} +
r^{(k+1)}  \le \B\ind{k} r^{(k)} + \left[(1 - \B\ind{k}) -  \1 \2 \eta\ind{k} \right] r^{(k)}_y
+ 2 \eta\ind{k} \left[ L\ind{k} \eta\ind{k} - \1 \right] \ep^{(k)}
+ 2\B\ind{k}\eta\ind{k}\tilde{\ep}.
\]
}
\noindent Substituting in $\eta\ind{k} = \frac{1}{\sqrt{\2 L\ind{k}}} = \frac{1 - \beta\ind{k}}{\1 \2}$ and $\tilde{\ep} = 0$, this implies that 
\[
\frac{2}{\2} \ep^{(k+1)} +
r^{(k+1)} \le \B\ind{k} r^{(k)} + \frac{2}{\sqrt{\2 L\ind{k}}} \left[ \sqrt{\frac{L\ind{k}}{\2}}- \1 \right] \ep^{(k)}
= \B\ind{k} \left[r^{(k)} + \frac{2}{\2} \ep^{(k)} \right] ~.
\]
Multiplying by $\2 / 2$ and using the definition of $\B$ as $1-\1\sqrt{\ff{\2}{L\ind{k}}}$ and the fact that $0 < L\ind{k} < 2L$ yields \eqref{eq:strong_agd_one_step}.
Now, by \eqref{eq:strong_agd_one_step} and induction,
\[
\ep^{(k)} + \frac{\2}{2} r^{(k)} \le
\left(1 - \frac{\1}{\sqrt{2\kappa}} \right)^k \left[\ep^{(0)} + \frac{\2}{2} r^{(0)}\right] 
\le \exp \left( - \frac{k \1}{\sqrt{2\kappa}} \right) \left[\ep^{(0)} + \frac{\2}{2} r^{(0)}\right]  ~.
\]
Therefore, whenever $k \ge \frac{\sqrt{2\kappa}}{\1} \log^{+}\left( \frac{\ep^{(0)} + \frac{\2}{2} r^{(0)} }{\ep} \right)$ we have $\ep^{(k)} = f( x^{(k)}) - f(\qx) \le \ep$, as $r^{(k)} \ge 0$ always. By Corollary \ref{rem:distbound}, $\ff{2\ep^{(0)}}{\1} \ge \frac{\2}{2} r^{(0)}$, so it suffices to run $k \ge \ceil{\frac{\sqrt{2\k}}{\1} \log^{+} \lt \ff{3\ep^{(0)}}{\1\ep} \rt}$ iterations.
\end{proof}
}

\strongconvergeProof

Note that when $f$ is $(1, \2)$-strongly quasar-convex with $\2 > 0$, Lemma \ref{lem:strong_converge} implies that the number of \textit{iterations} Algorithm \ref{alg:strongly_agd} needs to find an $\ep$-minimizer of $f$ is of the same order as the number of iterations required by standard AGD to find an $\ep$-minimizer of a $\2$-strongly convex function. In each iteration of Algorithm \ref{alg:strongly_agd}, we compute $\alpha^{(k)}$ and then simply perform $O(1)$ vector operations to compute $y^{(k)}$, $x^{(k+1)}$, and $v^{(k+1)}$. Consequently, to obtain a complete bound on the overall complexity of Algorithm \ref{alg:strongly_agd}, it remains to bound the cost of computing $\alpha^{(k)}$, which we do using Lemma \ref{lem:linesearch}. This leads to Theorem \ref{thm:strong_runtime}.

\begin{restatable}{thm}{strongruntime}
\label{thm:strong_runtime}
If $f$ is $L$-smooth and $(\1,\2)$-strongly quasar-convex with $\1 \in (0,1]$ and $\2 > 0$, then Algorithm \ref{alg:strongly_agd} produces an $\ep$-optimal point after
$O\lt \1^{-1}\k^{1/2} \log \lt \1^{-1}\k \rt \log^{+} \lt \ff{f( x^{(0)} ) - f(\qx)}{\1\ep} \rt \rt$ function and gradient evaluations.
\end{restatable}

\newcommand{\strongruntimeProof}[1]{
\begin{proof}
Lemma \ref{lem:strong_converge} implies that $O\lt \ff{\sqrt{\k}}{\1} \log^+ \lt \ff{\ep^{(0)}}{\1\ep} \rt \rt$ iterations are needed to get an $\ep$-optimal point. Lemma \ref{lem:linesearch} implies that each iteration uses
$O \lt \log^+ \lt (1+c) \min \left\{\ff{L\norm{x-v}^2}{\tilde{\ep}}, \ff{L^3}{b^3} \right\} \rt \rt$ function and gradient evaluations. In this case,
$b = \ff{\1 \2}{2}$, $c = \sqrt{\ff{L\ind{k}}{\2}} \in \left[\sqrt{\ff{\1}{2}}, \ff{2L}{\mu}\right]$, and $\tilde{\ep} = 0$. Thus, this reduces to $O(\log^+ ( \sqrt{\k} \ff{L^3}{\1^3 \2^3} )) = O(\log(\ff{\k}{\1}))$.
So, the total number of required function and gradient evaluations is $O\lt \ff{\sqrt{\k}}{\1} \log \lt \ff{\k}{\1} \rt \log^+ \lt \ff{\ep^{(0)}}{\1\ep} \rt \rt$ as claimed.

Note that Lemma \ref{lem:strong_converge} shows that $x\ind{k}$ will be $\ep$-optimal if $k = \ceil{\frac{\sqrt{2\kappa}}{\1} \log^{+}\left( \frac{3\ep^{(0)}}{\1\ep} \right)}$, while the above argument shows that $O\lt \ff{\sqrt{\k}}{\1} \log \lt \ff{\k}{\1} \rt \log^+ \lt \ff{\ep^{(0)}}{\1\ep} \rt \rt$ function and gradient evaluations are required to compute such an $x\ind{k}$. Thus, Algorithm \ref{alg:strongly_agd} produces \textit{an} $\ep$-optimal point using at most this many evaluations; however, of course, the algorithm need not return instantly and may still continue to run if the specified number of iterations $K$ is larger. (Future iterates will also be $\ep$-optimal.)
\end{proof}
}

\strongruntimeProof

Standard AGD on $L$-smooth $\mu$-strongly-convex functions requires $O\lt \k^{1/2} \log^{+} \lt \ff{f( x^{(0)} ) - f(\qx)}{\ep} \rt \rt$ function and gradient and evaluations to find an $\ep$-optimal point \citep{Nesterov04}. Thus, as the class of $L$-smooth $(1,\mu)$-strongly quasar-convex functions contains the class of $L$-smooth $\mu$-strongly convex functions, our algorithm requires only a $O(\log(\k))$ factor extra function and gradient evaluations in the smooth strongly convex case, while also being able to efficiently minimize a much broader class of functions than standard AGD.

\subsection{Non-Strongly Quasar-Convex Minimization}
\label{sec:quasar-convex}
Now, we provide and analyze our algorithm (Algorithm \ref{alg:nonstrong_agd}) for \textit{non-strongly} quasar-convex function minimization, i.e. when $\2 = 0$. Once again, this algorithm is an instance of Algorithm \ref{alg:agd}, the general AGD framework, with a different choice of parameters. We assume $L > 0$, since otherwise quasar-convexity implies the function is constant.

\begin{algorithm}[H]
	\caption{Accelerated Non-Strongly Quasar-Convex Function Minimization}
	\label{alg:nonstrong_agd}
	\SetAlgoLined
	\SetKwInOut{Input}{input}
	\textbf{Input}: $L$-smooth $f : \R^n \rightarrow \R$ that is $\1$-quasar-convex, \\
	initial point $x^{(0)} \in \R^n$, number of iterations $K$, solution tolerance $\ep > 0$ \\
	Define $\w^{(-1)} = 1$, and $\w^{(k)} = \ff{\w^{(k-1)}}{2} \lt\sqrt{(\w^{(k-1)})^2 + 4}-\w^{(k-1)}\rt$ for $k \ge 0$ \\
	\nl Set $L\ind{-1} = \texttt{BacktrackingSearch}(f, \ep, x\ind{0}, {\small \texttt{run\_halving=True}})$ \\
	\nl \textbf{return} output of Algorithm \ref{alg:agd} on $f$ with initial point $x^{(0)}$,
	where for all $k$,
	\vskip 0ex
	$\B\ind{k} = 1$,
	$L\ind{k} = \texttt{BacktrackingSearch}(f, \quad \max\limits_{\mathclap{\scriptscriptstyle k' \in [-1,k-1]}} \,\,\, L\ind{k'}, \, x\ind{k})$, \,$\eta^{(k)} = \frac{\1}{L\ind{k} \w^{(k)}}$,
	and \vskip 0ex
	$\A^{(k)} = \texttt{BinaryLineSearch}(f, x^{(k)}, v^{(k)}, b = 0, c = \gamma(\ff{1}{\w\ind{k}}-1), \tilde{\ep} = \ff{\1\ep}{2})$
\end{algorithm}

\begin{restatable}[Non-Strongly Quasar-Convex AGD Convergence]{lem}{nonstrong}
\label{lem:nonstrong_converge}
If $f$ is $L$-smooth and $\1$-quasar-convex with respect to a minimizer $\qx$, with $\1 \in (0,1]$, then in each iteration $k \ge 0$ of Algorithm~\ref{alg:nonstrong_agd},
\begin{equation}
\label{eq:nonstrong_agd_one_step}
\ep^{(k)} \leq \f{8}{(k+2)^2} \left[\ep^{(0)} + \frac{L}{2\1^2} r^{(0)} \right] + \f{\ep}{2}~,
\end{equation}
where $\ep^{(k)} \defeq f( x^{(k)} ) - f(\qx)$ and $r^{(k)} \defeq \norm{v^{(k)} - \qx}^2$.
Therefore, if $R \ge \norm{x^{(0)} - \qx}$ and the number of iterations
$K \ge \floor{8\1^{-1}L^{1/2}R\ep^{-1/2}}$,
then the output $x^{(K)}$ satisfies ${f(x^{(K)}) \le f(\qx) + \ep}$.
\end{restatable}

Combining the bound on the number of iterations from Lemma \ref{lem:nonstrong_converge},
and the bound from Lemma \ref{lem:linesearch} on the number of function and gradient evaluations during the line search, leads to the bound in Theorem \ref{thm:nonstrong_runtime} on the total number of function and gradient evaluations required to find an $\ep$-optimal point. The proofs of Lemma \ref{lem:nonstrong_converge} and Theorem \ref{thm:nonstrong_runtime} are given in \Cref{sec:nonstrong-analysis}.

\begin{restatable}{thm}{nonstrongruntime}
\label{thm:nonstrong_runtime}
If $f$ is $L$-smooth and $\1$-quasar-convex with respect to a minimizer $\qx$, with $\1 \in (0,1]$ and $\norm{x^{(0)} - \qx} \le R$, then Algorithm \ref{alg:nonstrong_agd} produces an $\ep$-optimal point after
\newline
$O\lt \1^{-1} L^{1/2}R\ep^{-1/2} \log^{+}\lt \1^{-1} L^{1/2}R \ep^{-1/2}\rt \rt$ function and gradient evaluations.
\end{restatable}

Note that standard AGD on the class of $L$-smooth \textit{convex} functions requires $O\lt L^{1/2}R\ep^{-1/2} \rt$ function and gradient evaluations to find an $\ep$-optimal point; so, again, our algorithm requires only a logarithmic factor more evaluations than does standard AGD.
\section{Lower bounds}\label{sec:lb}
\newcommand{\fB}{\bar{f}_{T,\sigma}}

In this section, we construct lower bounds which demonstrate that the algorithms we presented in Section~\ref{sec:algorithms} obtain, up to logarithmic factors, the best possible worst-case iteration bounds for deterministic first-order minimization of quasar-convex functions. To do so, we extend the ideas from \cite{carmon2017lower}, a seminal paper which mechanized the process of constructing lower bounds. The key idea is to construct a \emph{zero-chain}, which is defined as a function $f$ for which if $x_j = 0, \forall j \ge t$ then $\frac{\partial f(x)}{\partial x_{t+1}} = 0$. On these zero-chains, one can provide lower bounds for a particular class of methods known as \emph{first-order zero-respecting (FOZR) algorithms}, which are algorithms that only query the gradient at points $x\ind{t}$ with $x\ind{t}_i \neq 0$ if there exists some $j < t$ with $\grad_i f(x\ind{j}) \neq 0$. Examples of FOZR algorithms include gradient descent, accelerated gradient descent, and nonlinear conjugate gradient \citep{cg}. It is relatively easy to form lower bounds for FOZR algorithms applied to zero-chains, because one can prove that if the initial point is $x\ind{0} = \mathbf{0}$, then $x\ind{T}$ has at most $T$ nonzeros \cite[Observation~1]{carmon2017lower}. The particular zero-chain we use to derive our lower bounds is\vspace{-1em}
$$
\vspace{-1em}\fB(x) \defeq q(x) + \sigma \sum_{i=1}^{T} \Upsilon (x_i) ~,
$$
where
\newcommand{\funcdefs}{
\begin{flalign*}
\Upsilon (\theta) &\defeq 120 \int_{1}^{\theta} \frac{t^2 (t-1) }{1 + t^2} \,dt \\
q(x) &\defeq \frac{1}{4} (x_1-1)^2 + \frac{1}{4} \sum_{i=1}^{T-1} (x_i - x_{i+1})^2.
\end{flalign*}
}
\funcdefs
This function $\fB$ is similar to the function $\bar{f}_{T,\mu,r}$ of \cite{carmon2017lower2}.
However, the lower bound proof is different because the primary challenge is to show $\fB$ is quasar-convex, rather than showing that $\norm{ \grad \fB(x) } \ge \epsilon$ for all $x$ with $x_{T} = 0$.
Our main lemma shows that this function is in fact $\frac{1}{100 T \sqrt{\sigma}}$-quasar-convex.

\begin{restatable}{lem}{lemUnscaledLB}
\label{lem:lb-unscaled}
Let $\sigma \in (0, 10^{-4}], T \in \left[\sigma^{-1/2}, \infty \right) \cap \Z$.
The function $\bar{f}_{T,\sigma}$ is $\frac{1}{100 T \sqrt{\sigma}}$-quasar-convex and $3$-smooth, with unique minimizer $x^{*} = \mathbf{1}$. Furthermore, if $x_{t} = 0$ for all $t = \ceil{T/2}, \dots, T$, then $\fB(x) - \fB(\mathbf{1}) \ge 2 T \sigma$.
\end{restatable}

The proof of Lemma~\ref{lem:lb-unscaled} appears in Appendix~\ref{sec:lem-lb-proof}.
The argument rests on showing that the quasar-convexity inequality $\frac{1}{100 T \sqrt{\sigma}}(\fB(x) - \fB(\mathbf{1})) \le \grad \fB(x)^T (x - \mathbf{1}) $ holds for all $x \in \R^T$. The nontrivial situation is when there exists some $j_1 < j_2$ such that $x_{j_1} \ge 0.9$, $x_{j_2} \le 0.1$, and $0.1 \le x_i \le 0.9$ for $i \in \{j_1 + 1, \dots, j_2 - 1\}$. In this situation, we use ideas closely related to the transition region arguments made in Lemma~3 of \cite{carmon2017lower2}. The intuition is as follows. If the gaps $x_{i+1} - x_{i}$ are large, then the convex function $q(x)$ dominates the function value and gradient of $\fB(x)$, allowing us to establish quasar-convexity. Conversely, if the $x_{i+1} - x_{i}$'s are small, then a large portion of the $x_i$'s must lie in the quasar-convex region of $\Upsilon$, and the corresponding $\Upsilon'(x_i) (x_i - 1)$ terms make $\grad \fB(x)^\top (x - \mathbf{1})$ sufficiently positive.

\def\LbSetup{Let $\ep \in (0, \infty)$, $\1 \in (0, 10^{-2}]$, $T = \ceil{10^{-3} \1^{-1} L^{1/2} R \epsilon^{-1/2}}$, and $\sigma = \frac{1}{10^4 T^2 \1^2}$, and assume $L^{1/2}R\ep^{-1/2} \ge 10^{3}$.}
\begin{restatable}{lem}{lemMainLBquasar}\label{lem:main-lb-quasar}
\LbSetup{} Consider the function
\begin{flalign}\label{eq:f-hat}
\hat{f}(x) \defeq \ff{1}{3} L R^2 T^{-1} \cdot \fB ( x T^{1/2} R^{-1} ).
\end{flalign}
This function is $L$-smooth and $\1$-quasar-convex, and its minimizer $x^*$ is unique and has $\norm{x^*} = R$.
Furthermore, if $x_{t} = 0\,\, \forall t \in \mathbb{Z} \cap [T/2,T]$, then $\hat{f}(x) - \inf_z \hat{f}(z) > \epsilon$.
\end{restatable}

The proof of Lemma~\ref{lem:main-lb-quasar} appears in Appendix~\ref{sec:lem-lb-proof}. Combining Lemma~\ref{lem:main-lb-quasar} with Observation~1 from \cite{carmon2017lower} yields a lower bound for first-order zero-respecting algorithms, and an extension of this lower bound to the class of all deterministic first-order methods. This leads to Theorem \ref{thm:main-lb-quasar}, whose proof appears in Appendix~\ref{sec:coro-lb-proof}.

\begin{restatable}{thm}{coroMainLBquasar}\label{thm:main-lb-quasar}
Let $\epsilon, R, L \in (0,\infty)$, $\1 \in (0, 1]$, and assume $L^{1/2}R\ep^{-1/2} \ge 1$. Let $\mathcal{F}$ denote the set of $L$-smooth functions that are $\1$-quasar-convex with respect to some point with Euclidean norm less than or equal to $R$. Then, given any deterministic first-order method, there exists a function $f \in \mathcal{F}$ such that the method requires at least $\Omega(\1^{-1} L^{1/2} R \epsilon^{-1/2} )$ gradient evaluations to find an $\epsilon$-optimal point of $f$.
\end{restatable}

Theorem \ref{thm:main-lb-quasar} demonstrates that the upper bound for our algorithm for quasar-convex minimization is tight within logarithmic factors. We note that by reduction (Remark~\ref{rem:sc-lower-bound}), one can prove a lower bound of $\Omega(\1^{-1}  \k^{1/2} )$ for strongly quasar-convex functions; thus, our algorithm for strongly quasar-convex minimization is also optimal within logarithmic factors.

Although the construction of our lower bounds is similar to that of \cite{carmon2017lower2}, there are important differences between our lower bounds and theirs. First, the assumptions differ significantly; we assume quasar-convexity and Lipschitz continuity of the first derivative, while \cite{carmon2017lower2} assumes Lipschitz continuity of the first \emph{three} derivatives. Next, the bounds in \citep{carmon2017lower,carmon2017lower2} apply to finding $\ep$-stationary points, rather than $\ep$-optimal points. In addition, our lower and upper bounds only differ by logarithmic factors, whereas there is a gap of $\tilde{O}(\ep^{-1/15})$ between the lower bound of $\Omega(\epsilon^{-8/5})$ given by \citep{carmon2017lower2} and the best known corresponding upper bound of $O(\epsilon^{-5/3}\log(\ep^{-1}))$ \citep{carmon2017convex}. Finally, we require $x_{t} = 0$ for all $t > T/2$ to guarantee $\hat{f}(x) - \inf_z \hat{f}(z) > \epsilon$, whereas \cite{carmon2017lower,carmon2017lower2} only need $x_{T} = 0$ to guarantee $\norm{ \grad \hat{f}(x) } > \epsilon$.

\section{Conclusion}

In this work, we introduce a generalization of star-convexity called quasar-convexity and provide
insight into the structure of quasar-convex functions.
We show how to obtain a near-optimal accelerated rate for the minimization of any smooth function in this broad class,
using a simple but novel binary search technique. In addition, we provide nearly matching theoretical lower bounds for
the performance of any first-order method on this function class.
Interesting topics for future research are to further understand the prevalence of quasar-convexity in problems of practical interest,
and to develop new accelerated methods for other structured classes of nonconvex problems.

\nocite{dang2013on}
\nocite{tyurin2017}
\nocite{joulani2017}
\newpage

\section*{Acknowledgements}
The work of Aaron Sidford was supported by NSF CAREER Award CCF-1844855.

{
\bibliographystyle{plainnat}
\small
\bibliography{acceleration_paper}

\begin{thebibliography}{62}
\providecommand{\natexlab}[1]{#1}
\providecommand{\url}[1]{\texttt{#1}}
\expandafter\ifx\csname urlstyle\endcsname\relax
  \providecommand{\doi}[1]{doi: #1}\else
  \providecommand{\doi}{doi: \begingroup \urlstyle{rm}\Url}\fi

\bibitem[Agarwal et~al.(2017)Agarwal, Allen-Zhu, Bullins, Hazan, and
  Ma]{agarwal2017finding}
Naman Agarwal, Zeyuan Allen-Zhu, Brian Bullins, Elad Hazan, and Tengyu Ma.
\newblock Finding approximate local minima faster than gradient descent.
\newblock In \emph{Symposium on Theory of Computing ({STOC})}, pages
  1195--1199. ACM, 2017.

\bibitem[Allen-Zhu(2017)]{allen2017katyusha}
Zeyuan Allen-Zhu.
\newblock Katyusha: The first direct acceleration of stochastic gradient
  methods.
\newblock \emph{Journal of Machine Learning Research}, 18\penalty0
  (1):\penalty0 8194--8244, 2017.

\bibitem[Allen-Zhu and Orecchia(2017)]{allen2014linear}
Zeyuan Allen-Zhu and Lorenzo Orecchia.
\newblock Linear coupling: An ultimate unification of gradient and mirror
  descent.
\newblock In \emph{Innovations in Theoretical Computer Science ({ITCS})}, pages
  1--22, 2017.

\bibitem[Anitescu(2000)]{anitescu}
Mihai Anitescu.
\newblock Degenerate nonlinear programming with a quadratic growth condition.
\newblock \emph{SIAM Journal on Optimization}, 10\penalty0 (4):\penalty0
  1116--1135, 2000.

\bibitem[Arrow and Enthoven(1961)]{quasi}
Kenneth Arrow and Alain Enthoven.
\newblock Quasi-concave programming.
\newblock \emph{Econometrica}, 16\penalty0 (5):\penalty0 779--800, 1961.

\bibitem[Bartlett et~al.(2019)Bartlett, Helmbold, and
  Long]{bartlett2019gradient}
Peter Bartlett, David Helmbold, and Philip Long.
\newblock Gradient descent with identity initialization efficiently learns
  positive-definite linear transformations by deep residual networks.
\newblock \emph{Neural Computation}, 31\penalty0 (3):\penalty0 477--502, 2019.

\bibitem[Beck and Teboulle(2009)]{beck2009fast}
Amir Beck and Marc Teboulle.
\newblock A fast iterative shrinkage-thresholding algorithm for linear inverse
  problems.
\newblock \emph{SIAM Journal on Imaging Sciences}, 2\penalty0 (1):\penalty0
  183--202, 2009.

\bibitem[Boyd and Vandenberghe(2004)]{BoydVa04}
Stephen Boyd and Lieven Vandenberghe.
\newblock \emph{Convex Optimization}.
\newblock Cambridge University Press, 2004.

\bibitem[Bubeck et~al.(2015)Bubeck, Lee, and Singh]{bubeck2015geometric}
S{\'e}bastien Bubeck, Yin~Tat Lee, and Mohit Singh.
\newblock A geometric alternative to {Nesterov's} accelerated gradient descent.
\newblock \emph{arXiv preprint arXiv:1506.08187}, 2015.

\bibitem[Bubeck et~al.(2019)Bubeck, Jiang, Lee, Li, and
  Sidford]{bubeck2018near}
S{\'e}bastien Bubeck, Qijia Jiang, Yin~Tat Lee, Yuanzhi Li, and Aaron Sidford.
\newblock Near-optimal method for highly smooth convex optimization.
\newblock In \emph{Conference on Learning Theory ({COLT})}, pages 492--507,
  2019.

\bibitem[Carmon et~al.(2017)Carmon, Duchi, Hinder, and
  Sidford]{carmon2017convex}
Yair Carmon, John Duchi, Oliver Hinder, and Aaron Sidford.
\newblock Convex until proven guilty: Dimension-free acceleration of gradient
  descent on non-convex functions.
\newblock In \emph{International Conference on Machine Learning (ICML)}, pages
  654--663, 2017.

\bibitem[Carmon et~al.(2018)Carmon, Duchi, Hinder, and
  Sidford]{carmon2018accelerated}
Yair Carmon, John Duchi, Oliver Hinder, and Aaron Sidford.
\newblock Accelerated methods for nonconvex optimization.
\newblock \emph{SIAM Journal on Optimization}, 28\penalty0 (2):\penalty0
  1751--1772, 2018.

\bibitem[Carmon et~al.(2019{\natexlab{a}})Carmon, Duchi, Hinder, and
  Sidford]{carmon2017lower}
Yair Carmon, John Duchi, Oliver Hinder, and Aaron Sidford.
\newblock Lower bounds for finding stationary points {I}.
\newblock \emph{Mathematical Programming}, pages 1--50, 2019{\natexlab{a}}.

\bibitem[Carmon et~al.(2019{\natexlab{b}})Carmon, Duchi, Hinder, and
  Sidford]{carmon2017lower2}
Yair Carmon, John Duchi, Oliver Hinder, and Aaron Sidford.
\newblock Lower bounds for finding stationary points {II}: First-order methods.
\newblock \emph{Mathematical Programming}, 2019{\natexlab{b}}.

\bibitem[Chang and Lin(2011)]{libsvm}
Chih-Chung Chang and Chih-Jen Lin.
\newblock {LIBSVM}: A library for support vector machines.
\newblock \emph{ACM Transactions on Intelligent Systems and Technology},
  2:\penalty0 27:1--27:27, 2011.
\newblock Software available at \url{http://www.csie.ntu.edu.tw/~cjlin/libsvm}.

\bibitem[Craven and Glover(1985)]{craven1985invex}
Bruce Craven and Barney Glover.
\newblock Invex functions and duality.
\newblock \emph{Journal of the Australian Mathematical Society}, 39\penalty0
  (1):\penalty0 1--20, 1985.

\bibitem[Dang and Lan(2015)]{dang2013on}
Cong Dang and Guanghui Lan.
\newblock On the convergence properties of non-{Euclidean} extragradient
  methods for variational inequalities with generalized monotone operators.
\newblock \emph{Computational Optimization and Applications}, 60:\penalty0
  277--310, 2015.

\bibitem[Dua and Graff(2017)]{uci}
Dheeru Dua and Casey Graff.
\newblock {UCI} machine learning repository, 2017.
\newblock \url{http://archive.ics.uci.edu/ml}.

\bibitem[Fabian et~al.(2010)Fabian, Henrion, Kruger, and Outrata]{fabian2010}
Marian Fabian, Ren{\'e} Henrion, Alexander Kruger, and Ji{\v{r}}{\'i} Outrata.
\newblock Error bounds: Necessary and sufficient conditions.
\newblock \emph{Set-Valued and Variational Analysis}, 18\penalty0 (2):\penalty0
  121--149, 2010.

\bibitem[Fercoq and Richt{\'a}rik(2015)]{fercoq2015accelerated}
Olivier Fercoq and Peter Richt{\'a}rik.
\newblock Accelerated, parallel, and proximal coordinate descent.
\newblock \emph{SIAM Journal on Optimization}, 25\penalty0 (4):\penalty0
  1997--2023, 2015.

\bibitem[Fletcher and Reeves(1964)]{cg}
Roger Fletcher and Colin Reeves.
\newblock Function minimization by conjugate gradients.
\newblock \emph{The Computer Journal}, 7\penalty0 (2):\penalty0 149--154, 1964.

\bibitem[Frostig et~al.(2015)Frostig, Ge, Kakade, and
  Sidford]{frostig2015regularizing}
Roy Frostig, Rong Ge, Sham Kakade, and Aaron Sidford.
\newblock Un-regularizing: approximate proximal point and faster stochastic
  algorithms for empirical risk minimization.
\newblock In \emph{International Conference on Machine Learning (ICML)}, pages
  2540--2548, 2015.

\bibitem[Gasnikov et~al.(2018)Gasnikov, Dvurechensky, Gorbunov, Vorontsova,
  Selikhanovych, and Uribe]{gasnikov18arxiv}
Alexander Gasnikov, Pavel Dvurechensky, Eduard Gorbunov, Evgeniya Vorontsova,
  Daniil Selikhanovych, and C{\'e}sar Uribe.
\newblock The global rate of convergence for optimal tensor methods in smooth
  convex optimization.
\newblock \emph{Computer Research and Modeling}, 10\penalty0 (6):\penalty0
  737--753, 2018.

\bibitem[Ge et~al.(2016)Ge, Lee, and Ma]{ge2016matrix}
Rong Ge, Jason Lee, and Tengyu Ma.
\newblock Matrix completion has no spurious local minimum.
\newblock In \emph{Advances in {Neural Information Processing Systems
  (NeurIPS)}}, pages 2973--2981, 2016.

\bibitem[Ghadimi and Lan(2016)]{ghadimi2016accelerated}
Saeed Ghadimi and Guanghui Lan.
\newblock Accelerated gradient methods for nonconvex nonlinear and stochastic
  programming.
\newblock \emph{Mathematical Programming}, 156\penalty0 (1-2):\penalty0 59--99,
  2016.

\bibitem[Guminov and Gasnikov(2017)]{guminov2017accelerated}
Sergey Guminov and Alexander Gasnikov.
\newblock Accelerated methods for $\alpha$-weakly-quasi-convex problems.
\newblock \emph{arXiv preprint arXiv:1710.00797}, 2017.

\bibitem[Hanzely and Richt{\'a}rik(2019)]{hanzely2018acd}
Filip Hanzely and Peter Richt{\'a}rik.
\newblock Accelerated coordinate descent with arbitrary sampling and best rates
  for minibatches.
\newblock In \emph{International Conference on Artificial Intelligence and
  Statistics ({AISTATS})}, pages 304--312, 2019.

\bibitem[Hardt and Ma(2017)]{hardt2016identity}
Moritz Hardt and Tengyu Ma.
\newblock Identity matters in deep learning.
\newblock In \emph{International Conference on Learning Representations
  ({ICLR})}, 2017.

\bibitem[Hardt et~al.(2018)Hardt, Ma, and Recht]{weakquasiconvexity}
Moritz Hardt, Tengyu Ma, and Benjamin Recht.
\newblock Gradient descent learns linear dynamical systems.
\newblock \emph{Journal of Machine Learning Research}, 19\penalty0
  (29):\penalty0 1--44, 2018.

\bibitem[Jiang et~al.(2019)Jiang, Wang, and Zhang]{jiang2018arxiv}
Bo~Jiang, Haoyue Wang, and Shuzhong Zhang.
\newblock An optimal high-order tensor method for convex optimization.
\newblock In \emph{Conference on Learning Theory ({COLT})}, pages 1799--1801,
  2019.

\bibitem[Johnson and Zhang(2013)]{JohnsonZh13}
Rie Johnson and Tong Zhang.
\newblock Accelerating stochastic gradient descent using predictive variance
  reduction.
\newblock In \emph{Advances in {Neural Information Processing Systems
  (NeurIPS)}}, pages 315--323, 2013.

\bibitem[Joulani et~al.(2017)Joulani, Gy\"{o}rgy, and
  Szepesv\'{a}ri]{joulani2017}
Pooria Joulani, Andr\'{a}s Gy\"{o}rgy, and Csaba Szepesv\'{a}ri.
\newblock A modular analysis of adaptive \mbox{(non-)convex} optimization:
  Optimism, composite objectives, and variational bounds.
\newblock In \emph{International Conference on Algorithmic Learning Theory
  ({ALT})}, 2017.

\bibitem[Karimi et~al.(2016)Karimi, Nutini, and Schmidt]{karimi2016linear}
Hamed Karimi, Julie Nutini, and Mark Schmidt.
\newblock Linear convergence of gradient and proximal-gradient methods under
  the {Polyak-{\L}ojasiewicz} condition.
\newblock In \emph{Joint European Conference on Machine Learning and Knowledge
  Discovery in Databases ({ECML-PKDD})}, pages 795--811. Springer, 2016.

\bibitem[Kleinberg et~al.(2018)Kleinberg, Li, and
  Yuan]{kleinberg2018alternative}
Robert Kleinberg, Yuanzhi Li, and Yang Yuan.
\newblock An alternative view: When does {SGD} escape local minima?
\newblock In \emph{International Conference on Machine Learning ({ICML})},
  pages 2698--2707, 2018.

\bibitem[Lee and Valiant(2016)]{lee2016optimizing}
Jasper Lee and Paul Valiant.
\newblock Optimizing star-convex functions.
\newblock In \emph{Symposium on Foundations of Computer Science (FOCS)}, pages
  603--614. IEEE, 2016.

\bibitem[Li and Lin(2015)]{li2015prox}
Huan Li and Zhouchen Lin.
\newblock Accelerated proximal gradient methods for nonconvex programming.
\newblock In \emph{Advances in {Neural Information Processing Systems
  (NeurIPS)}}, pages 379--387, 2015.

\bibitem[Li and Yuan(2017)]{li2017relu}
Yuanzhi Li and Yang Yuan.
\newblock Convergence analysis of two-layer neural networks with {ReLU}
  activation.
\newblock In \emph{Advances in {Neural Information Processing Systems
  (NeurIPS)}}, pages 597--607, 2017.

\bibitem[Lin et~al.(2015)Lin, Mairal, and Harchaoui]{lin2015universal}
Hongzhou Lin, Julien Mairal, and Zaid Harchaoui.
\newblock A universal catalyst for first-order optimization.
\newblock In \emph{Advances in {Neural Information Processing Systems
  (NeurIPS)}}, pages 3384--3392, 2015.

\bibitem[Mangasarian(1965)]{pseudo}
Olvi Mangasarian.
\newblock Pseudo-convex functions.
\newblock \emph{Journal of the Society for Industrial and Applied Mathematics
  Series A Control}, 3\penalty0 (2):\penalty0 281--290, 1965.

\bibitem[Munkres(1975)]{munkres}
James Munkres.
\newblock \emph{Topology}.
\newblock Pearson, 1975.

\bibitem[Necoara et~al.(2019)Necoara, Nesterov, and Glineur]{necoara}
Ion Necoara, Yurii Nesterov, and Fran\c{c}ois Glineur.
\newblock Linear convergence of first order methods for non-strongly convex
  optimization.
\newblock \emph{Mathematical Programming}, 175\penalty0 (1):\penalty0 69--107,
  2019.

\bibitem[Nemirovski(1982)]{nemirovski1982}
Arkadi Nemirovski.
\newblock Orth-method for smooth convex optimization.
\newblock \emph{Izvestia {AN} {SSSR}, Ser. Tekhnicheskaya Kibernetika}, 2,
  1982.

\bibitem[Nemirovski and Yudin(1983)]{nemirovsky1979}
Arkadi Nemirovski and David Yudin.
\newblock \emph{Problem Complexity and Method Efficiency in Optimization}.
\newblock Wiley, 1983.

\bibitem[Nesterov(1983)]{nesterov1983method}
Yurii Nesterov.
\newblock A method of solving a convex programming problem with convergence
  rate ${O}(1/k^2)$.
\newblock \emph{Soviet Mathematics Doklady}, 27\penalty0 (2):\penalty0
  372--376, 1983.

\bibitem[Nesterov(2004)]{Nesterov04}
Yurii Nesterov.
\newblock \emph{Introductory Lectures on Convex Optimization}.
\newblock Kluwer Academic Publishers, 2004.

\bibitem[Nesterov(2012)]{nesterov2012efficiency}
Yurii Nesterov.
\newblock Efficiency of coordinate descent methods on huge-scale optimization
  problems.
\newblock \emph{SIAM Journal on Optimization}, 22\penalty0 (2):\penalty0
  341--362, 2012.

\bibitem[Nesterov and Polyak(2006)]{nesterov2006cubic}
Yurii Nesterov and Boris Polyak.
\newblock Cubic regularization of {N}ewton method and its global performance.
\newblock \emph{Mathematical Programming}, 108\penalty0 (1):\penalty0 177--205,
  2006.

\bibitem[Nesterov et~al.(2019)Nesterov, Gasnikov, Guminov, and
  Dvurechensky]{nesterov2018primal}
Yurii Nesterov, Alexander Gasnikov, Sergey Guminov, and Pavel Dvurechensky.
\newblock Primal-dual accelerated gradient descent with line search for convex
  and nonconvex optimization problems.
\newblock \emph{Proceedings of the Russian Academy of Sciences (RAS)},
  485\penalty0 (1):\penalty0 15--18, 2019.

\bibitem[Paszke et~al.(2017)Paszke, Gross, Chintala, Chanan, Yang, DeVito, Lin,
  Desmaison, Antiga, and Lerer]{pytorch}
Adam Paszke, Sam Gross, Soumith Chintala, Gregory Chanan, Edward Yang, Zachary
  DeVito, Zeming Lin, Alban Desmaison, Luca Antiga, and Adam Lerer.
\newblock Automatic differentiation in {PyTorch}.
\newblock In \emph{Advances in Neural Information Processing Systems
  ({NeurIPS}) - Autodiff Workshop}, 2017.

\bibitem[Polyak(1963)]{polyak}
Boris Polyak.
\newblock Gradient methods for minimizing functionals.
\newblock \emph{Zhurnal Vychislitel'noi Matematiki i Matematicheskoi Fiziki},
  3\penalty0 (4):\penalty0 643--653, 1963.

\bibitem[Shalev-Shwartz and Zhang(2014)]{shalev2014accelerated}
Shai Shalev-Shwartz and Tong Zhang.
\newblock Accelerated proximal stochastic dual coordinate ascent for
  regularized loss minimization.
\newblock In \emph{International Conference on Machine Learning (ICML)}, pages
  64--72, 2014.

\bibitem[Su et~al.(2014)Su, Boyd, and Cand\`{e}s]{su2014differential}
Weijie Su, Stephen Boyd, and Emmanuel Cand\`{e}s.
\newblock A differential equation for modeling {Nesterov's} accelerated
  gradient method: Theory and insights.
\newblock In \emph{Advances in {Neural Information Processing Systems
  (NeurIPS)}}, pages 2510--2518, 2014.

\bibitem[Sutskever et~al.(2013)Sutskever, Martens, Dahl, and
  Hinton]{sutskever2013importance}
Ilya Sutskever, James Martens, George Dahl, and Geoffrey Hinton.
\newblock On the importance of initialization and momentum in deep learning.
\newblock In \emph{International Conference on Machine Learning (ICML)}, pages
  1139--1147, 2013.

\bibitem[Tyurin(2017)]{tyurin2017}
Alexander Tyurin.
\newblock Mirror version of similar triangles method for constrained
  optimization problems.
\newblock \emph{arXiv preprint arXiv:1705.09809}, 2017.

\bibitem[Van~Ngai and Penot(2007)]{ngai2007}
Huynh Van~Ngai and Jean-Paul Penot.
\newblock Approximately convex functions and approximately monotonic operators.
\newblock \emph{Nonlinear Analysis: Theory, Methods \& Applications},
  66\penalty0 (3):\penalty0 547--564, 2007.

\bibitem[Vial(1983)]{vial1983strong}
Jean-Philippe Vial.
\newblock Strong and weak convexity of sets and functions.
\newblock \emph{Mathematics of Operations Research}, 8\penalty0 (2):\penalty0
  231--259, 1983.

\bibitem[Woodworth and Srebro(2016)]{woodworth2016tight}
Blake Woodworth and Nati Srebro.
\newblock Tight complexity bounds for optimizing composite objectives.
\newblock In \emph{Advances in {Neural Information Processing Systems
  (NeurIPS)}}, pages 3639--3647, 2016.

\bibitem[Xu et~al.(2018)Xu, He, De~Sa, Mitliagkas, and R{\'e}]{xu2018pca}
Peng Xu, Bryan He, Christopher De~Sa, Ioannis Mitliagkas, and Christopher
  R{\'e}.
\newblock Accelerated stochastic power iteration.
\newblock In \emph{International Conference on Artificial Intelligence and
  Statistics ({AISTATS})}, pages 58--67, 2018.

\bibitem[Zhang and Yin(2013)]{zhang2013gradient}
Hui Zhang and Wotao Yin.
\newblock Gradient methods for convex minimization: better rates under weaker
  conditions.
\newblock \emph{arXiv preprint arXiv:1303.4645}, 2013.

\bibitem[Zhang et~al.(2019)Zhang, Sra, and Jadbabaie]{sra2019}
Jingzhao Zhang, Suvrit Sra, and Ali Jadbabaie.
\newblock Acceleration in first order quasi-strongly convex optimization by
  {ODE} discretization.
\newblock In \emph{{IEEE} Conference on Decision and Control (CDC)}, 2019.

\bibitem[Zhou et~al.(2019)Zhou, Yang, Zhang, Liang, and Tarokh]{zhou2019sgd}
Yi~Zhou, Junjie Yang, Huishuai Zhang, Yingbin Liang, and Vahid Tarokh.
\newblock {SGD} converges to global minimum in deep learning via star-convex
  path.
\newblock In \emph{International Conference on Learning Representations
  ({ICLR})}, 2019.

\bibitem[Zhou et~al.(2017)Zhou, Mertikopoulos, Bambos, Boyd, and
  Glynn]{zhou2017vc}
Zhengyuan Zhou, Panayotis Mertikopoulos, Nicholas Bambos, Stephen Boyd, and
  Peter Glynn.
\newblock Stochastic mirror descent in variationally coherent optimization
  problems.
\newblock In \emph{Advances in {Neural Information Processing Systems
  (NeurIPS)}}, pages 7040--7049, 2017.

\end{thebibliography}
}

\newpage
\appendix

\section{Extended Related Work}
\subsection{Related Function Classes to Quasar-Convexity}
\label{sec:related-classes}
In this section, we provide a brief taxonomy of related conditions (relaxations of convexity or strong convexity), and describe how they relate to quasar-convexity. For simplicity, here we assume $f$ is $L$-smooth with domain $\mathcal{X} = \R^n$. We denote the minimum of $f$ by $f^*$ and the set of minimizers of $f$ by $\mathcal{X}^*$; when $\mathcal{X}^*$ consists of a single point, we denote the point by $x^*$.

First, we review the definitions of quasar-convexity, star-convexity, and convexity. Recall that (strong) quasar-convexity is a generalization of (strong) star-convexity, which itself generalizes (strong) convexity.
\begin{itemize}
\item \emph{(Strong) quasar-convexity} (with parameters $\1 \in (0, 1]$,\,\, $\mu \ge 0$): for some $x^* \in \mathcal{X}^*$, \linebreak $f(x^*) \ge f(x) + \ff{1}{\1}\G f(x)^\top (x^*-x) + \ff{\2}{2} \norm{x^*-x}^2$ for all $x \in \X$.
\begin{itemize}
\item When $\mu = 0$, this is merely referred to as \emph{quasar-convexity}, which is also known as \emph{weak quasi-convexity} \citep{weakquasiconvexity}.
\item When $\mu > 0$, $f$ has exactly one minimizer $x^*$.
\end{itemize}

\item \emph{(Strong) star-convexity} (with parameter $\mu \ge 0$): for some $x^* \in \mathcal{X}^*$, $f(x^*) \ge f(x) + {\G f(x)^\top (x^*-x)} \\ + \ff{\2}{2} \norm{x^*-x}^2$ for all $x \in \X$.
\begin{itemize}
\item When $\mu = 0$, this is merely referred to as \emph{star-convexity}.
\item When $\mu > 0$, this is also known as \emph{quasi-strong convexity} \citep{necoara}.
\item When $\mu = 0$, $f$ may not have a unique minimizer; some authors require the condition to hold for \emph{all} $x^* \in \mathcal{X}^*$ \citep{nesterov2006cubic}, while others only require it for \emph{some} $x^* \in \mathcal{X}^*$ \citep{lee2016optimizing}; we use the latter definition.
\item When $\mu > 0$, $f$ has exactly one minimizer $x^*$.
\end{itemize}

\item \emph{(Strong) convexity} (with parameter $\mu \ge 0$): $f(y) \ge f(x) + \G f(x)^\top(y-x) + \ff{\mu}{2}\norm{y-x}^2$ for all $x,y \in \X$.
\begin{itemize}
\item When $\mu = 0$, this is merely referred to as \emph{convexity}.
\end{itemize}
\end{itemize}

Next, we enumerate some other generalizations of strong convexity from the literature, and state whether they generalize quasar-convexity, are generalized by quasar-convexity, or neither.
\begin{itemize}
\item \emph{Weak convexity} \citep{vial1983strong} (with parameter $\mu > 0$): $f(y) \ge f(x) + \G f(x)^\top(y-x) - \ff{\mu}{2}\norm{y-x}^2$ for all $x,y \in \X$.
\begin{itemize}
\item Neither implies nor is implied by quasar-convexity.
\end{itemize}

\item \emph{Quadratic growth condition} (with parameter $\mu > 0$) \citep{anitescu}: $f(x) \ge f(x^*) + \ff{\mu}{2} \norm{x^*-x}^2$ for all $x \in \X$.
\begin{itemize}
\item Neither implies nor is implied by quasar-convexity.
\end{itemize}

\item \emph{Restricted secant condition} (with parameter $\mu > 0$) \citep{zhang2013gradient}: $0 \ge \G f(x)^\top(x^*- x) + \ff{\mu}{2}\norm{x^*-x}^2$ for all $x \in \X$.
\begin{itemize}
\item Implied by $(\1,\ff{\2}{\1})$-strong quasar-convexity (for any choice of $\1 \in (0, 1]$).
\end{itemize}

\item \emph{One-point strong convexity} (with parameter $\mu > 0$) \citep{li2017relu}: for some $y \in \X$, $0 \ge \G f(x)^\top(y - x) + \ff{\mu}{2}\norm{y-x}^2$ for all $x \in \X$.
\begin{itemize}
\item This is a generalization of the restricted secant property (which is one-point strong convexity in the special case $y = x^*$), and is therefore likewise implied by strong quasar-convexity.
\end{itemize}

\item \emph{Variational coherence} \citep{zhou2017vc}: $0 \ge \G f(x)^\top(x^*- x)$ for all $x \in \X$, $x^* \in \X^*$, with equality iff $x \in \X^*$.
\begin{itemize}
\item Implied by strong quasar-convexity (for any $\2 > 0$ and $\1 \in (0, 1]$). The closely related weaker condition ``for some $x^* \in \X^*$, $0 \ge \G f(x)^\top (x^* - x)$ for all $x \in \X$, with equality iff $x \in \X^*$'' is implied by quasar-convexity (for any $\2 \ge 0, \1 \in (0, 1]$). In fact, the set of functions satisfying this condition is the limiting set of the class of $\1$-quasar-convex functions as $\1 \ra 0$; this is the set of differentiable functions with star-convex sublevel sets.
\end{itemize}

\item \emph{Polyak-{\L}ojasiewicz condition} \citep{polyak} (with parameter $\mu > 0$): $\ff{1}{2} \norm{\G f(x)}^2 \ge \mu (f(x) - f_*)$ for all $x \in \X$.
\begin{itemize}
\item This is implied by the restricted secant property \citep{karimi2016linear}, and therefore by strong quasar-convexity.
\end{itemize}

\item \emph{Quasiconvexity} \citep{quasi}: $f(\M x + (1-\M)y) \le \max\{f(x), f(y)\}$ for all $x,y \in \X$ and $\M \in [0,1]$.
\begin{itemize}
\item Neither implies nor is implied by quasar-convexity. (However, the set of differentiable quasiconvex functions is contained in the limiting  set of the the class of $\1$-quasar-convex functions as $\1 \ra 0$.)
\end{itemize}

\item \emph{Pseudoconvexity} \citep{pseudo}: $f(y) \ge f(x)$ for all $x,y \in \X$ such that $\G f(x) \cdot (y-x) \ge 0$.
\begin{itemize}
\item Neither implies nor is implied by quasar-convexity.
\end{itemize}

\item \emph{Invexity} \citep{craven1985invex}: $x \in \X^*$ for all $x \in \X$ such that $\G f(x) = \mathbf{0}$.
\begin{itemize}
\item Implied by quasar-convexity (for any $\2 \ge 0, \1 \in (0, 1]$).
\end{itemize}

\end{itemize}

\subsection{Comparison to Nesterov et al. (2018)}
\label{app:nesterov}
As discussed in Section \ref{sec:star-related}, both our method (Algorithm \ref{alg:nonstrong_agd}) and Algorithm 2 of \cite{nesterov2018primal} minimize $\gamma$-quasar-convex functions at accelerated rates. Compared to the method in \cite{nesterov2018primal}, our algorithm attains a better runtime bound by a factor of $\gamma^{1/2}$, and does not require the optimal function value to be known. Meanwhile, the method of \cite{nesterov2018primal} can handle functions that are $L$-smooth (where $L$ need not be known) with respect to more general norms (not necessarily Euclidean). The algorithms themselves are quite similar (in the case of $\gamma$-quasar-convex functions that are $L$-smooth with respect to the Euclidean norm), as both are generalizations of standard AGD. However, unlike \cite{nesterov2018primal}, our algorithm does not rely on any restarts. Furthermore, while both algorithms conduct a one-dimensional minimization between current iterates $x\ind{k}$ and $v\ind{k}$ in each iteration, we use the insight that this can actually be done via a carefully implemented \emph{binary} search to do this efficiently.

\section{Numerical Experiments}
\label{sec:experiments}
We first consider optimizing a ``hard function''---an example of the type of function used to construct the lower bound in Theorem 2. This function class is parameterized by $\sigma$ and the dimension $T$; we denote these functions by $\bar{f}_{T,\sigma}$ (see Appendix \ref{sec:lb} for the definition). We compare our method to other commonly used first-order methods: gradient descent (GD), [standard] accelerated gradient descent (AGD), nonlinear conjugate gradients (CG), and the limited-memory BFGS (L-BFGS) algorithm. (Out of all these algorithms, only our method and GD offer theoretical guarantees for quasar-convex function minimization.)

We next evaluate our algorithm on real-world tasks: we use our algorithm to train a support vector machine (SVM) on the nine LIBSVM UCI binary classification datasets \citep{libsvm} (which are derived from the UCI ``Adult'' datasets \citep{uci}). The SVM loss function we use is a smoothed version of the hinge loss: $f(x) = \sum_{i=1}^n \phi_\alpha(1-b_ia_i^\top x)$, where $a_i\in \R^d, b_i = \pm 1$ are given by the training data (the $a_i$'s are the covariates and the $b_i$'s are the labels), and $\phi_{\alpha}(t) = 0$ for $t \le 0$, $\ff{t^2}{2}$ for $t \in [0,1]$, and $\ff{t^\A-1}{\A}+\ff{1}{2}$ for $t \ge 1$. When $\A=1$, $\phi_\A = \ff{t^2}{2}$ for all $t \ge 0$, and thus $\phi_\A$ and $f$ are convex. For all $\A \in (0, 1]$, $\phi_\A$ is smooth and $\A$-quasar-convex. Line searches for this function are inexpensive, as the quantities $b_ia_i^\top x$ need only be calculated once per outer loop iteration. Results are given in Table \ref{tab:expts}.

Finally, we evaluate on the problem of learning linear dynamical systems, which was shown to be quasar-convex (under certain assumptions) by \cite{weakquasiconvexity}. In this problem, we are given observations $\{(x_t, y_t)\}_{i=1}^T$ generated by the time-invariant linear system $h_{t+1} = Ah_t + Bx_t; y_t = Ch_t + Dx_t$, where $x_t, y_t \in \R$; $h_t \in \R^n$ is the \textit{hidden state} at time $t$; and $\Theta = (A,B,C,D)$ are the (unknown) parameters of the system. Informally, we seek to learn $\hat{\Theta}$ to minimize $\ff{1}{T}\sum_{i=1}^T (y_t - \hat{y}_t)^2$, where $\hat{h}_{t+1} = \hat{A}\hat{h}_t + \hat{B}x_t; \hat{y}_t = \hat{C}\hat{h}_t + \hat{D}x_t$, and $\hat{h}_0 = 0$. When parameterized in \textit{controllable canonical form}, this problem was shown to be quasar-convex on a subset of the domain near the optimum in \citep{weakquasiconvexity}. We describe this problem and our experimental approach in more detail in Appendix~\ref{sec:expt-details}. Representative plots are given in Figure~\ref{fig:lds}. Despite the nonconvexity, AGD performs quite well on this problem. Nonetheless, we observe that our method is competitive with AGD in terms of \textit{iteration} count; we use more \textit{function evaluations} due to the line search, but gradient evaluations are about twice as expensive in this setting, and the line search can also be parallelized. The design of better heuristics to speed up our method
is an interesting question for future empirical investigation.

In all experiments, we use adaptive step sizes for our method, as well as GD and AGD, as in practice $L$ may not be known \emph{a priori}. We do not use an initial guess for the line search.

\aboverulesep=0ex
\belowrulesep=0ex
\begin{table}[h]
	\tiny
	\setlength{\tabcolsep}{5pt}
	\centerfloat
		\begin{tabular}{| c | c | c | c | c | c | c |}
			\toprule
			$\downarrow$ Function~~~/~~~Algorithm $\rightarrow$ & Ours (Alg. 4) & Gradient Descent (GD) & Standard AGD & Nonlinear CG & L-BFGS \\
			\midrule
			$\bar{f}_{T,\sigma}$ ($\sigma = 10^{-1}, T = 10^2; \ep = 10^{-4}$) & 422;\, 1,451 & 336; 738 & 272;\, 869 & 312;\, 1,599 & 354;\, 1,778 \\
			$\bar{f}_{T,\sigma}$: ($\sigma = 10^{-4}, T = 10^3; \ep = 10^{-6}$) & 12,057;\, 55,357 & 18,607;\, 40,684 & 3,891;\, 12,399 & 1,251;\, 3,647 & 1,093;\, 6,554 \\
			$\bar{f}_{T,\sigma}$: ($\sigma = 10^{-6}, T = 10^3; \ep = 10^{-8}$) & 17,135;\, 167,447 & 275,572;\, 602,561 & 55,623;\, 177,247 & 10,007;\, 30,023 & 2,079;\, 12,476 \\
			\hline
			LIBSVM UCI ($\A = 1$; $\ep = 10^{-4}$) & 0.92;\, +0.017\% & 4.65;\, +0.036\% & --- & 0.46;\, +0.001\% & 0.29;\, +0.010\% \\
			LIBSVM UCI ($\A = 0.5$; $\ep = 10^{-4}$) & 1.32;\, +0.016\% & 4.78;\, +0.033\% & --- & 0.48;\, +0.001\% & 0.30;\, +0.011\% \\
			\bottomrule
		\end{tabular}
	\caption{
	Experimental results. The stopping criterion used is $\norm{\nabla f(x)}_\infty \le \ep$.
	For $\bar{f}_{T,\sigma}$ we report \textit{(\# iterations; \# function+gradient evals)}; the initial point is $x_0 = \vec{0}$.
	For LIBSVM UCI datasets, we report: the \textit{ratio} of the total number of iterations required compared to standard AGD, averaged over all 9 datasets and 3 different random initializations (shared across algorithms) per dataset, and the average final \textit{test classification accuracy difference} compared to AGD.}
	\label{tab:expts}
\end{table}

\begin{figure}[ht]
\center
\includegraphics[scale=0.375]{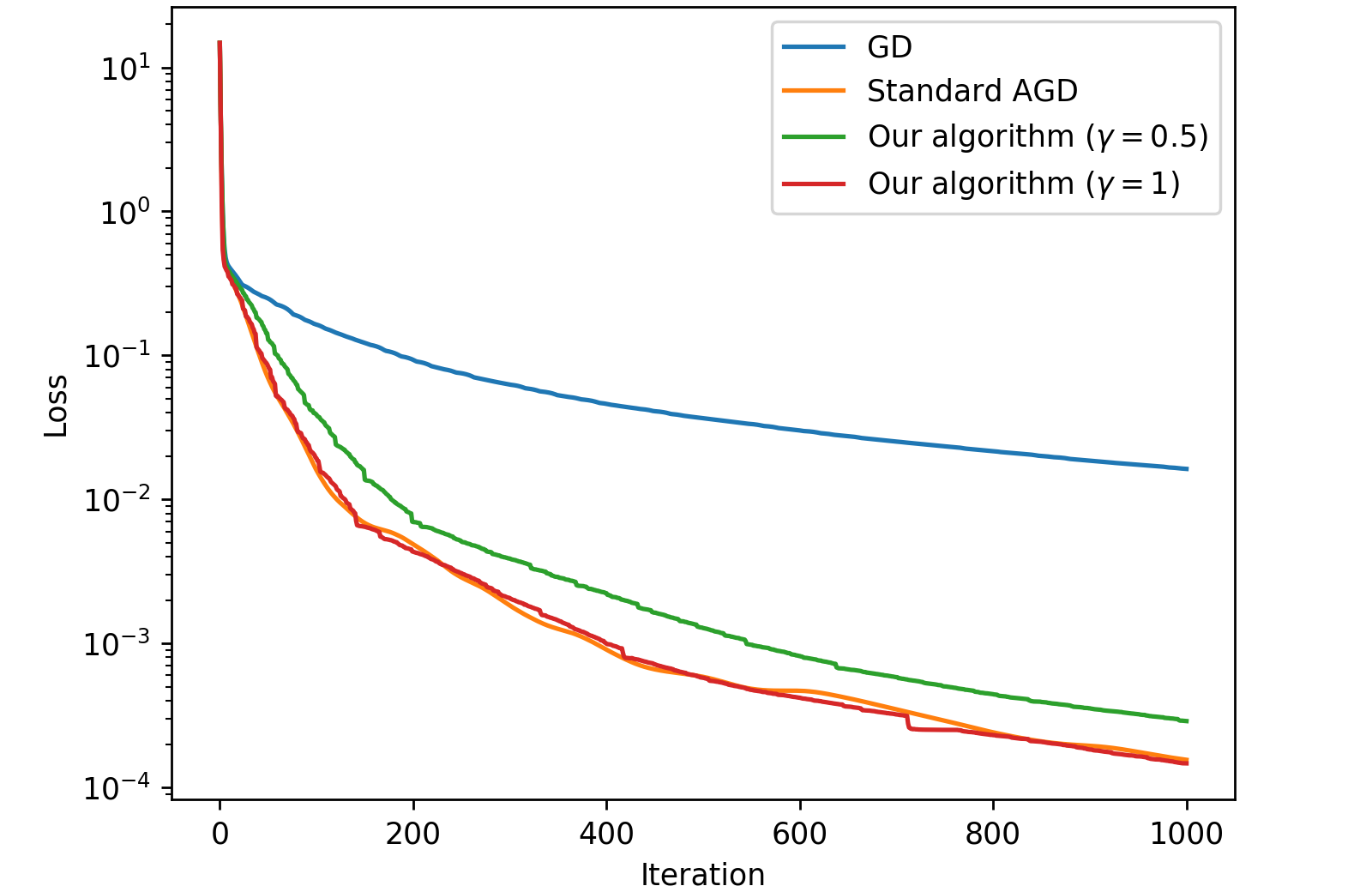}
\includegraphics[scale=0.375]{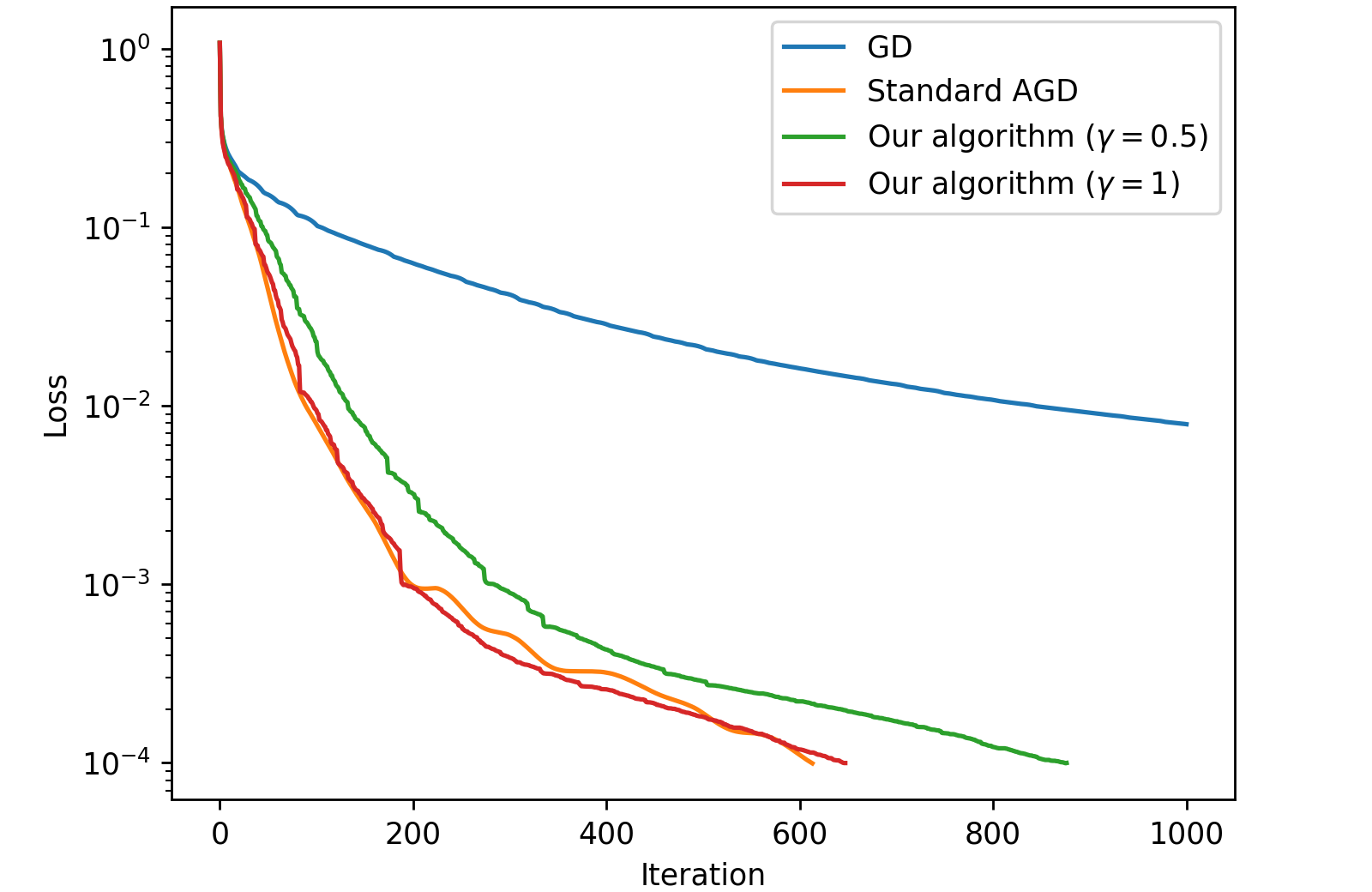}
\caption{Results on learning linear dynamical systems, for two different problem instances. We evaluate our method with $\gamma = \{0.5, 1\}$, and compare to GD and AGD. We run until the loss is $< 10^{-4}$ or 1000 iterations have been reached. Our method uses $\approx$4x as many total evaluations as AGD; for instance, in the first setting all methods run for 1000 iterations and use 2195, 3195, 13562 and 14626 total evaluations respectively (out of which 1000 are gradient evaluations).}
\label{fig:lds}
\end{figure}

\subsection{Additional Experimental Details}
\label{sec:expt-details}
We implement our algorithm, as well as AGD and GD, in Julia and Python.\footnote{Code for our implementation and experiments is available at \url{https://github.com/nimz/quasar-convex-acceleration}.}
We run our experiments on learning linear dynamical systems (LDS) using the PyTorch framework \citep{pytorch}. We generate the true parameters and the dynamical model inputs the same way as in \citep{weakquasiconvexity}, using the same parameters $n = 20, T = 500$. However, differently from this paper, we do not generate fresh sequences $\{(x_t, y_t)\}$ at each iteration, but instead generate 100 sequences at the beginning which are used throughout (so, it is no longer a stochastic optimization problem). As in \citep{weakquasiconvexity}, we actually minimize the loss $\ff{1}{|\mathcal{B}|} \sum_{(x,y) \in \mathcal{B}} \lt \ff{1}{T-T_1}\sum_{i>T_1} (y_t - \hat{y}_t)^2\rt$, where the outer summation is over the batch $\mathcal{B}$ of 100 sequences and the inner summation starts at time $T_1 \defeq T/4$, to mitigate the fact that the initial hidden state is not known. In addition, we generate the initial point $(\hat{A}_0, \hat{C}_0, \hat{D}_0)$ by perturbing the true dynamical system parameters $(A,C,D)$ with random noise; we additionally ensure that the spectral radius of $\hat{A}_0$ remains less than 1.

The quasar-convexity parameter $\gamma$ derived in \citep{weakquasiconvexity} for the LDS objective is defined as the supremum of the real part of a ratio of two degree-$n$ univariate polynomials over the complex unit circle. Therefore, it is difficult to calculate in practice. We instead simply evaluate different values of $\gamma$ in our experiments; we find that, while the choice of $\gamma$ does affect performance somewhat, our method does not break down even if the ``wrong'' choice is used.

\cite{weakquasiconvexity} presented two better-performing alternatives to fixed-stepsize SGD: SGD with gradient clipping or projected SGD. By contrast, as we use an adaptive step size, there is no need to clip gradients; in addition, we find projection to be unnecessary as the initial iterate we generate already has $\rho(\hat{A}_0) < 1$ by construction.

In the LDS experiments, we use forward difference to approximate the 1D gradients in the line search, since full gradient evaluations require backpropagation and are thus more expensive than function evaluations in this case; we do not find this to incur significant numerical error.

For the adaptive step sizes, we use a standard scheme in which the step size at iteration $k > 0$ [which we denote $\ff{1}{L^{(k)}}$] is initialized to the previous step size $\ff{1}{L^{(k-1)}}$ times a fixed value $\zeta_1 \ge 1$, and then multiplied by a fixed value $\zeta_2 \in (0, 1)$ until it is small enough so that the function value decrease is sufficient,\footnote{Specifically, for GD, we decrease the step size $\ff{1}{L^{(k)}}$ until the criterion $f(x^{(k+1)}) \le f(x^{(k)}) - \ff{1}{2 L^{(k)}} ||\G f(x^{(k)})||^2$ is satisfied; for AGD and our method, the criterion is $f(x^{(k+1)}) \le f(y^{(k)}) - \ff{1}{2 L^{(k)}} ||\G f(y^{(k)}||^2$. These criteria are guaranteed to hold when $L^{(k)} \ge L$.} where $\zeta_1, \zeta_2$ are constant hyperparameters. (This slightly generalizes Algorithm \ref{alg:backtracking}, which simply sets $\zeta_2 = 1/2$.) In all experiments for GD, AGD, and our method, we used $\zeta_1 = 1.1, \zeta_2 = 0.6$, and $L^{(0)} = 1$ (these values were only coarsely tuned; the algorithms are fairly insensitive to them when reasonable settings are used).

\section{Algorithm Analysis}

\label{sec:analysis-lemmas}
Here, we provide omitted proofs and details for Sections \ref{sec:acceleration_framework}-\ref{sec:algorithms}.

\subsection{Backtracking Step Size Search Analysis}\label{sec:onestep-proofs}

In Algorithm~\ref{alg:backtracking} (analyzed in Lemma~\ref{lem:stepsize}), we show how to efficiently compute an $L\ind{k}$ such that ${f(y\ind{k}-\ff{1}{L\ind{k}}\G f(y\ind{k}))} \le f(y\ind{k}) - \ff{1}{2L\ind{k}} \norm{\G f(y\ind{k})}^2$ holds in Line 3 of Algorithm~\ref{alg:agd}, even when the true Lipschitz constant $L$ is unknown. This is done using standard backtracking line search; we provide the details of the algorithm and analysis for completeness (Algorithm \ref{alg:backtracking}). [Note that $\texttt{run\_halving}$ means we halve $\hat{L}$ repeatedly as long as the descent inequality is satisfied---corresponding to \emph{doubling} the step size each time.]

\setcounter{AlgoLine}{0}
\begin{algorithm}[H]
\setstretch{0.95}
\SetAlgoLined
\caption{\texttt{BacktrackingSearch}($f, \zeta, x, \texttt{run\_halving}=\texttt{False})$}
\label{alg:backtracking}
\SetKwInOut{Input}{input}
\textit{Assumptions}: $f : \R^n \ra \R$ is $L$-smooth; $x \in \R^n$; $\zeta > 0$ and ($\zeta < 2L$ or \texttt{run\_halving=False})
\vskip 0ex
\nl $\hat{L} \leftarrow \zeta$ \\
\nl \If{\texttt{run\_halving}}{
 \nl \While{$f(x - \ff{1}{\hat{L}} \G f(x)) \le f(x) - \ff{1}{2\hat{L}} \norm{\G f(x)}^2$}{
 \nl $\hat{L} \leftarrow \hat{L}/2$
 }
 \nl $\hat{L} \leftarrow 2\hat{L}$
}
\nl \While{$f(x - \ff{1}{\hat{L}} \G f(x)) > f(x) - \ff{1}{2\hat{L}} \norm{\G f(x)}^2$}{
\nl $\hat{L} \leftarrow 2\hat{L}$
}
\nl \Return{$\hat{L}$}
\BlankLine
\end{algorithm}

\begin{lem}
\label{lem:stepsize}
Let $L$ be the minimum real number such that $f: \R^n \ra \R$ is $L$-smooth. Then, Algorithm~\ref{alg:backtracking} computes an ``inverse step size'' $\hat{L}$ such that $f\lt x - \ff{1}{\hat{L}} \G f(x)\rt \le f(x) - \ff{1}{2\hat{L}} \norm{\G f(x)}^2$. If $\texttt{run\_halving}$ is $\texttt{False}$, $\hat{L} \in [\zeta, 2L)$ and Algorithm~\ref{alg:backtracking} uses at most $\ceil{\log_2^+ \ff{L}{\zeta}}+3$ function and gradient evaluations. If $\texttt{run\_halving}$ is $\texttt{True}$, $\hat{L} \in (0, 2L)$ and Algorithm~\ref{alg:backtracking} uses at most $\left\lceil \log_2^+ \max\left\{ \ff{L}{\zeta}, \ff{\zeta}{L} \right\} \right\rceil +3$ evaluations.
\end{lem}
\begin{proof}
We use the elementary fact that if $f$ is $L$-smooth, then for any $x \in \R^n$ if we define $y \defeq x - \frac{1}{{L}} \grad f(x)$, then $f(y) \le f(x) - \frac{1}{2{L}} \norm{\grad f(x)}^2$ (for example, see \citep{Nesterov04} for proof).

In Algorithm~\ref{alg:backtracking}, we use $\zeta$ as the initial guess for $\hat{L}$, and when $\texttt{run\_halving}$ is $\texttt{False}$
simply double $\hat{L}$ until the desired condition holds. Note that since an $L$-smooth function is also $L'$-smooth for any $L' \ge L$, the desired condition holds for any $L' \ge L$; we will use $L$ to denote the minimum value of $L'$ such that $f$ is $L'$-smooth.
We need to double $\hat{L}$ at most $\ceil{\log_2^+(L/\zeta)}$ times until it is greater than or equal to $L$, so the while loop condition is checked at most $\ceil{\log_2^+(L/\zeta)}+1$ times. Since we stop increasing $\hat{L}$ when the desired condition holds, and it holds whenever $\hat{L} \ge L$, the final value of $\hat{L}$ will be less than $2L$. Each check of the while loop condition requires computing $f\lt x - \ff{1}{\hat{L}} \G f(x)\rt$ for the current value of $\hat{L}$; we also need to compute $f(x)$ and $\G f(x)$ at the beginning.

When \texttt{run\_halving} is \texttt{True} (branch in Line 2), we also halve the initial guess $\hat{L}$ until the condition no longer holds,
then double this value to recover the last value of $\hat{L}$ for which the condition holds. Similarly, at most $\ceil{\log_2^+ \ff{\zeta}{L}}$ iterations of this halving procedure are required. Finally, notice that if the while loop condition in Line 3 ever evaluates to \texttt{True}, then the value $\hat{L}$ at the end of Line 5 will satisfy $f(x-\ff{1}{\hat{L}}\G f(x)) \le f(x) - \ff{1}{2\hat{L}}\norm{\G f(x)}^2$, meaning that the while loop on Line 6 will immediately terminate.
\end{proof}

Note that the constant 2 used in Algorithm~\ref{alg:backtracking} is arbitrary; we can use any constant larger than 1 to multiplicatively increase $\hat{L}$ each time, which merely changes both the runtime and the final upper bound on $\hat{L}$ by a constant factor. The term ``backtracking'' is used because increasing $\hat{L}$ corresponds to decreasing the ``step size.''

\subsection{Analysis of Algorithm~\ref{alg:linesearch}}
\label{sec:linesearch-proofs}

We first present a simple fact that is useful in our proofs of Lemmas~\ref{lem:ak_existence} and \ref{lem:linesearch}.

\begin{fact}
\label{fact:minimizer}
Suppose that $a <b$, $g : \R \ra \R$ is differentiable, and that $g(a) \ge g(b)$.
Then, there is a $c \in (a,b]$ such that $g(c) \le g(b)$ and either $g'(c) = 0$, or $c = b$ and $g'(c) \le 0$. 
\end{fact}
\begin{proof}
If $g'(b) \le 0$, the claim is trivially true. If not, then $g'(b) > 0$, so the minimum value of $g$ on $[a,b]$ is strictly less than $g(b)$
(and therefore strictly less than $g(a)$ as well).
By continuity of $g$ and the extreme value theorem, $g$ must therefore attain its minimum on $[a,b]$ at some point in $c \in (a,b)$.
By differentiability of $g$ and the fact that $c$ minimizes $g$, we then have $g'(c) = 0$.
\end{proof}
\begin{fact}
\label{fact:line_lipschitz}
Suppose $f$ is $L$-smooth. Define $g(\A) \defeq f(\A x + (1-\A)v)$; then, $g$ is $L\norm{x-v}^2$-smooth.
\end{fact}
\begin{proof}
By $L$-smoothness of $f$, $\norm{\G f(x) - \G f(y)} \le L\norm{x-y}$ for all $x,y$. So,
\begin{flalign*}
\norm{\G f(y(\A_1)) - \G f(y(\A_2))} &= \norm{\G f(\A_1 x + (1-\A_1)v) - \G f(\A_2 x + (1-\A_2)v)} \\
&\le L\norm{(\A_1-\A_2)x-(\A_1-\A_2)v} = L|\A_1-\A_2|\norm{x-v}.
\end{flalign*}
By definition of $g$ and the Cauchy-Schwarz inequality,
\begin{flalign*}
|g'(\A_1) - g'(\A_2)| &= |\G f(y(\A_1))^\top(x-v) - \G f(y(\A_2))^\top(x-v)| \\ &\le \norm{\G f(y(\A_1)) - \G f(y(\A_2))}\norm{x-v},
\end{flalign*}
so $|g'(\A_1)-g'(\A_2)| \le L\norm{x-v}^2|\A_1-\A_2|$ as desired.
\end{proof}

\noindent Using Lemma~\ref{lem:ak_existence} and Fact~\ref{fact:line_lipschitz}, we prove Lemma~\ref{lem:linesearch}.

\linesearch*

\begin{proof}
\renewcommand{\lo}{\textbf{lo}}
\renewcommand{\hi}{\textbf{hi}}
Define $\Ll \defeq L\norm{x-v}^2$; by Fact \ref{fact:line_lipschitz}, $g$ is $\Ll$-smooth.
Note that if $p+\epp \ge \Ll$ and $g'(\A) = 0$, then by $\Ll$-smoothness of $g$, we have ${g'(1) \le \epp + p}$. So, it must be the case that $p+\epp < \Ll$ if Algorithm~\ref{alg:linesearch} enters the binary search phase.
Thus, if $g'(1) > \tilde{\ep}+p$,
then by Lemma~\ref{lem:stepsize} and the definition of $\tau$ we have $g'(\lhi) > 0$ and $g(\lhi)-g(1) \le -\ff{(\epp+p)^2}{4\Ll}$.
Recall that the loop termination condition in Algorithm~\ref{alg:linesearch} is $\A (g'(\A) - \A p) \le c(g(1) - g(\A)) + \epp$.
First, we claim that the invariants $g(\lo) > g(\lhi)$, $g(\hi) \le g(\lhi)$, and $g'(\hi) > \tilde{\ep}$ hold at the start of every loop iteration.
This is true at the beginning of the loop, since otherwise the algorithm would return before entering it.
In the loop body, $\hi$ is only ever set to a new value $\A$ if $g(\A) \le g(\lhi)$.
If the loop does not subsequently terminate, this also implies $g'(\A) > \epp$ since then
\[\A (g'(\A) - \A p) > c(g(1) - g(\A)) + \epp \ge c(g(1) - g(\lhi)) + \epp \ge \epp~.\]
Similarly, $\lo$ is only ever set to a new value $\A$ if $g(\A) > g(\lhi)$.
Thus, these invariants indeed hold at the start of each loop iteration.

Now, suppose $\A = (\lo + \hi)/2$ does not satisfy the termination condition.
If $g(\A) \le g(\lhi)$, this implies $g'(\A) > \tilde{\ep}$.
As $g(\lo) > g(\lhi) \ge g(\A)$, by Fact \ref{fact:minimizer}, there must be an $\hat{\A} \in (\lo, \A)$ with $g'(\hat{\A}) = 0$ and $g(\hat{\A}) \le g(\lhi)$ [and thus satisfying the termination condition].
The algorithm sets $\hi$ to $\A$, which will keep $\hat{\A}$ in the new search interval $[\lo, \A]$.

Similarly, if $g(\A) > g(\lhi)$, then since $g(\lhi) \ge g(\hi)$ and $g'(\hi) > 0$, there must be an $\hat{\A} \in (\A, \hi)$ with $g'(\hat{\A}) = 0$ and $g(\hat{\A}) \le g(\lhi)$ [and thus satisfying the termination condition], by applying Fact \ref{fact:minimizer}.
The algorithm sets $\lo$ to $\A$, which will keep $\hat{\A}$ in the search interval.
Thus, there is always at least one point $\hat{\A} \in [\lo, \hi]$ satisfying the termination condition.

In addition, note that if an interval $[z_1, z_2] \subseteq [0,1]$ of points satisfies the termination condition, then at every loop iteration, either the entire interval lies in $[\lo, \hi]$ or none of the interval does, i.e. either $[z_1,z_2] \subseteq [\lo, \hi]$ or $[z_1,z_2] \cap [\lo,\hi] = \emptyset$. The reason is that if a point $\A$ satisfies the termination condition we terminate immediately. If not, then $\A$ is not in an interval of points satisfying the termination condition, so either $z_2 < \A$ or $z_1 > \A$. Thus, all intervals of points satisfying the termination condition either disjointly lie in the set of points that remain in our search interval, or the set of points we throw away (i.e. an interval of satisfying points never gets split).

Suppose that $\A \in [0, \tau]$, $g'(\A) = 0$, and $g(\A) \le g(\lhi)$.
By $\Ll$-Lipschitz continuity of $g'$, we have that for all $t$, $|g'(t)| = |g'(t)-g'(\A)| \le \Ll |t-\A|$ and $g(t) - g(1) \le g(t) - g(\lhi) \le g(t)-g(\A) \le \ff{\Ll}{2}(t-\A)^2$. So, for all $t \in [\A/2, \tau]$, 
\begin{flalign*}
t (g'(t)-tp) + c (g(t)-g(\lhi)) &\le t(\Ll |t-\A| - (t-\A)p) + \tfrac{c\Ll}{2}(t-\A)^2 - \A t p
\\
&\le \lt \Ll(1+\tfrac{c}{2})+p\rt \cdot |t-\A| - \A^2 p/2~.
\end{flalign*}
Suppose $|t-\A| \le \f{\A^2 p/2 + \tilde{\ep}}{\Ll(1+\ff{c}{2})+p}$. Then, $\lt \Ll(1+\tfrac{c}{2})+p\rt \cdot |t-\A| - \A^2 p/2 \le \tilde{\ep}$.

So, if $\A \in [0, \tau]$, $g'(\A) = 0$, and $g(\A) \le g(\lhi)$, then all $t \in \Big[\A - \ff{\A^2 p/2 + \tilde{\ep}}{\Ll(1+c/2)+p}, 
\A + \ff{\A^2 p/2 + \tilde{\ep}}{\Ll(1+c/2)+p}\Big] \cap [\A/2, \tau]$
also satisfy the termination condition $t(g'(t)-tp) + c(g(t) - g(1)) \le \epp$. If $\ff{\A^2 p/2 + \epp}{\Ll(1+c/2)+p} \le \A/2$, the lower bound of the first interval is $\ge \A/2$ and the intersection of the two intervals contains $[\A-\ff{\A^2 p/2 + \epp}{\Ll(1+c/2)+p}, \A]$. If not, then the first interval contains $[\A/2, \A]$ as does the second interval, so the intersection of the two intervals contains $[\A/2, \A]$. Therefore, the length of the interval of points satisfying the termination condition is at least $\min\{\ff{\A}{2}, \ff{\A^2 p/2 + \tilde{\ep}}{\Ll(1+c/2)+p}\}$.

If $g'(\A) = 0$ and $g(\A) \le g(\lhi)$, then $g(0) \le g(\lhi) + \ff{\Ll}{2} \A^2$ by $\Ll$-smoothness. Since ${g(\lhi) + \ff{(p+\tilde{\ep})^2}{4\Ll}} \le g(1) < g(0)$, this implies $\A \ge \ff{p+\epp}{\Ll\sqrt{2}}$.
Therefore, the interval length is at least \linebreak
$\min\left\{\f{p+\epp}{2\sqrt{2}\Ll}, \f{p^3 / (4\Ll^2) + \epp}{(1+c/2)\Ll+p}\right\} \ge
\min\left\{\f{p+\epp}{\Ll\sqrt{8}}, \f{p^3 / (4\Ll^2) + \epp}{(2+c/2)\Ll}\right\} \ge \f{p^3 / (4\Ll^2) + \epp/\sqrt{2}}{(2+c/2)\Ll}$.

$\f{p^3 / (4\Ll^2) + \epp/\sqrt{2}}{(2+c/2)\Ll} \ge \max\left\{\f{p^3}{(8+2c)\Ll^3}, \f{\tilde{\ep}}{(4+c)\Ll}\right\} =
\max\left\{\f{b^3}{(8+2c)L^3}, \f{\tilde{\ep}}{(4+c)\Ll}\right\}$, using the fact that $\Ll = L\norm{x-v}^2$ and $p = b\norm{x-v}^2$.

Since we know at least one such interval of points satisfying the termination condition is always contained within our current search interval, this implies that if we run the algorithm until the current search interval has length at most $\max\left\{\ff{b^3}{(8+2c)L^3}, \ff{\tilde{\ep}}{(4+c)L\norm{x-v}^2}\right\}$, we will terminate with a point satisfying the necessary condition. As we halve our search interval (which is initially $[0,\lhi] \subset [0,1]$) at every iteration, we must therefore terminate in at most
$\ceil{\log^+_2 \lt (4+c) \min\left\{\ \ff{2L^3}{b^3}, \ff{L\norm{x-v}^2}{\tilde{\ep}} \right\} \rt}$ iterations.

Before each loop iteration (including the last which does not get executed when the termination condition is satisfied),
we compute $g(\A)$ and $g'(\A)$, so there are two function and gradient evaluations per iteration.
Before the loop begins, we require (at most) three function and gradient evaluations to evaluate $g(0), g(1), g'(1)$,
in addition to the evaluations required to compute $\tau$. [This becomes five if an initial guess is provided, as then we also compute $g(\text{guess})$, $g'(\text{guess})$.]
As argued earlier, if $p+\epp \ge \hat{L}$, Algorithm~\ref{alg:linesearch} terminates before Line 3.
Thus, we compute $\tau$ only if $g'(1) \ge p+\epp$, in which case Lemma~\ref{lem:stepsize}
says that at most $\ceil{\log_2(\ff{\hat{L}}{p+\epp})}+1$ additional function evaluations are required to compute $\tau$.
Note that $\ff{\hat{L}}{p+\epp} \le \min\left\{\ff{\Ll}{p}, \ff{\Ll}{\epp}\right\}$ since $p, \epp \ge 0$; thus,
 $\ceil{\log_2(\ff{\Ll}{p+\epp})} \le \ceil{\log_2 \lt \min\left\{\ff{\Ll}{p}, \ff{\Ll}{\epp}\right\}\rt} \le \ceil{\log^+_2 \lt (4+c) \min\left\{\ \ff{2L^3}{b^3}, \ff{L\norm{x-v}^2}{\tilde{\ep}} \right\} \rt}$.

Thus, the total number of function and gradient evaluations made is at most \linebreak
$8 + 3\ceil{\log^+_2 \lt (4+c) \min\left\{\ \ff{2L^3}{b^3}, \ff{L\norm{x-v}^2}{2\tilde{\ep}} \right\} \rt}$.

Note that we define $\min\{x, +\infty\} = x$ for any $x \in \R \cup \{\pm\infty\}$. Note also that if $b = 0$ and $L = 0$, or if $\tilde{\ep} = 0$ and either $L = 0$ or $x = v$, the above expression is technically indeterminate; however, observe that $g$ is constant in all of these cases, so at most one gradient evaluation is performed and the point $\A = 1$ is returned  (or, if an initial guess is passed in, then there are three evaluations — g(guess), g'(guess), and g(1) — and the point ``guess'' is returned).
\end{proof}

\subsection{Non-Strongly Quasar-Convex Algorithm Analysis}
\label{sec:nonstrong-analysis}
\begin{lem}
Suppose $\w^{(-1)} = 1$ and $\w^{(k)} = \ff{1}{2} \lt \w^{(k-1)} \lt \sqrt{\left( \w^{(k-1)}\right)^2+4}-\w^{(k-1)}\rt \rt$ for $k \ge 0$.
In the following sub-lemmas, we prove various simple properties of this sequence:
\begin{sublemma}
\label{lem:sk}
$\w^{(k)} \le \f{4}{k+6}$ for all $k \ge 0$.
\end{sublemma}
\end{lem}
\begin{proof}
The case $k = 0$ is clearly true as $\w^{(0)} = \ff{\sqrt{5}-1}{2} < \ff{2}{3}$. Suppose that $\w^{(i-1)} \le \f{4}{i+5}$ for some $i \ge 1$. $\w^{(i)} = \f{\w^{(i-1)}}{2} \lt \sqrt{\left(\w^{(i-1)}\right)^2 + 4} - \w^{(i-1)} \rt$. Using the fact that $\sqrt{x^2 + 1} \le 1+\ff{x^2}{2}$ for all $x$ and the fact that $\w^{(i-1)} \in (0,1)$,
\[\w^{(i)} \le \f{\w^{(i-1)}}{2} \lt 2 - \w^{(i-1)} + \f{\left(\w^{(i-1)}\right)^2}{2} \rt \le \w^{(i-1)}\lt 1- \f{\w^{(i-1)}}{4}\rt.\]

If $y > 0$, then $x(1-\ff{x}{4}) < \ff{4}{y+1}$ for all $0 \le x \le \ff{4}{y}$. Thus, setting $y = i+5$ yields that $\w^{(i)} \le \ff{4}{i+6}$ by the inductive hypothesis.
\end{proof}

\begin{sublemma}
\label{lem:sk2}
$\w^{(k)} \ge \f{1}{k+2}$ for all $k \ge 0$.
\end{sublemma}
\begin{proof}
The case $k = 0$ is clearly true as $\w^{(0)} = \ff{\sqrt{5}-1}{2} > \ff{1}{2}$. Suppose that $\w^{(i-1)} \ge \f{1}{i+1}$ for some $i \ge 1$.
Observe that the function $h(x) = \ff{1}{2}(x(\sqrt{x^2+4}-x))$ is increasing for all $x$. Therefore, $\w^{(i)} = h(\w^{(i-1)}) \ge h(\ff{1}{i+1}) = \ff{1}{2(i+1)} \lt \sqrt{\ff{1}{(i+1)^2} + 4} - \ff{1}{i+1}\rt = \ff{1}{2(i+1)^2}\lt \sqrt{4(i+1)^2+1}-1\rt$.

Now, it just remains to show that $\sqrt{4x^2+1} \ge \f{2x^2}{x+1}+1$ for all $x \ge 0$.
To prove this, note that $4x^2(x+1)^2 = 4x^4 + 8x^3 + 4x^2$, so
\[4x^2 + 1 = \f{4x^4 + 8x^3 + 4x^2}{(x+1)^2} + 1 \ge \f{4x^4 + 4x^3 + 4x^2}{(x+1)^2} + 1 = \lt \f{2x^2}{x+1} + 1\rt^2~.\]

\noindent Thus,
\[\w^{(i)} \ge \f{1}{2(i+1)^2}\lt \sqrt{4(i+1)^2+1}-1\rt \ge \f{1}{2(i+1)^2} \cdot \f{2(i+1)^2}{(i+2)} = \f{1}{i+2}~.\]
\end{proof}

\begin{sublemma}
\label{lem:wint}
$\w^{(k)} \in (0,1)$ for all $k \ge 0$. Additionally, $\w^{(k)} < \w^{(k-1)}$ for all $k \ge 0$.
\end{sublemma}
\begin{proof}
The fact that $\w^{(k)} > 0$ follows from Lemma \ref{lem:sk2}. To show the rest, we simply observe that $\ff{1}{2}(\sqrt{x^2+4}-x) < \ff{2}{2} = 1$ for all $x > 0$; as $\w^{(-1)} = 1$ and $\w^{(k)} = \ff{1}{2}(\sqrt{(\w\ind{k-1})^2+4}-\w\ind{k-1}) \cdot \w\ind{k-1}$ for all $k \ge 0$, the result follows.
\end{proof}

\begin{sublemma}
\label{lem:sk_real}
Define $s^{(k)} = 1+\su{i=0}{k-1} \f{1}{\w^{(i)}}$. Then, $\left( s^{(k)}\right)^{-1} \le \f{8}{(k+2)^2}$ for all $k \ge 0$.
\end{sublemma}
\begin{proof}
Applying Lemma \ref{lem:sk}, $s^{(k)} \ge 1 + \su{i=0}{k-1} \lt \f{i+6}{4} \rt = \f{k(k+11)+8}{8} \ge \f{k(k+4)+4}{8} = \ff{1}{8}(k+2)^2$, and so $\left( s^{(k)} \right)^{-1} \le \f{8}{(k+2)^2}$.
\end{proof}

\begin{sublemma}
\label{lem:ak_ind}
$\f{1}{(\w\ind{k})^2} - \f{1}{\w\ind{k}} = \su{i=-1}{k-1} \f{1}{\w\ind{i}} = s\ind{k}$ for all $k \ge 0$.
\end{sublemma}
\begin{proof}
Notice that $(\w^{(k)})^2 = (1-\w^{(k)})(\w^{(k-1)})^2$ for all $k \ge 0$, by definition of the sequence $\{\w\ind{k}\}$.
Thus, since $w\ind{k} \in (0,1)$ for all $k \ge 0$, $\f{1}{(\w\ind{k})^2}-\f{1}{\w\ind{k}} = \f{1}{(\w\ind{k-1})^2}$.
This proves the base case $k = 0$, since $\w\ind{-1}=1$.
Now, for $k \ge 0$ define $B\ind{k} = \f{1}{(\w\ind{k})^2} - \f{1}{\w\ind{k}}$.
Then for all $k \ge 0$, $B\ind{k+1}- \lt B\ind{k} + \f{1}{\w\ind{k}} \rt =
\f{1}{(\w\ind{k+1})^2} - \f{1}{\w\ind{k+1}} - \f{1}{(\w\ind{k})^2} = 0$.
Thus $B\ind{k+1} = B\ind{k} + \f{1}{\w\ind{k}} = \f{1}{\w\ind{k}}+ \su{i=-1}{k-1} \f{1}{\w\ind{i}}$ by the inductive hypothesis.
\end{proof}

{\renewcommand\footnote[1]{}\nonstrong*}

\begin{proof}
For simplicity of exposition, we present the proof in the case where $L\ind{k} = L$ for all $k$ (i.e., $L$ is known).
The general case can be handled by tightening the analysis,
similarly to the analysis of standard AGD with adaptive step size on convex functions.

In the non-strongly quasar-convex case, $\2 = 0$ and $\B = 1$.
For all $k$, $\eta^{(k)} = \ff{\1}{L\ind{k}\w^{(k)}} \ge \ff{\1}{L\ind{k}}$ since $\w^{(k)} \in (0,1)$ by Lemma \ref{lem:wint}.
Additionally, $\A^{(k)}$ is in $[0,1]$ and $(\A,x,y_{\alpha},v) = (\A^{(k)},x\ind{k},y\ind{k},v\ind{k})$ satisfies \eqref{eq:ak_existence_2} with $b = \ff{1-\B}{2\eta^{(k)}} = 0$, $c = \ff{L\ind{k}\eta^{(k)}-\1}{\B} = L\ind{k}\eta^{(k)}-\1$ by construction.
Lemmas \ref{lem:agd_one_step} and \ref{lem:agd_linesearch} thus imply that for all $k \ge 0$,
\begin{equation}
\label{eq:nonstrong_onestep}
2 ( \eta^{(k)})^2 L\ind{k} \ep^{(k+1)} + r^{(k+1)} \le
r^{(k)} + 2\eta^{(k)} \lt L\ind{k}\eta^{(k)} - \1 \rt \ep^{(k)} + 2\eta^{(k)} \tilde{\ep}~.
\end{equation}

 Define $A^{(k)} \defeq 2\left( \eta^{(k)}\right)^2 L\ind{k} - 2\eta^{(k)} \1$. So, $(A^{(k)} + 2 \eta^{(k)} \1)\ep^{(k+1)} + r^{(k+1)} \le A^{(k)} \ep^{(k)} + r^{(k)} +  2\eta^{(k)}\tilde{\ep}$.
\newline
Recall that $(\w^{(k+1)})^2 = (1-\w^{(k+1)})(\w^{(k)})^2$ and $\w^{(k)} \in (0,1)$ for all $k \ge 0$. So,
\begin{align*}
A^{(k+1)} - (A^{(k)} + 2\eta^{(k)}\1) &&=  \\
2( \eta^{(k+1)})^2 L\ind{k+1} - 2\eta^{(k+1)} \1 - 2( \eta^{(k)})^2 L\ind{k} &&= \\
2\lt \f{\1^2 L\ind{k+1}}{(L\ind{k+1})^2 ( \w^{(k+1)})^2} - \f{\1^2}{L\ind{k+1} \w^{(k+1)}} - \f{\1^2 L\ind{k}}{(L\ind{k})^2 ( \w^{(k)})^2} \rt &&= \\
2\1^2 \lt \f{1}{L\ind{k+1}} \cdot \f{1-\w^{(k+1)}}{( \w^{(k+1)})^2} - \f{1}{L\ind{k}} \cdot \f{1}{( \w^{(k)})^2} \rt &&= \\
2\1^2 \lt \f{1}{L\ind{k+1}} \cdot \f{1}{(\w^{(k)})^2} - \f{1}{L\ind{k}} \cdot \f{1}{( \w^{(k)})^2} \rt &&\le 0~. \\
\end{align*}
The final inequality comes from the fact that $L\ind{k+1} \ge L\ind{k}$, by definition of the sequence $\{L\ind{k}\}$ in Algorithm~\ref{alg:nonstrong_agd}.
So, $A^{(k+1)} = \ff{L\ind{k}}{L\ind{k+1}} (A\ind{k} + 2 \eta^{(k)} \1) \le A^{(k)} + 2 \eta^{(k)} \1$
and thus $A\ind{k+1}\ep\ind{k+1} + r\ind{k+1} \le 
(A\ind{k}+2\eta\ind{k}\1)\ep\ind{k+1}+r\ind{k+1} \le A\ind{k}\ep\ind{k} + r\ind{k} + 2\eta\ind{k}\epp$.
Applying \eqref{eq:nonstrong_onestep} repeatedly, we thus have
\begin{equation}
\label{eq:nonstrong_induct}
A^{(k)} \ep^{(k)} + r^{(k)} \le
A^{(k-1)} \ep^{(k-1)} + r^{(k-1)} + 2\eta^{(k-1)} \tilde{\ep} \le \dots \le
A^{(0)} \ep^{(0)} + r^{(0)} + 2\tilde{\ep}\su{i=0}{k-1}\eta^{(i)}.
\end{equation}
By Lemma~\ref{lem:ak_ind}, $A^{(k)} = 2( \eta^{(k)})^2 L\ind{k} - 2\eta^{(k)} \1 =
\f{2\1^2}{L\ind{k}} \lt \f{1}{(\w\ind{k})^2} - \f{1}{\w\ind{k}}\rt =
\f{2\1^2}{L\ind{k}} s^{(k)}$, where $s^{(k)} \defeq \lt 1 + \su{i=0}{k-1} \f{1}{\w^{(i)}} \rt$. Since $0 < L\ind{k} < 2L$
for all $k \ge 0$, we thus have $A\ind{k} \ge \f{\1^2}{L} s^{(k)}$.

Also, $A^{(0)} = 2( \eta^{(0)})^2 L\ind{0} - 2 \eta^{(0)} \1 = 2\f{\1^2}{L\ind{0} ( \w^{(0)})^2} - 2\f{\1^2}{L\ind{0} \w^{(0)}} = \f{2\1^2}{L\ind{0}}$, as $\w^{(0)} = \f{\sqrt{5}-1}{2}$.
\newline
So, as $r^{(k)} \ge 0$ and (by our simplifying assumption) $L\ind{k} = L$,
\begin{align*}
\ep^{(k)} &\le (A^{(k)})^{-1} \lt A^{(0)} \ep^{(0)} + r^{(0)} \rt + 2(A^{(k)})^{-1} \tilde{\ep}\su{i=0}{k-1}\eta^{(i)} & \\
&\le \f{L}{\1^2} (s^{(k)})^{-1} \lt \f{2\gamma^2}{L\ind{0}} \ep^{(0)} + r^{(0)} \rt +
\f{2\epp L}{\1} (s\ind{k})^{-1} \lt \su{i=0}{k-1} \eta^{(i)}\rt
\end{align*}

Then, the previous expression becomes $(s^{(k)})^{-1} \lt 2 \ep^{(0)} + \f{L}{\gamma^2}r^{(0)}\rt + \gamma^{-1} \tilde{\ep}$.
$\tilde{\ep} = \f{\1 \ep}{2}$ by definition and $\left( s^{(k)}\right)^{-1} \le \f{8}{(k+2)^2}$ by Lemma \ref{lem:sk_real}, which proves the bound on $\ep^{(k)}$.

For the iteration bound, we simply require $K$ large enough such that $\ff{8}{(K+2)^2}\lt \ep^{(0)} + \ff{L}{2\gamma^2}r^{(0)}\rt \le \ff{\ep}{2}$. Observe that as $f(x^{(0)}) \le f(\qx) + \ff{L}{2}\norm{x^{(0)}-\qx}^2$ by Fact \ref{fact:upper_bd}, $2\ep^{(0)} \le L r^{(0)} \le \ff{L}{\1^2} r^{(0)}$.

So, it suffices to have $\ff{8}{(K+2)^2} \lt \ff{2L}{\1^2}r^{(0)}\rt \le \ff{\ep}{2}$. Rearranging, this is equivalent to $K+2 \ge 8\1^{-1}L^{1/2}R\ep^{-1/2}$, as $r^{(0)} = R^2$. As $K$ must be a nonnegative integer, it suffices to have $K \ge \floor{8\1^{-1}L^{1/2}R\ep^{-1/2}}$.
\end{proof}

\nonstrongruntime*
\begin{proof}
Lemma \ref{lem:nonstrong_converge} implies $O(\1^{-1}L^{1/2}R\ep^{-1/2})$ iterations are needed to get an $\ep$-optimal point.
Lemma \ref{lem:linesearch} implies that each line search uses
$O \lt \log^+ \lt (1+c) \min \left\{ \ff{L\norm{x\ind{k}-v\ind{k}}^2}{\tilde{\ep}}, \ff{L^3}{b^3} \right\} \rt \rt$ function and gradient evaluations.
Again, for simplicity we focus on the case where $L\ind{k} = L$ for all $k \ge 0$; the analysis for the general case proceeds analogously.
In this case,
$b = 0$, $c = L\eta\ind{k} - \1 = \1 \lt \ff{1}{\w^{(k)}} - 1 \rt$, and $\tilde{\ep} = \ff{\1 \ep}{2}$.
By Lemma \ref{lem:sk2} and \ref{lem:wint}, $1 < \ff{1}{\w^{(k)}} \le k+2$ for all $k \ge 0$.
Thus, the number of function and gradient evaluations required for the line search at iteration $k$ of Algorithm \ref{alg:nonstrong_agd} is
$O\lt \log^+ \lt (\1 k+1) \ff{L\norm{x^{(k)}-v^{(k)}}^2}{\1 \ep} \rt\rt$.

Now, we bound $\norm{x\ind{k} - v\ind{k}}^2$. To do so, we first bound $\norm{v\ind{k} - \qx}^2 = r\ind{k}$. Recall that equation \eqref{eq:nonstrong_induct} in the proof of Lemma \ref{lem:nonstrong_converge} says that
$A\ind{k}\ep\ind{k} + r\ind{k} \le A^{(0)} \ep^{(0)} + r^{(0)} + 2\tilde{\ep}\sum\limits_{i=0}^{k-1}\eta^{(i)}$, where $A\ind{j} \defeq \ff{2\1^2}{L} \lt 1 + \sum\limits_{i=0}^{j-1} \ff{1}{\w^{(i)}}\rt$. As $A\ind{k}, \ep\ind{k} \ge 0$, this means that
\begin{align*}
r\ind{k} \le A\ind{0}\ep\ind{0} + r\ind{0} + 2\epp\su{i=0}{k-1} \eta\ind{i} = \f{2\1^2}{L} \ep\ind{0} + r\ind{0} + \f{\1^2\ep}{L}\su{i=0}{k-1}\f{1}{\w\ind{i}}~,\end{align*}
using that $\eta\ind{i} = \ff{\1}{L\w\ind{i}}$, $\epp = \ff{\1\ep}{2}$, and $A\ind{0} = \ff{2\1^2}{L}$ (as previously shown in the proof of Lemma \ref{lem:nonstrong_converge}).
Now, by Lemma \ref{lem:sk2} we have that $\sum\limits_{i=0}^{k-1}\ff{1}{\w\ind{i}} \le \sum\limits_{i=0}^{k-1} (i+2) = \ff{k(k+3)}{2}$, and by $L$-smoothness of $f$ and Fact \ref{fact:upper_bd} we have that $\ep\ind{0} \le \ff{L}{2}r\ind{0} \le \ff{L}{2\1^2}r\ind{0}$. Thus, for all $k \ge 1$, we have
\begin{align*}
r\ind{k} \le 2r\ind{0} + \ff{\1^2\ep k(k+3)}{2L} \le 2 (R^2 + \ff{\1^2 \ep k^2}{L})~,
\end{align*}
as $r\ind{0} = R^2$ and $k+3 \le 4k$ for all $k \ge 1$. In fact, the above holds for $k = 0$ as well, because $r\ind{k}$ is simply $r\ind{0}$ in this case.

By the triangle inequality, $\norm{v\ind{k} - v\ind{k-1}} \le \norm{v\ind{k}-\qx} + \norm{v\ind{k-1}-\qx} \le 2\sqrt{2(R^2 + \ff{\1^2\ep k^2}{L})}$.
Since $\B = 1$, we have that $v\ind{k-1} - \eta\ind{k-1} \G f(y\ind{k-1})$ and so $\norm{v\ind{k} - v\ind{k-1}} = \eta\ind{k-1} \norm{\G f(y\ind{k-1})}$.
Thus, \begin{equation}
\label{eq:gradbound}
\norm{\G f(y\ind{k-1})} \le (\eta\ind{k-1})^{-1} \cdot 2\sqrt{2(R^2 + \ff{\1^2\ep k^2}{L})} = L\w\ind{k-1}\1^{-1}\sqrt{8(R^2 + \ff{\1^2\ep k^2}{L})}~.
\end{equation}

\noindent Now, by definition of $x\ind{k}$, $v\ind{k}$, and $y\ind{k-1}$,
\begin{align*}
x\ind{k}-v\ind{k} &= y\ind{k-1} - \ff{1}{L} \G f(y\ind{k-1}) - v\ind{k} \\
&= \A\ind{k-1} x\ind{k-1} + (1-\A\ind{k-1})v\ind{k-1} - \ff{1}{L} \G f(y\ind{k-1}) - v\ind{k} \\
&= \A\ind{k-1} x\ind{k-1} + (1-\A\ind{k-1})v\ind{k-1} - \ff{1}{L} \G f(y\ind{k-1}) - \lt v\ind{k-1} - \eta\ind{k-1}\G f(y\ind{k-1})\rt\\
&= \A\ind{k-1} (x\ind{k-1}-v\ind{k-1}) + (\eta\ind{k-1}-\ff{1}{L})\G f(y\ind{k-1})~.
\end{align*}
\noindent Therefore,
\begin{align*}
\norm{x\ind{k}-v\ind{k}} &\le \A\ind{k-1} \norm{x\ind{k-1}-v\ind{k-1}} + \left|\eta\ind{k-1}-\ff{1}{L}\right| \cdot \norm{\G f(y\ind{k-1})}
\\&\le \norm{x\ind{k-1}-v\ind{k-1}} + \left(\eta\ind{k-1}+\ff{1}{L}\right) \cdot \norm{\G f(y\ind{k-1})}
\\&\le \norm{x\ind{k-1}-v\ind{k-1}} + \ff{2}{L\w\ind{k-1}} \cdot \norm{\G f(y\ind{k-1})}
\\&\le \norm{x\ind{k-1}-v\ind{k-1}} + \1^{-1}\sqrt{32(R^2 + \ff{\1^2\ep k^2}{L})}
\\&\le \norm{x\ind{k-1}-v\ind{k-1}} + \sqrt{32}\1^{-1}\lt R + \1 k\sqrt{\ff{\ep}{L}}\rt~,
\end{align*}
where the first inequality is the triangle inequality, the third inequality uses that $\eta\ind{k-1} = \ff{\1}{L\w\ind{k-1}}$ and that $\1, \w\ind{k-1} \in (0,1]$, the fourth inequality uses \eqref{eq:gradbound}, and the final inequality uses that $\sqrt{a+b} \le \sqrt{a}+\sqrt{b}$ for any $a,b \ge 0$.

As this holds for all $k \ge 1$, we have by induction that for all $k \ge 0$,
\begin{align*}
\norm{x\ind{k}-v\ind{k}} \le \norm{x\ind{0}-v\ind{0}} + \su{j=1}{k} \sqrt{32}\1^{-1}\lt R + \1 j\sqrt{\ff{\ep}{L}}\rt = \sqrt{32}\1^{-1}\su{j=1}{k} \lt R + \1 j\sqrt{\ff{\ep}{L}}\rt~,
\end{align*}
since $x\ind{0}=v\ind{0}$. Simplification yields $\norm{x\ind{k}-v\ind{k}} \le \sqrt{32}k\1^{-1}R + \sqrt{8}k(k+1)\sqrt{\ff{\ep}{L}}$. For all $k \ge 1$, it is the case that $k+1 \le 2k$, so $\norm{x\ind{k}-v\ind{k}} \le \sqrt{32}\lt k\1^{-1}R + k^2\sqrt{\ff{\ep}{L}}\rt$; this inequality holds for $k = 0$ as well, as $\norm{x\ind{0}-v\ind{0}} = 0$ in this case.

\noindent Suppose $k \le \floor{4\1^{-1}L^{1/2}R\ep^{-1/2}}$. Then
\begin{align*}
\norm{x\ind{k}-v\ind{k}} &\le \sqrt{32}\lt 4\1^{-1}L^{1/2}R\ep^{-1/2} \cdot \1^{-1}R + 16\1^{-2}LR^2\ep^{-1}\cdot \sqrt{\ff{\ep}{L}}\rt
\\ &= 80\sqrt{2} \cdot \1^{-2}L^{1/2}R^2\ep^{-1/2}~.
\end{align*}

Recall that the line search at iteration $k$ requires $O\lt \log^+ \lt (\1 k+1) \ff{L\norm{x^{(k)}-v^{(k)}}^2}{\1 \ep} \rt\rt$
function and gradient evaluations.
$(\1 k+1) \ff{L\norm{x^{(k)}-v^{(k)}}^2}{\1 \ep} \le (4L^{1/2}R\ep^{-1/2} + 1) \cdot 12800 (\1^{-5}L^{2}R^4\ep^{-2})$.
Therefore, each line search indeed requires $O\lt \log^+ \lt  \1^{-1} L^{1/2} R \ep^{-1/2} \rt\rt$ function and gradient evaluations.

As the number of iterations $k$ is $O(\1^{-1}L^{1/2}R\ep^{-1/2})$, the total number of function and gradient evaluations required is thus
$O \lt \1^{-1} L^{1/2}R \ep^{-1/2} \log^+ \lt \1^{-1} L^{1/2} R \ep^{-1/2} \rt \rt$, as claimed.

As in the strongly convex case, the algorithm may continue to run if the specified number of $\text{iterations}$ $K$ is larger; however, this theorem combined with Lemma \ref{lem:nonstrong_converge} shows that
$x\ind{k}$ will be $\ep$-optimal if ${k = \floor{4\1^{-1}L^{1/2}R\ep^{-1/2}}}$, and this $x\ind{k}$ will be produced using \\
$O \lt \1^{-1} L^{1/2}R \ep^{-1/2} \log^+ \lt \1^{-1} L^{1/2} R \ep^{-1/2} \rt \rt$ function and gradient evaluations. (Future iterates $x\ind{k'}$ with $k' > \floor{4\1^{-1}L^{1/2}R\ep^{-1/2}}$ will also be $\ep$-optimal.)
\end{proof}

\begin{remark}
If $f$ is $L$-smooth and $\1$-quasar-convex with $\1 \in (0,1]$ and $\norm{x^{(0)}-x^*} \le R$, then
gradient descent with step size $\ff{1}{L}$ returns a point $x$ with $f(x) \le f(\qx) + \ep$ after
$O\lt \1^{-1}LR^2 \ep^{-1} \rt$ function and gradient evaluations.
\end{remark}
\begin{proof}
\,See Theorem 1 in \citep{guminov2017accelerated}.
\end{proof}

\subsection{Comparisons with Standard AGD}
We have described how our algorithms relate to standard (convex) AGD. We briefly comment on the difference between the analysis of our algorithms and that of standard AGD, and on the use of a line search ``initial guess'' inspired by the setting of $\A\ind{k}$ in standard AGD.
\subsubsection{Analysis}
\label{app:agd-analysis}
Concretely, a key step in the proof of convergence of algorithms that extend Algorithm \ref{alg:agd} (including our algorithms as well as standard AGD) is to bound $Q\ind{k}$; this bound is then combined with Lemma \ref{lem:agd_one_step} to get the final convergence bound. In standard AGD, we bound $Q\ind{k}$ by setting $\A\ind{k}$ to a specific predetermined value. For instance, in the non-strongly convex case (where $\beta= 1$), $\A\ind{k}$ is set such that $\f{\A\ind{k}}{1-\A\ind{k}} = L\ind{k}\eta\ind{k}-1$.\footnote{As $\eta\ind{k} = \ff{1}{L\ind{k}\w\ind{k}}$, this implies that $\A\ind{k} = 1-\w\ind{k}$.} We then have $Q\ind{k} = 2\eta\ind{k} \cdot \f{\A\ind{k}}{1-\A\ind{k}}\G f(y\ind{k})^\top{(x\ind{k}-y\ind{k})} = 2\eta\ind{k}{(L\ind{k}\eta\ind{k}-1)} \G f(y\ind{k})^\top (x\ind{k}-y\ind{k}) $, and then we use convexity to obtain that $Q\ind{k} \le \\2\eta\ind{k}{(L\ind{k}\eta\ind{k}-1) } {(f(x\ind{k})-f(y\ind{k}))} =  2\eta\ind{k}(L\ind{k}\eta\ind{k}-1)  (\ep\ind{k}-\ep_y\ind{k})$. By contrast, for our algorithms, we bound $Q\ind{k}$ using Lemma \ref{lem:agd_linesearch}.

\subsubsection{Line Search Initial Guess}
\label{sec:guess}
In special cases, specifying an ``initial guess'' for $\A$ in the binary line search (Algorithm \ref{alg:linesearch}) can speed up our algorithms, by allowing the line search to be circumvented a large portion of the time.
For instance, at each step $k$ we can use the $\A^{(k)}$ prescribed by the standard version of AGD as a guess: this value is $\ff{\sqrt{L\ind{k} / \mu}}{1+\sqrt{L\ind{k} / \mu}}$ in the strongly convex case (Algorithm \ref{alg:strongly_agd}), and $1 - \w^{(k)}$ in the non-strongly convex case (Algorithm \ref{alg:nonstrong_agd}). Thus, when $f$ is convex or strongly convex (and thus $\1 = 1$), our respective algorithms using the initial guess are equivalent to standard AGD (as described in \citep{Nesterov04}), since this initial guess always satisfies the necessary condition \eqref{eq:ak_existence_2} by convexity [in fact, it satisfies the stronger \eqref{eq:ak_existence}] and will thus be chosen as the value of $\A^{(k)}$.
Moreover, even when $f$ is nonconvex, checking this initial guess costs at most one extra function and gradient evaluation each per invocation of Algorithm \ref{alg:linesearch}.
So, when $\1 = 1$ we can interpret the overall algorithm as a ``robustified'' version of standard AGD---each iteration is identical to that of standard AGD unless a ``convexity violation'' between $x\ind{k}$ and $v\ind{k}$ is detected, in which case we fall back to the binary search.

\subsubsection{Analysis Techniques}
We remark that our analysis can also be recast in the framework of estimate sequences (for instance, following \citep{Nesterov04}), by generalizing the analysis for standard AGD.
The analysis presented in this work is an adaptation of a somewhat different style of analysis of standard AGD, based on analyzing the one-step decrease in the more general potential function presented in Lemma~\ref{lem:agd_one_step}.
Indeed, as mentioned, the standard AGD algorithms for both convex and strongly convex minimization
are also specific instances of the framework presented in Algorithm~\ref{alg:agd}.

\section{The Structure of Quasar-Convex Functions}
\label{sec:quasar-structure}
In this section, we prove various properties of quasar-convex functions. First, we state a slightly more general definition of quasar-convexity on a convex domain.

\begin{defn}
Let $\X \subseteq \R^n$ be convex. Furthermore, suppose that either $\X$ is open or $n = 1$.
Let $\1 \in (0,1]$ and let $\xStar \in \X$ be a minimizer of the differentiable function $f : \X \rightarrow \R$.
The function $f$ is \emph{$\1$-quasar-convex} on $\X$ with respect to $x^*$ if for all $x \in \X$,
\begin{equation*}
f(\xStar) \ge  f(x) + \frac{1}{\1} \grad f(x)^\top (\xStar-x).
\end{equation*}
Suppose also $\2 \ge 0$. The function $f$ is \emph{$(\1,\2)$-strongly quasar-convex} on $\X$ if for all $x \in \X$,
\begin{equation*}
f(\xStar) \ge f(x) + \frac{1}{\1} \grad f(x)^\top (\xStar-x) + \frac{\2}{2} \norm{ \xStar -x }^2.
\end{equation*}
If $\X$ is of the form $[a,b] \subseteq \R$, then $\G f(a)$ and $\G f(b)$ here denote $\lim\limits_{h \ra 0^+} \ff{f(a+h)-f(a)}{h}$ and $\lim\limits_{h \ra 0^-} \ff{f(b+h)-f(b)}{h}$, respectively. Differentiability simply means that $\G f(x)$ exists for all $x \in \X$.
\label{defn:gen-defns}
\end{defn}
Definition~\ref{defn:gen-defns} is exactly the same as Definition~\ref{defn:defns} if the domain $\X = \R^n$. We remark that it is possible to generalize Definition~\ref{defn:gen-defns} even further to the case where $\X$ is a \textit{star-convex} set with star center $\qx$.

\subsection{Proof of Observation~\ref{obs:unimodalImpliesQuasar}}\label{sec:unimodal-quasar}

\unimodalImpliesQuasar*

\begin{proof}
First, we prove that if $f$ is continuously differentiable and unimodal with nonzero derivative except at minimizers, then $f$ is $\1$-quasar-convex for some $\1 > 0$.

Let $\qx$ be a minimizer of $f$ on $[a,b]$, and let $x \in [a,b]$ be arbitrary. Define $g_x(t) = f((1-t)\qx+tx)$. By unimodality of $f$, $g_x$ is differentiable and increasing on $[0,1]$, so $g_x'(t) \ge 0$ for $t \in [0,1]$, and
\[
f(x) - f(\qx) = g_x(1) - g_x(0) = \I{0}{1} g_x'(t) \,dt ~.
\]
Also, $g_x'(1) = f'(x)(x-\qx) \ne 0$ by assumption for all $x$ with $f(x) > f(\qx)$. Note that if $f(x) = f(\qx)$, then $g_x(t)$ is constant on $[0, 1]$ by unimodality and so $g_x'(t) = 0$ for all $t \in [0,1]$.

Define $C_{\qx} = \sup\limits_{x \in [a,b]}\sup\limits_{t \in [0,1]} \f{g_x'(t)}{g_x'(1)}$, where we define the inner supremum to be 1 if $f(x) = f(x^*)$. By continuity of each $g_x'$ over $[0,1]$ and the fact that $g_x'(1) > 0$ for all $x \in [a,b]$ with $f(x) > f(x^*)$, $\sup_{t \in [0,1]} \ff{g_x'(t)}{g_x'(1)}$ is a continuous function of $x$. Thus as the outer supremum is over the compact interval $[a,b]$, \,$C_{\qx}$ indeed exists; note that $C_{\qx}  \in [1,\infty)$.

For any $x \in [a,b]$ with $f(x) > f(x^*)$, we thus have $\f{f(x)-f(\qx)}{f'(x)(x-\qx)} = \f{\int_{0}^{1}g_x'(t) \, dt}{g_x'(1)} \le C_{\qx}$, meaning $f(\qx) \ge f(x) + C_{\qx} (f'(x)(\qx-x))$. This also holds for all $x$ such that $f(x) = f(\qx)$, as either $x = \qx$ or $f'(x) = 0$ in these cases. Thus, $f$ is $\ff{1}{C_{\qx}}$ quasar-convex on $[a,b]$ with respect to $x^*$. Finally, if we define $C_{\max} = \max\limits_{\qx \in \argmin_{x \in [a,b]} f(x)} C_{\qx}$, we have that $f$ is $\ff{1}{C_{\max}}$ quasar-convex on $[a,b]$ where $\ff{1}{C_{\max}} \in (0,1]$ is a constant depending only on $f$, $a$, and $b$. This completes the proof.

Now, we prove the other direction (which is much simpler). Suppose that $f : [a,b] \ra \R$ is differentiable and quasar-convex for some $\1 \in (0, 1]$. Then $\ff{1}{\1}f'(x)(x-\qx) \ge f(x) - f(\qx) \ge 0$. If $x$ is not a minimizer of $f$, then the last inequality is strict; otherwise, either $x \in \{a,b\}$ or $f'(x) = 0$. In other words, assuming $x$ is not a minimizer, when $x < \qx$ [i.e. to the left of $\qx$], $f' < 0$ and so $f$ is strictly decreasing, while when $x > \qx$ [i.e. to the right of $\qx$], $f' > 0$ and so $f$ is strictly increasing. This implies that $f$ is unimodal.

Finally, suppose $h : \R^n \ra \R$ is $\1$-quasar-convex with respect to a minimizer $\qx$, suppose $d \in \R^n$ has $\norm{d} = 1$, and define $f(\theta) \defeq h(\qx + \theta d)$. Note that $f'(\theta) = d^\top \G h(\qx + \theta d)$ and that $\theta = 0$ minimizes $f$. By $\1$-quasar-convexity of $h$ with respect to $\qx$, we have for all $\theta \in \R$ that
\begin{align*}
f(0) = h(\qx) \ge h(\qx + \theta d) + \ff{1}{\1}\G h(\qx + \theta d)^\top (\qx - (\qx + \theta d)) = f(\theta) + \ff{1}{\1} f'(\theta)(0-\theta)~,
\end{align*}
meaning that $f$ is $\1$-quasar-convex.
\end{proof}

\subsection{Characterizations of Quasar-Convexity}
\label{sec:equivs}
\begin{lem}
\label{lem:star_char}
Let $f : \X \ra \R$ be differentiable with a minimizer $x^* \in \X$, where the domain $\X \subseteq \R^n$ is open and convex.\footnote{We remark that this lemma still holds if $\X$ is open and star-convex with star center $\qx$, or if $\X$ is any subinterval of $\R.$} Then, the following two statements:
\begin{equation}
\label{eq:star_def_undiff}
f(tx^*+(1-t)x) + t\left(1-\f{t}{2-\1}\right)\f{\1\2}{2}\norm{x^*-x}^2 \le \1 t f(x^*) + (1-\1 t)f(x)\,\, \forall x \in \X, \,t \in [0,1]
\end{equation}
\begin{equation}
\label{eq:star_diff}
f(x^*) \ge f(x) + \f{1}{\1}\G f(x)^\top(x^*-x) + \f{\2}{2} \norm{x^*-x}^2 \,\,\forall x \in \X
\end{equation}
are equivalent for all $\2 \ge 0$, $\1 \in (0,1]$.
\end{lem}

\begin{proof}
First, we prove that \eqref{eq:star_diff} implies \eqref{eq:star_def_undiff}.

Suppose \eqref{eq:star_diff} holds and $\2 = 0$. Let $x \in \X$ be arbitrary and for all $t\in [0,1]$ let $x_t \defeq (1-t)x^* + t x$ and let $g(t) \defeq f(x_t) - f(x^*)$. Since $g'(t) = \grad f(x_t)^\top (x - x^*)$ and $x^* - x_t = - t (x^* - x)$, substituting these equalities into  $\eqref{eq:star_diff}$ yields that 
$g(t) \le \frac{t}{\1} g'(t)$ for all $t \in [0,1]$. 

Rearranging, we see that the inequality in \eqref{eq:star_def_undiff} [for fixed $x$] is equivalent to the condition
that $g(t) \le \ell(t)$ for all $t \in [0,1]$, where ${\ell(t) \defeq (1 - \1 (1 - t)) g(1)}$. We proceed by contradiction: suppose that for some $\alpha \in [0, 1]$ it is the case that $g(\alpha) > \ell(\alpha)$. Note that $\A > 0$ necessarily. Let $\beta$ be the minimum element of the set $\{ t \in [\alpha, 1] : g(t) = \ell(t) \}$. Since $g(1) = \ell(1)$, such a $\beta$ exists with $\alpha < \beta$. Consequently, for all $t \in (\alpha, \beta)$ we have $g(t) \geq \ell(t)$ and so
\begin{equation}
\label{eq:g1}
\int_{\alpha}^{\beta} g'(t) \,dt
= g(\beta) - g(\alpha)
< \ell(\beta) - \ell(\alpha)
= \1 (\beta - \alpha) g(1)
\end{equation}
and
\begin{equation}
\label{eq:g2}
(\beta - \alpha) g(1)
= \int_{\alpha}^{\beta}  \frac{\ell(t)}{1 - \1 (1 - t)} \,dt
\leq \int_{\alpha}^{\beta}  \frac{g(t)}{1 - \1 (1 - t)} \,dt ~.
\end{equation}
Combining \eqref{eq:g1} and \eqref{eq:g2} and using that $g(t) \le \ff{t}{\1}g'(t)$, we have

\[
\int_{\alpha}^{\beta}
\left[
\frac{1}{t}
- \frac{1}{1 - \1 (1 - t)}
\right]
g(t) \,dt
\le
\int_{\alpha}^{\beta} \f{g'(t)}{\1} \,dt - \int_{\alpha}^{\beta}  \frac{g(t)}{1 - \1 (1 - t)} \,dt 
< 0
\]
As $g(t) = f(x_t)-f(x^*) \ge 0$ and $1/t \ge 1/(1-\1 (1 - t))$ for all $t \in [\A,\B] \subset (0,1]$, we have a contradiction.

Now, suppose $\2 > 0$. Define $h(x) \triangleq f(x) - \f{\1\2}{2(2-\1)} \norm{x^*-x}^2$. Observe that $h(x^*) = f(x^*)$, $\G h(x) = \G f(x) - \f{\1\2}{2-\1} (x-x^*)$, and $\G h(x)^\top(x^*-x) = \G f(x)^\top(x^*-x) + \f{\1\2}{2-\1} \norm{x^*-x}^2$. Thus, by algebraic simplification and then application of \eqref{eq:star_diff} by assumption,
\begin{flalign*}
h(x) + \f{1}{\1} \G h(x)^\top(x^*-x) &= f(x) - \f{\1\2}{2(2-\1)} \norm{x^*-x}^2 + \f{1}{\1}\G f(x)^\top(x^*-x) + \f{\2}{2-\1} \norm{x^*-x}^2 && \\
&= f(x) + \f{1}{\1}\G f(x)^\top(x^*-x) + \f{\2}{2}\norm{x^*-x}^2 \left(- \f{\1}{2-\1} +  \f{2}{2-\1} \right)&& \\
&=  f(x) + \f{1}{\1}\G f(x)^\top(x^*-x) + \f{\2}{2}\norm{x^*-x}^2 &&\\
&\le f(x^*) = h(x^*)~.
\end{flalign*}

\noindent As we earlier showed that \eqref{eq:star_diff} implies \eqref{eq:star_def_undiff} in the $\mu = 0$ case, we have that
\[h(tx^* + (1-t)x) \le \1 t h(x^*) + (1-\1 t)h(x)~.\]

\noindent Substituting in the definition of $h$:
\begin{flalign*}
&f(tx^* + (1-t)x) - \f{\1\2}{2(2-\1)} \norm{x^*-tx^*-(1-t)x}^2 \\ 
\le\,\,& \1 t f(x^*) + (1-\1 t)f(x) - (1-\1 t)\f{\1\2}{2(2-\1)}\norm{x^*-x}^2~.
\end{flalign*}

\noindent Rearranging terms and simplifying yields
\begin{flalign*}
&f(tx^* + (1-t)x) + \f{\1\2}{2(2-\1)} \left( (1-\1 t) \norm{x^*-x}^2 - (1-t)^2\norm{x^*-x}^2\right)
\\
\le\,\,& \1 t f(x^*) + (1-\1 t)f(x)~.
\end{flalign*}
Finally, $(1-\1t) - (1-t)^2 = t((2-\1)-t)$, which gives the desired result.

\noindent Now, we prove that \eqref{eq:star_def_undiff} implies \eqref{eq:star_diff}.

This time, define $g(t) \triangleq f(tx^* + (1-t)x)$. For $t \in [0,1)$, $g'(t) = \G f(tx^* + (1-t)x)^\top(x^*-x)$. By assumption, $g(t) + t\left(1-\f{t}{2-\1}\right)\f{\1\2}{2}\norm{x^*-x}^2 \le \1 tg(1) + (1-\1 t)g(0)$ for all $t \in [0,1]$, so $g(1) \ge g(0) + \f{g(t) - g(0)}{\1 t} + \left(1-\f{t}{2-\1}\right)\f{\2}{2}\norm{x^*-x}^2$ for all $t \in (0,1]$. Taking the limit as $t \downarrow 0$ yields $f(x^*) = g(1) \ge g(0) + \f{1}{\1}g'(0) + \f{\2}{2}\norm{x^*-x}^2 = f(x) + \f{1}{\1}\G f(x)^\top(x^*-x) + \f{\2}{2}\norm{x^*-x}^2$.
\end{proof}

\begin{rem}
\label{rem:star_char_gen}
A modified version of Lemma \ref{lem:star_char} holds if $x^*$ is replaced with any point $\hat{x} \in \X$,
where either $\1 = 1$ or \eqref{eq:star_def_undiff} and \eqref{eq:star_diff} hold for all $x \in \X$ with $f(x) \ge f(\hat{x})$.
If $f$ satisfies either of these equivalent properties, we then say that $f$ is ``$(\1,\2)$-strongly quasar-convex with respect to $\hat{x}$.''
\end{rem}

\begin{rem}\label{eq:quasar-convex-more-general}
Using Remark \ref{rem:star_char_gen}, we can show that even if $\hat{x}$ is not a minimizer of the function $f$, Algorithms~\ref{alg:strongly_agd} and \ref{alg:nonstrong_agd} can still be applied to efficiently finding a point that has an objective value of at most $f(\hat{x}) + \ep$;
the respective runtime bounds are the same, and the proofs remain essentially unchanged.
\end{rem}

\noindent Note that when $\1 = 1, \2 = 0$, and \eqref{eq:star_def_undiff} is required to hold for \textit{all} minimizers of $f$, it becomes the standard definition of star-convexity \citep{nesterov2006cubic}.

\begin{coro}
\label{rem:distbound}
If $f$ is $(\1,\2)$-strongly quasar-convex with minimizer $\qx$, then
$$f(x) \ge f(x^*) + \f{\1\2}{2(2-\1)}\norm{x^*-x}^2,~ \forall x$$
\end{coro}
\begin{proof}
Plug in $t = 1$ to \eqref{eq:star_def_undiff} to get
\[f(x^*) + \left(1-\f{1}{2-\1}\right)\f{\1\2}{2}\norm{x^*-x}^2 \le \1 f(x^*) + (1-\1)f(x)~.\]
Simplifying yields
\[f(x) \ge f(x^*) + \left(1-\f{1}{2-\1}\right)\f{\1\2}{2(1-\1)}\norm{x^*-x}^2 = f(x^*) + \f{\1\2}{2(2-\1)}\norm{x^*-x}^2~. \tag*{\qedhere}\]
\end{proof}

\begin{fact}
\label{fact:upper_bd}
If $f : \X \ra \R$ is $L$-smooth, $x^*$ is a minimizer of $f$, and the domain $\X \subseteq \R^n$ is open and star-convex with star center $x^*$, then $f(y) \le f(x^*) + \ff{L}{2}\norm{y-x^*}^2$ for all $y \in \X$.
\end{fact}
\begin{proof}
This is a simple and well-known fact that is true of any $L$-smooth function (whether or not it is quasar-convex); for completeness, we provide the proof.

Define $g(t) \defeq f((1-t)x^* + ty)$, for $t \in [0, 1]$. So, $g'(t) = {\G f((1-t)x^* + ty)^\top} (y-x^*)$, $g(0) = f(x^*)$, and $g(1) = f(y)$. Since $g'(0) = 0$ and $f$ is $L$-smooth, $\norm{\G f((1-t)x^* + ty)} \le L \norm{(1-t)x^* + ty - x^*} = L t \norm{y - x^*}$. So, $g'(t) \le |g'(t)| \le Lt \norm{y-x^*}^2$, and thus $f(y) = g(1) = \I{0}{1} g'(t) \, dt + g(0) \le \I{0}{1} Lt \norm{y-x^*}^2\,dt + g(0) = \ff{L}{2} \norm{y-x^*}^2 + f(x^*)$.
\end{proof}

\begin{observation}
\label{obs:l_vs_mu}
If $f$ is $(\1,\2)$-strongly quasar-convex, then $f$ is not $L$-smooth for any $L < \ff{\1\2}{2-\1}$.
\end{observation}
\begin{proof}
If $f$ is $(\1,\2)$-strongly quasar-convex, Corollary \ref{rem:distbound} says that $f(x) \ge f(x^*) + \ff{\1\2}{2(2-\1)}\norm{x^*-x}^2$ for all $x$. If $f$ is $L$-smooth, Fact \ref{fact:upper_bd} says that $f(x) \le f(x^*) + \ff{L}{2}\norm{x^*-x}^2$ for all $x$.

Thus, if $f$ is $(\1,\2)$-strongly quasar-convex and $L$-smooth, we have $\ff{\1\2}{2(2-\1)}\norm{x^*-x}^2 \le \ff{L}{2}\norm{x^*-x}^2$ for all $x$, which means that we must have $L \ge \ff{\1\2}{2-\1}$.
\end{proof}

\begin{observation}
\label{obs:minimizers}
If $f$ is $\1$-quasar convex, the set of its minimizers is star-convex.
\end{observation}
\begin{proof}
Recall that a set $S$ is termed \textit{star-convex} (with star center $x_0$) if there exists an $x_0 \in S$ such that for all $x \in S$ and $t \in [0,1]$, it is the case that $tx_0 + (1-t)x \in S$ \citep{munkres}.

Suppose $f : \X \ra \R$ is $\1$-quasar-convex with respect to a minimizer $\qx \in \X$, where $\X$ is convex. Suppose $y \in \X$ also minimizes $f$. Then for any $t \in [0,1]$, equation \eqref{eq:star_def_undiff} implies that $f(t\qx + (1-t)y) \le \1t f(\qx) + (1-\1t)f(y) = \1 t f(\qx) + (1-\1t)f(\qx) = f(\qx)$. So, $t\qx + (1-t)y$ is in $\X$ and also minimizes $f$. Thus, the set of minimizers of $f$ is star-convex, with star center $\qx$.
\end{proof}

\begin{observation}
\label{obs:unique}
If $f$ is $(\1,\2)$-strongly quasar-convex with $\2 > 0$, $f$ has a unique minimizer.
\end{observation}
\begin{proof}
\,By Corollary \ref{rem:distbound}, $f(x) > f(x^*)$ if $\mu > 0$ and $x \ne x^*$, implying that $x$ minimizes $f$ iff $x = x^*$.
\end{proof}

\begin{observation}
\label{obs:tradeoff}
Suppose $f$ is differentiable and $(\1,\2)$-strongly quasar-convex. Then $f$ is also $(\theta \1,\2/\theta)$-strongly quasar-convex for any $\theta \in (0,1]$.
\end{observation}
\begin{proof}
$(\1,\2)$-strong quasar-convexity states that $0 \ge f(\xStar) - f(x) \ge \frac{1}{\1} \grad f(x)^\top (\qx - x) + {\frac{\2}{2} \norm{\qx - x}^2}$ for some $\qx$ and all $x$ in the domain of $f$. Multiplying by $\ff{1}{\theta}-1 \ge 0$, it follows that \newline
$f(\qx) \ge f(x) + \frac{1}{\1} \grad f(x)^\top (\xStar - x) + \frac{\2}{2} \norm{ x - \xStar }^2 \ge f(x) + \frac{1}{\1 \theta} \grad f(x)^\top (\xStar - x) + \frac{\2}{2\theta} \norm{\qx - x}^2$.

Note that any $(\1,\2)$-strongly quasar-convex function is also $(\1,\tilde{\2})$-strongly quasar-convex for any $\tilde{\2} \in [0,\2]$. Thus, the restriction $\1 \in (0,1]$ in the definition of quasar-convexity may be made without any loss of generality compared to the restriction $\1 > 0$.
\end{proof}

\begin{observation}
\label{obs:scaling}
The parameter $\1$ is a dimensionless quantity, in the sense that if $f$ is $\1$-quasar-convex on $\R^n$, the function $g(x) \defeq a \cdot f(b x)$ is also $\1$-quasar-convex on $\R^n$, for any $a \ge 0, b \in \R$.
\end{observation}
\begin{proof}
If $a$ or $b$ is 0, then $g$ is constant so the claim is trivial. Now suppose $a,b \ne 0$.
Let $x^*$ denote the quasar-convex point of $f$. Observe that as $x^*$ minimizes $f$, $x^* / b$ minimizes $g$.
By \eqref{eq:star_def_undiff}, for all $x \in \R^n$ we have
\begin{flalign*}
\ff{1}{a}g((tx^*+(1-t)x)/b) &= f(tx^*+(1-t)x) \\
&\le \1 t f(x^*) + (1-\1 t)f(x) \\
&= \1 t \cdot \ff{1}{a} g(x^*/ b) + (1-\1 t) \cdot \ff{1}{a} g(x / b)~.
\end{flalign*}
Multiplying by $a$, we have
$g(t(x^* / b) + (1-t)(x/b)) \le \1 t g(x^* / b) + (1-\1 t)g(x / b)$ for all $x \in \R^n$.
Since $x/b$ can take on any value in $\R^n$, this means that $g$ is $\1$-quasar-convex with respect to $x^* / b$.
\end{proof}

\subsection{Construction of Quasar-Convex Functions}
\label{app:quasar_construction}
We now briefly describe some basic ``building blocks'' and closure properties of the family of quasar-convex functions,
inspired by the analogous discussion for star-convex functions in Appendix A of \cite{lee2016optimizing}.
(Recall also that star-convex functions, such as the examples in \cite{lee2016optimizing}, are quasar-convex with $\gamma=1$.)

\newcommand{\bfm}{\mathbf{M}}
\begin{enumerate}
\item Suppose $f : \R^n \ra \R$ is $(\1,\2)$-quasar-convex with respect to $x^* = \mathbf{0}$. Let $a \ge 0, c \in \R$ be scalars, $\bfm \in \R^{m \times n}$, and $b \in \R^m$. Then $g(x) = a \cdot f(\bfm(x + b)) + c$ is $(\1, \2\cdot \sigma_{\text{min}}^2(\bfm))$-quasar-convex, where $\sigma_{\text{min}}(\bfm)$ denotes the smallest singular value of $\bfm$.
\begin{itemize}
\item \emph{Proof}: It is easy to see that adding the constant $c$ does not affect the quasar-convexity properties, and that $g(x-b)$ has the same quasar-convexity properties as $g(x)$; so, by Observation \ref{obs:scaling}, it suffices to prove the claim for $a = 1, b = \mathbf{0}, c =0$. We have $f(\mathbf{0}) \ge f(x) - \ff{1}{\1} \G f(x)^\top x + \ff{\2}{2} \norm{x}^2$ for all $x \in \R^n$, by $(\1,\2)$-quasar-convexity of $f$ with respect to $x^* = \mathbf{0}$. So $g(\mathbf{0}) = f(\mathbf{0}) \ge f(\bfm y) - \ff{1}{\1} \G f(\bfm y)^\top (\bfm y) + \ff{\2}{2} \norm{\bfm y}^2 =  f(\bfm y) - \ff{1}{\1} \bfm \G f(\bfm y)^\top y + \ff{\2}{2} \norm{\bfm y}^2 \ge f(\bfm y) - \ff{1}{\1} \bfm \G f(\bfm y)^\top y + \ff{\2\sigma_{\text{min}}^2(\bfm)}{2} \norm{y}^2 = \linebreak g(y) + \ff{1}{\1} \G g(y)^\top (\mathbf{0} - y) +  \ff{\2\sigma_{\text{min}}^2(\bfm)}{2} \norm{\mathbf{0} - y}^2$ for all $y \in \R^m$, which proves the claim.

\end{itemize}
\item If $f,g$ are $(\1_1,\2_1)$ and $(\1_2,\2_2)$ quasar-convex respectively with respect to the same minimizer $x^*$, then $h(x) = f(x) + g(x)$ is $(\min\{\1_1,\1_2\}, \,\2_1 + \2_2)$ quasar-convex with respect to the same minimizer $x^*$.
\begin{itemize}
\item \emph{Proof}:
$h(x^*) = f(x^*) + g(x^*) \ge f(x) + \ff{1}{\1_1} \G f(x)^\top (x^* - x) + \ff{\2_1}{2} \norm{x^* - x}^2 +  g(x) + \ff{1}{\1_2} \G g(x)^\top (x^* - x) + \ff{\2_1}{2} \norm{x^* - x}^2 = 
h(x) + \ff{\2_1 + \2_2}{2} \norm{x^* - x}^2 + \ff{1}{\1_1} \G f(x)^\top (x^* - x) + \ff{1}{\1_2} \G g(x)^\top (x^* - x) \ge h(x) + \ff{\2_1 + \2_2}{2} \norm{x^* - x}^2 + \ff{1}{\min\{\1_1,\1_2\}} (\G f(x) + \G g(x))^\top (x^* - x)$ as desired, since $\G f(x)^\top (x^* - x) \le f(x^*) - f(x) \le 0$ and similarly $\G g(x)^\top (x^* - x) \le 0$.
\end{itemize}
\item Suppose $f, g$ are $\1$-quasar-convex with respect to the same point $x^*$, and $f(x^*) = g(x^*) = 0$. Then $h(x) = f(x)g(x)$ is also $\1$-quasar-convex with respect to $x^*$.
\begin{itemize}
\item \emph{Proof}:
Using Lemma \ref{lem:star_char} and the fact that $f, g$ are nonnegative,
$h(tx^* + (1-t)x) = f(tx^* + (1-t)x)g(tx^* + (1-t)x) \le (1-\gamma t) f(x) \cdot (1-\gamma t) g(x) = (1 + (\gamma t)^2 - 2\gamma t) h(x)$ for all $t \in [0,1]$. As $\gamma t \in [0,1]$, $(\gamma t)^2 \le \gamma t$, so $(1 + (\gamma t)^2 - 2\gamma t) h(x) \le (1 - \gamma t) h(x)$ by nonnegativity of $h$. Applying Lemma \ref{lem:star_char} yields the result.
\end{itemize}
\item 
Let $\mathcal{X}$ be a bounded star-convex set with $C^1$ boundary and star center $x^* = \mathbf{0}$. Let $f : \R \ra \R$ be $\gamma$-quasar-convex with respect to the point 0, with $f(0) = 0$, and let $g(x)$ be an arbitrary $C^1$ nonnegative function defined on the boundary of $\mathcal{X}$. For each point $x \ne x^* \in \X$, let $P(x)$ be the (unique) intersection of the boundary of $\X$ with the ray from $x^*$ to $x$, and let $P(x^*) = x^*$. Then the function $h(x) = f(\norm{x}) \cdot g\lt P(x) \rt$ is nonnegative, $C^1$, and $\gamma$-quasar-convex on $\mathcal{X}$ with respect to $x^* = \mathbf{0}$. (Note: The rightmost function plotted in Figure \ref{fig:examples} was constructed in this manner, where $\X$ is the unit circle [so $P(x) = \ff{x}{\norm{x}}$], $f(x) = \ff{x^2}{1+x^2}$, and $g$ was a randomly generated linear combination of exponentiated high-frequency trigonometric functions.)
\begin{itemize}
\item \emph{Proof}: Nonnegativity of $f$ and $g$ implies that of $h$. The properties of $\X$ imply that $\lim\limits_{x \ra \mathbf{0}} h(x) = 0 = h(\mathbf{0})$, so the fact that $f,g \in C^1$ implies that $h$ is also $C^1$. Also, for any $x \ne x^*$ and any $t \in [0,1)$, $h(t x^* + (1-t)x) = h((1-t)x) = f(\norm{(1-t)x}) \cdot g(P((1-t)x)) =
\linebreak
f(t\cdot 0 + (1-t) \cdot \norm{x}) \cdot g(P(x)) \le (1- \gamma t) f(\norm{x}) \cdot g(P(x)) = (1-\gamma t) h(x) = \gamma t h(x^*) + (1- \gamma t) h(x)$. To obtain the preceding inequalities, we used Lemma \ref{lem:star_char} for $f$, and the fact that $P(tx) = P(x)$ for any $t \in (0, 1]$, since the ray from $x^*$ to $tx$ also passes through $x$.
Finally, it is trivially true that $h(t x^* + (1-t)x) \le \gamma t h(x^*) + (1- \gamma t) h(x)$ when $x = x^*$ or when $t = 1$, so  applying Lemma \ref{lem:star_char} yields the result.
\end{itemize}
\end{enumerate}

\section{Lower Bound Proofs}
\label{sec:lb-app}

In this section, we use $\mathbf{0}$ to denote a vector with all entries equal to 0, and $\mathbf{1}$ to denote a vector with all entries equal to 1.

\subsection{Proof of Lemma~\ref{lem:main-lb-quasar}} \label{sec:lem-lb-proof}

Before we prove Lemma~\ref{lem:main-lb-quasar}, we prove two useful results related to the properties of $q$ and $\Upsilon$. For convenience, these functions are restated below:
\funcdefs
\begin{observation}\label{obs:props-q}
$q$ is convex and $2$-smooth with minimizer $x^* = \mathbf{1}$. Also, for any $1 \le j_1 < j_2 \le T$,
$$
q(x) = \frac{1}{2} \grad q(x)^\top (x - x^{*})  \ge \max\left\{ \frac{1}{4} (x_1-1)^2, \frac{(x_{j_1} - x_{j_2})^2}{4 (j_2 - j_1)} \right\}.
$$
\end{observation}

\begin{proof}
Convexity and $2$-smoothness of $q$ follow from definitions. It is easy to see that $q$ is always nonnegative and $q(\mathbf{1}) = 0$, so $\mathbf{1}$ minimizes $q$. In fact $\mathbf{1}$ is the unique minimizer, since $q$ is strictly positive for all nonconstant vectors and all vectors with $x_1 \ne 1$.

Notice that as $q$ is a convex quadratic, $q(x) = \frac{1}{2} (x-x^{*})^\top \grad^2 q(x) (x-x^{*})$ where $\grad^2 q(x)$ is a constant matrix. Therefore $\grad q(x) = \grad^2 q(x) (x - x^{*})$. It follows that $q(x) = \frac{1}{2} \grad q(x)^\top (x - x^{*})$.

By definition $q(x) \ge  \frac{1}{4} (x_1-1)^2$. Furthermore, $\frac{1}{j_2 - j_1} \sum_{i=j_{1}}^{j_{2}} (x_i - x_{i+1})^2 \ge \left( \frac{1}{j_2 - j_1} \sum_{i=j_1}^{j_2} (x_i - x_{i+1}) \right)^2 \linebreak = \frac{(x_{j_1} - x_{j_2})^2}{(j_2 - j_1)^2}$, where the inequality uses that the expectation of the square of a random variable is greater than the square of its expectation. The result follows.
\end{proof}

\noindent Properties of $\Upsilon$ that we will use are listed below.
\begin{lem}\label{lem:props-ups}
   The function $\Upsilon$ satisfies the following.
  \begin{enumerate}
  \item \label{item:ups-grad-zero}
  $\Upsilon'(0) = \Upsilon'(1) = 0$.
  \item \label{item:ups-quasi-convex}
    For all $\theta \le 1$, $\Upsilon'(\theta) \le 0$, and
    for all $\theta \ge 1$, $\Upsilon'(\theta) \ge 0$.
  \item \label{item:ups-min-one} For all $\theta \in \R$ we have $\Upsilon(\theta) \ge
    \Upsilon(1) = 0$, and $\Upsilon(0) \le 10$.
  \item \label{item:ups-large-grad}
  $\Upsilon'(\theta) < -1$ for all $\theta \in (-\infty,-0.1] \cup [0.1,0.9]$.
  \item \label{item:ups-smooth} $\Upsilon$ is $180$-smooth.
    \item \label{item:ups-theta} For all $\theta \in \R$ we have $\Upsilon(\theta) \le \min\{30 \theta^4 - 40 \theta^3 + 10, \,\, 60 (\theta-1)^2\}$,
    and $\Upsilon(0) \ge 5$.
    \item \label{item:ups-quasar} For all $\theta \not\in (-0.1,0.1)$ we have $40 (\theta - 1) \Upsilon'(\theta) \ge \Upsilon(\theta)$.
  \end{enumerate}
\end{lem} 
\begin{proof}
Properties \ref{item:ups-grad-zero}-\ref{item:ups-large-grad} were proved in \cite[Lemma~2]{carmon2017lower2}.

\noindent \emph{Property \ref{item:ups-smooth}.} $|\Upsilon''(\theta)| = 120 \left|\frac{\theta (\theta^3 + 3 \theta - 2)}{(1 + \theta^2)^2}\right| \le 120 \cdot \ff{3}{2} = 180$ for all $\theta \in \R$. Thus, for any $\theta_1,\theta_2 \in \R$, $|\Upsilon'(\theta_1) - \Upsilon'(\theta_2)| \le \max\limits_{\theta \in [\theta_1,\theta_2]} |\Upsilon''(\theta)| \cdot |\theta_1 - \theta_2| \le 180 |\theta_1-\theta_2|$.

\noindent \emph{Property \ref{item:ups-theta}.} We have
$\Upsilon(0) = 120 \int_{0}^{1} \frac{t^2 (1-t)}{1 + t^2} \,dt \ge 120 \int_{0}^{1} \ff{t^2 (1-t)}{2} \,dt = \frac{120}{2 \cdot 12} = 5$. For all $\theta \in \R$ we have
$\Upsilon(\theta) = 120 \int_{1}^{\theta} \frac{t^2 (t-1) }{1 + t^2} \,dt \le 120 \int_{1}^{\theta} t^2 (t-1) \,dt = 120 ( (\theta^4/4 + \theta^3/3) -  (1/4 - 1/3) ) = 30 \theta^4 - 40 \theta^3 + 10$. In addition, since $\ff{t^2}{1+t^2} \le 1$ for all $t$, we have for all $\theta \in \R$ that
$\Upsilon(\theta) \le 120 \int^{\theta}_1 (t-1) \,dt = 120 (\theta-1)^2/2$.

\noindent \emph{Property \ref{item:ups-quasar}.}
If $\theta \in (\infty, -1.0] \cup [1.0, \infty)$ then $\ff{\theta^2}{1+\theta^2} \ge \ff{1}{2}$, so by property \ref{item:ups-theta} we have
\begin{flalign*}
\Upsilon(\theta) + 40 (1 - \theta) \Upsilon'(\theta) &\le
60 (\theta-1)^2 - 40 \cdot 120 \frac{\theta^2 (\theta-1)^2}{1 + \theta^2} \\
&\le 60 (\theta - 1)^2 - 40 \cdot 60 (\theta-1)^2 \\
&= -60 \cdot 39 (\theta-1)^2 \\
&\le 0. \tag*{\qedhere}
\end{flalign*}

Alternatively, if $\theta \in [-1.0,-0.1] \cup [0.1, 1.0]$ then $\ff{1}{1+\theta^2} \ge \ff{1}{2}$, so by property \ref{item:ups-theta} we have
\begin{flalign*}
\Upsilon(\theta) + 40 (1 - \theta) \Upsilon'(\theta) &\le 10 + 30 \theta^4 - 40 \theta^3  - 40 \cdot 120 \frac{\theta^2 (\theta-1)^2 }{1 + \theta^2} \\
&\le 10 \left( 1 + \theta^2 \left( 3 \theta^2 - 4 \theta - 240 (\theta-1)^2 \right) \right) \\
&= 10 \left( 1 - 237 \theta^4 + 476 \theta^3 - 240 \theta^2 \right) \\
&= 10P(\theta)~,
\end{flalign*}
where we define $P(\theta) \defeq 1 - 237 \theta^4 + 476 \theta^3 - 240 \theta^2$.
Observe that $P'(\theta) = {-12 \theta (40 - 119 \theta + 79 \theta^2)}$ has exactly three roots: at $\theta = 0, \theta = 1$ and $\theta = 40/79$. Furthermore, at $\theta = 1$, $\theta = 40/79$ and $\theta = 0.1$ we have $P(\theta) \le 0$, which implies $P(\theta) \le 0$ for $\theta \in [0.1,1]$. We conclude that $\Upsilon(\theta) + 40 (1 - \theta) \Upsilon'(\theta) \le 0$ for $\theta \in [0.1,1]$. In addition, $P(\theta)$ is negative while $P'(\theta)$ is positive for $\theta = -0.1$, which means that $P(\theta)$ and thus $\Upsilon(\theta) + 40 (1 - \theta) \Upsilon'(\theta)$ are also negative on $[-1.0, -0.1]$.
\end{proof}

In order to prove Lemma~\ref{lem:main-lb-quasar}, we first prove an ``unscaled version'' in Lemma~\ref{lem:lb-unscaled}. This is the critical and most difficult part of the proof of the result;
some intuition is provided in Section \ref{sec:lb}.
\lemUnscaledLB*

\begin{proof}
Since $\sigma \in (0, 10^{-4}]$, $\Upsilon$ is $180$-smooth, and $q$ is $2$-smooth, we deduce $\fB$ is $3$-smooth. By Observation \ref{obs:props-q} and Lemma \ref{lem:props-ups}.\ref{item:ups-min-one} we deduce $\fB(\mathbf{1}) = 0 < \fB(x)$ for all $x \ne \mathbf{1}$. Therefore, $x^{*} = \mathbf{1}$ is the unique minimizer of $\fB$.

Now, we will show $\fB$ is $\frac{1}{100 T \sqrt{\sigma}}$-quasar-convex, i.e. that $\G \fB(x)^\top(x-\mathbf{1}) \ge \ff{\fB(x)-\fB(\mathbf{1})}{100 T \sqrt{\sigma}}$ for all $x \in \R^T$. Define
\begin{flalign*}
\mathcal{A} &\defeq \{ i : x_i \in (-\infty, -0.1] \cup (0.9,\infty) \} \\
\mathcal{B} &\defeq \{ i : x_i \in (-0.1,0.1) \} \\
\mathcal{C} &\defeq \{ i : x_i \in [0.1,0.9] \}.
\end{flalign*}

\noindent First, we derive two useful inequalities. By Observation~\ref{obs:props-q} and the fact that $\Upsilon'(x_i) \le 0$ for $i \in \mathcal{B}$,
\begin{flalign}
\grad \bar{f}_{T,\sigma}(x)^\top (x - \mathbf{1}) &= \G q(x)^\top(x - \mathbf{1}) + \sigma \su{i \in \mathcal{A} \cup \mathcal{B} \cup \mathcal{C}}{} (x_i-1)\Upsilon'(x_i) \nonumber \\
&\ge 2 q(x) + \sigma \sum_{i \in \mathcal{A} \cup \mathcal{C} } (x_i - 1) \Upsilon'(x_i)~. \label{ineq-grad-ups}
\end{flalign}

By Lemma~\ref{lem:props-ups}.\ref{item:ups-quasi-convex} and \ref{lem:props-ups}.\ref{item:ups-theta} we deduce $\sum_{i \in \mathcal{B} \cup \mathcal{C}}\Upsilon(x_i) \le | \mathcal{B} \cup \mathcal{C} | \Upsilon(-0.1) \le 11 T$, so it follows that $\fB(x) \le q(x) + 11T\sigma + \sigma\sum_{i \in \mathcal{A}} \Upsilon(x_i)$, and therefore using $T \ge \sigma^{-1/2}$ and nonnegativity of $\Upsilon$ and $q$, we have
\begin{flalign}
 \frac{\fB(x) - \fB(\mathbf{1})}{100 T \sqrt{\sigma}} &= \frac{\fB(x)}{100 T \sqrt{\sigma}} \nonumber \\
 &\le \frac{11T\sigma}{100T\sqrt{\sigma}} + \frac{\sigma}{100T\sqrt{\sigma}} \sum_{i \in \mathcal{A}} \Upsilon(x_i)  + \frac{1}{100T\sqrt{\sigma}} q(x) \nonumber \\
 &\le \frac{11}{100} \sigma^{1/2} + \frac{\sigma}{100} \sum_{i \in \mathcal{A}} \Upsilon(x_i)  + \frac{1}{100} q(x) \nonumber \\
 &\le \frac{11}{100} \sigma^{1/2} + \frac{\sigma}{40} \sum_{i \in \mathcal{A}} \Upsilon(x_i)  + q(x) \label{ineq-func-ups}
\end{flalign}
We now consider three possible cases for the values of $x$.

\begin{enumerate}
\item
Consider the case that $x_1 \not\in [0.9,1.1]$. We have
\begin{flalign*}
\grad \fB(x)^\top (x - \mathbf{1}) &\ge  2 q(x)  +  \frac{\sigma}{40} \sum_{i \in \mathcal{A} \cup \mathcal{C} } \Upsilon(x_i) \\
& \ge \frac{0.1^2}{4} + q(x)  +  \frac{\sigma}{40} \sum_{i \in \mathcal{A} \cup \mathcal{C} } \Upsilon(x_i) \\
& = \frac{1}{\sqrt{10^4\sigma}} \cdot \frac{ \sqrt{\sigma}}{4} + \frac{\sigma}{40} \sum_{i \in \mathcal{A} \cup \mathcal{C} } \Upsilon(x_i) + q(x) \\
& \ge \frac{ \sqrt{\sigma}}{4} + \frac{\sigma}{40} \sum_{i \in \mathcal{A} \cup \mathcal{C} } \Upsilon(x_i) + q(x) \\
&\ge \frac{\fB(x) - \fB(\mathbf{1}) }{100 T \sqrt{\sigma} }
\end{flalign*}
where the first inequality uses \eqref{ineq-grad-ups} and Lemma~\ref{lem:props-ups}.\ref{item:ups-quasar}, the second inequality uses Observation~\ref{obs:props-q} and $x_1 \not\in [0.9,1.1]$, the penultimate inequality uses $\sigma \in (0, 10^{-6}] \subset (0,10^{-4}]$, and the final inequality uses \eqref{ineq-func-ups} and nonnegativity of $\Upsilon$.
\item
Consider the case that $\mathcal{B} = \emptyset$. By Lemma~\ref{lem:props-ups}.\ref{item:ups-quasar} and convexity of $q(x)$,
\begin{flalign*}
\grad \bar{f}_{T,\sigma}(x)^\top(x - \mathbf{1}) &= \G q(x)^\top(x-\mathbf{1}) + \sigma \su{i \in \mathcal{A} \cup \mathcal{C}}{} (x_i-1)\Upsilon'(x_i) \\
 &\ge q(x) - q(\mathbf{1}) + \f{\sigma}{40} \su{i \in \mathcal{A} \cup \mathcal{C}}{} \Upsilon(x_i) \\
 &= \f{1}{40} \lt q(x) + \sigma \su{i=1}{T} \Upsilon(x_i)\rt - \fB(\mathbf{1}) + \f{39}{40} q(x) \\
 &\ge \frac{\fB(x) - \fB(\mathbf{1})}{40} \\
 &\ge \frac{\fB(x) - \fB(\mathbf{1})}{100T\sqrt{\sigma}}.
\end{flalign*}
\item
Suppose cases 1-2 do not hold, i.e., $x_1 \in [0.9,1.1]$ and $\mathcal{B} \neq \emptyset$. Then there exist some $m \ge 1$ and $j \in \{1,\dots,T-m\}$ such that $x_{j} \ge 0.9$, $x_{j + m} \le 0.1$, and $x_i \in \mathcal{C}$ for all $i \in \{j+1, \dots, j+m-1\}$. Then,
\begin{flalign*}
\grad \bar{f}_{T,\sigma}(x)^\top (x - \mathbf{1} ) &\ge q(x) + \sigma \sum_{i \in \mathcal{A} \cup \mathcal{C}} (x_i - 1) \Upsilon'(x_i) + q(x) \\
&\ge \frac{0.8^2}{4 m} + \sigma \sum_{i \in \mathcal{C}} (x_i - 1) \Upsilon'(x_i) + \sigma \sum_{i \in \mathcal{A}} (x_i - 1) \Upsilon'(x_i) + q(x) \\
&\ge \frac{0.8^2}{4 m} + 0.1 \sigma (m - 2) + \frac{\sigma}{40} \sum_{i \in \mathcal{A}} \Upsilon(x_i) + q(x)  \\
&\ge \f{0.16}{\sqrt{1.6}} \sigma^{1/2} + \frac{\sigma }{40} \sum_{i \in \mathcal{A}} \Upsilon(x_i) + q(x)  \\
&\ge \frac{\fB(x) - \fB(\mathbf{1})}{100T \sqrt{\sigma}}
\end{flalign*}
where the the first inequality holds by \eqref{ineq-grad-ups}, the second inequality uses Observation~\ref{obs:props-q}, the third inequality uses Lemma~\ref{lem:props-ups}.\ref{item:ups-large-grad} and \ref{lem:props-ups}.\ref{item:ups-quasar}, the fourth inequality uses that $m = \sqrt{1.6} \sigma^{-0.5} \ge 2$ minimizes the previous expression, and the final inequality uses \eqref{ineq-func-ups} [and the fact that $0.16 / \sqrt{1.6} > 0.11$].
\end{enumerate}

Finally, suppose $x_t = 0$ for all $t = \ceil{T/2}, \dots, T$. Then we have $\fB(x) - \fB(\mathbf{1}) = \fB(x) \ge \sigma \ceil{T / 2} \Upsilon(0) \ge 2 T \sigma$, where the first inequality uses that $\Upsilon \ge 0$ and $q \ge 0$, and the last inequality uses that $T \ge 1$ and $\Upsilon(0) \ge 5$.
\end{proof}

With Lemma~\ref{lem:lb-unscaled} in hand, we are able to establish Lemma~\ref{lem:main-lb-quasar} which is a scaled version of Lemma~\ref{lem:lb-unscaled}.

\lemMainLBquasar*

\begin{proof}
We have $\sigma^{-1/2} = 10^{2} T \1 \le T$ and $\sigma = \frac{1}{10^4 T^2 \1^2} \le \frac{1}{(L^{1/2} R \epsilon^{-1/2})^2} \le 10^{-6}$, so $\bar{f}_{T,\sigma}$ satisfies the conditions of Lemma~\ref{lem:lb-unscaled}.

Let us verify the properties of $\hat{f}$.
The optimum of $\fB$ is $\mathbf{1}$, but after this rescaling it becomes $x^{*} = \frac{R}{\sqrt{T}} \mathbf{1}$, for which $\norm{x^{*}} = R$.
For all $x,y \in \R^T$, by $3$-smoothness of $\fB$ we have
\begin{flalign*}
\norm{ \grad \hat{f}(x) - \grad \hat{f}(y) } &=
\ff{1}{3} (L R^2 T^{-1}) \cdot (T^{1/2} R^{-1}) \norm{ \grad \fB(x T^{1/2} R^{-1}) - \grad \fB(y T^{1/2} R^{-1}) } \\
&\le (L R^2 T^{-1}) \cdot (T^{1/2} R^{-1})^2 \norm{ x-y } \\
&= L\norm{ x-y }~.
\end{flalign*}
Therefore $\hat{f}$ is $L$-smooth. By the definition of $\sigma$ we have $\frac{1}{100T\sqrt{\sigma}} = \1$, so $\fB$ is $\1$-quasar-convex. As quasar-convexity is invariant to scaling (Observation \ref{obs:scaling}), we deduce that $\hat{f}$ is $\1$-quasar-convex as well.
Finally, given $x^{(k)}_t = 0$ for $t = \ceil{T/2}, \dots, T$, we have
$$
\hat{f}(x^{(k)}) - \inf_z \hat{f}(z) \ge 2 T \sigma \cdot \frac{L R^2}{3T} = \ff{2}{3} L R^2 \sigma = \ff{2}{3} (10^{-2} \1^{-1} L^{1/2} R T^{-1})^2 \ge \ff{50}{3}\epsilon,
$$
where the first transition uses Lemma~\ref{lem:lb-unscaled}, the third transition uses that $\sigma = \frac{1}{10^4 T^2 \1^2}$, and the last transition uses that $T = \ceil{10^{-3} \1^{-1} L^{1/2} R \epsilon^{-1/2}} \le 2 \cdot 10^{-3} \1^{-1} L^{1/2} R\epsilon^{-1/2}$ since \mbox{$10^{-3}\1^{-1} (L^{1/2} R \epsilon^{-1/2}) \ge 1$.}
\end{proof}

\subsection{Proof of Theorem \ref{thm:main-lb-quasar}}
\label{sec:coro-lb-proof}
Before proving Theorem \ref{thm:main-lb-quasar} we recap definitions that were originally provided in \cite{carmon2017lower}.

\begin{defn}
A function $f$ is a first-order zero-chain if for every $x \in \R^n$,
$$
x_i = 0 \quad \forall i \ge t \quad \Rightarrow \quad \grad_i f(x) = 0 \quad \forall i > t.
$$
\end{defn}

\begin{defn}
An algorithm is a first-order zero-respecting algorithm (FOZRA) if, for all $i \in \{1, \dots, n \}$, its iterates $x\ind{0}, x\ind{1}, ... \in \R^{n}$ satisfy
$$
\grad_i f(x\ind{k}) = 0 \quad \forall k \le t \quad \Rightarrow \quad x\ind{t+1}_i = 0 
$$
for all $i \in \{1, \dots, n \}$.
\end{defn}

\begin{defn}
An algorithm $\mathcal{A}$ is a first-order deterministic algorithm (FODA) if there exists a sequence of functions $\mathcal{A}_k$ such the algorithm's iterates satisfy
$$
x\ind{k+1} = \mathcal{A}_k(x\ind{0}, \dots, x\ind{k}, \grad f(x\ind{0}),\dots, \grad f(x\ind{k}))
$$
for all $k \in \mathbb{N}$, input functions $f$, and starting points $x\ind{0}$.
\end{defn}

\begin{observation}\label{obs-zero-respecting}
Consider $\epsilon > 0$, a function class $\mathcal{F}$, and $K \in \mathbb{N}$. If $f : \R^{n} \rightarrow \R$ satisfies
\begin{enumerate}
\item $f$ is a first-order zero-chain,
\item $f$ belongs to the function class $\mathcal{F}$, i.e. $f \in \mathcal{F}$, and
\item $f(x) - \inf_{z} f(z) \ge \epsilon$ for every $x$ such that $x_t = 0$ for all $t \in \{K, K+1, \dots, n\}$;
\end{enumerate}
 then it takes at least $K$ iterations for any  FOZRA to find an $\epsilon$-optimal solution of $f$.
\end{observation}

\begin{proof}
Cosmetic modification of the proof of Observation~2 in \citep{carmon2017lower}.
\end{proof}

\coroMainLBquasar*

\begin{proof}
Applying Lemma~\ref{lem:main-lb-quasar} and Observation~\ref{obs-zero-respecting} implies this result for any first-order zero-respecting method. Applying Proposition~1 from \citep{carmon2017lower}, which states that lower bounds for first-order zero-respecting methods also apply to deterministic first-order methods, gives the result.
\end{proof}

\subsection{Lower Bounds via Reduction}

\begin{rem}\label{rem:quasar-approx}
If we have an algorithm that can approximately minimize a strongly quasar-convex function, we can use it to approximately minimize a quasar-convex function.
\end{rem}
\begin{proof}
This follows from the fact that if $f$ is $\1$-quasar-convex with respect to a minimizer $\xStar$, then the function $g_{\ep}(x) = f(x) + \ff{\epsilon}{2} \norm{x - x\ind{0}}^2$ is $(\1,\epsilon)$-strongly quasar-convex with respect to $\xStar$ (recall this terminology from Remark~\ref{rem:star_char_gen}). Note that $\xStar$ is not necessarily a minimizer of $g_\ep$, but
$g_\ep(x^{*}) \le f(x^*) + \epsilon R^2/2$, where $R = \norm{x\ind{0} - x^{*}}$. Therefore, if we obtain a point $\tilde{x}$ with
$g_\ep(\tilde{x}) \le \inf_x g(x) + \epsilon R^2/2$, then $f(\tilde{x}) \le g_\ep(\tilde{x}) \le g_\ep(x^*) + \ep / 2 \le f(x^*) + \epsilon R^2$.
\end{proof}

\begin{rem}\label{rem:sc-lower-bound}
Given any deterministic first-order method, there exists an $L$-smooth, $(\1,\2)$-strongly quasar-convex function such that the method requires at least
$\Omega(\max\{\1^{-1} L^{1/2} \mu^{-1/2}, \linebreak \1^{-1} L^{1/2} \mu^{-1/2} \log^+(\ep^{-1}) \})$
gradient evaluations to find an $\epsilon$-optimal point of $f$.
\end{rem}
\begin{proof}
Suppose there was a deterministic first-order method for minimizing $L$-smooth $(\1,\2)$-strongly quasar-convex functions which required $o(\1^{-1}\k^{-1/2})$ gradient evaluations to find an $\ep$-minimizer, where $\k = \ff{L}{\mu}$.
Let $f$ be an $L$-smooth function that is $\1$-quasar-convex with respect to a minimizer $x^*$, let $\ep > 0$, and let $R = \norm{x\ind{0}-x^*}$. Then, the function $g_{\ep/R^2}$ is $(L+\ff{\ep}{R^2})$-smooth and $(\1,\ff{\ep}{R^2})$-strongly quasar-convex with respect to $x^*$ as shown in Remark \ref{rem:quasar-approx}, so the condition number of $g_{\ep/R^2}$ is $\k = 1 + \ff{LR^2}{\ep}$. Thus, we could apply the method to find an $\ff{\ep}{2R^2}$-minimizer of $g_{\ep/R^2}$, and it would do so using $o(\1^{-1} \ceil{L^{1/2} R \ep^{-1/2}})$ gradient evaluations. But an $\ff{\ep}{2R^2}$-minimizer of $g_{\ep/R^2}$ is an $\ep$-minimizer of $f$, as argued in Remark \ref{rem:quasar-approx}; thus, this violates the lower bound on the complexity of minimizing quasar-convex functions shown in Theorem \ref{thm:main-lb-quasar}.

To prove the second part of the lower bound, we first note that any $(\1, \2)$-quasar-convex quadratic is also $(1, (\ff{2}{\1}-1)^{-1}\2)$-quasar-convex and thus $(1, \ff{\1\2}{2})$-quasar-convex, and in fact $\ff{\1\2}{2}$-strongly \emph{convex}; this follows from definitions.
Thus, direct application of the $\Omega((L/\mu)^{1/2} \log^+(\ep^{-1}))$ lower bound on the complexity of finding an $\ep$-minimizer of an $L$-smooth $\mu$-strongly convex quadratic with a deterministic first-order method \cite[Chapter~7]{nemirovsky1979} yields a lower bound of $\Omega (\1^{-1/2}(L/\mu)^{1/2} \log^+(\ep^{-1}))$ on the complexity of first-order minimization of $L$-smooth $(\1,\2)$-quasar-convex functions.
\end{proof}

\end{document}